\documentclass[letterpaper, 11pt]{article}

\usepackage[margin=1in]{geometry}
\usepackage[bookmarks, colorlinks=true, plainpages = false, colorlinks=true,
            citecolor=red,
            linkcolor=blue,
            anchorcolor=red,
            urlcolor=blue]{hyperref}
\usepackage{url}
\usepackage{amssymb, amsmath, mathtools, amsfonts, amsthm,  dsfont} 
\usepackage{float, xcolor, xspace, graphicx, subfigure, appendix, enumitem, cleveref, booktabs, multirow} 
\usepackage{algorithm, algpseudocode} 
\usepackage[normalem]{ulem}

\makeatletter
\algrenewcommand\alglinenumber[1]{\footnotesize #1. }  
\algrenewcommand\algorithmiccomment[1]{\hfill\textcolor{blue}{\(\triangleright\) #1}} 
\makeatother

\usepackage{prettyref}

\newrefformat{eq}{(\ref{#1})}
\newrefformat{sec}{Section~\ref{#1}}
\newrefformat{alg}{Algorithm~\ref{#1}}
\newrefformat{fig}{Figure~\ref{#1}}
\newrefformat{tab}{Table~\ref{#1}}
\newrefformat{rmk}{Remark~\ref{#1}}
\newrefformat{clm}{Claim~\ref{#1}}
\newrefformat{claim}{Claim~\ref{#1}}
\newrefformat{def}{Definition~\ref{#1}}
\newrefformat{assump}{Assumption~\ref{#1}}
\newrefformat{cor}{Corollary~\ref{#1}}
\newrefformat{lmm}{Lemma~\ref{#1}}
\newrefformat{prop}{Proposition~\ref{#1}}
\newrefformat{app}{Appendix~\ref{#1}}
\newrefformat{thm}{Theorem~\ref{#1}}

\usepackage[size=tiny]{todonotes}

\theoremstyle{plain}
\newtheorem{theorem}{Theorem}
\newtheorem{lemma}{Lemma}
\newtheorem{proposition}{Proposition}
\newtheorem{corollary}{Corollary}

\theoremstyle{definition}
\newtheorem{definition}{Definition}

\newtheorem{remark}{Remark}
\newtheorem{claim}{Claim}
\newtheorem{assumption}{Assumption}

\newif\ifhideproofs
\ifhideproofs
    \usepackage{environ}
    \NewEnviron{hide}{}

\fi
\newif\ifdraft
\drafttrue

\newcommand{\argmin}{\mathop{\arg\min}}

\newcommand{\termI}{\mathrm{(I)}}
\newcommand{\termII}{\mathrm{(II)}}

\newcommand{\stepa}[1]{\overset{\rm (a)}{#1}}
\newcommand{\stepb}[1]{\overset{\rm (b)}{#1}}
\newcommand{\stepc}[1]{\overset{\rm (c)}{#1}}

\newcommand{\reals}{{\mathbb{R}}}

\newcommand*{\medcap}{\mathbin{\scalebox{1.2}{\ensuremath{\cap}}}}
\newcommand*{\medcup}{\mathbin{\scalebox{1.2}{\ensuremath{\cup}}}}

\newcommand{\diff}{{\rm d}}

\newcommand{\toas}{\xrightarrow{{\rm a.s.}}}

\newcommand{\norm}[1]{\left\|{#1} \right\|}

\newcommand{\Prob}{\mathbb{P}}

\newcommand\independent{\protect\mathpalette{\protect\independenT}{\perp}}
\def\independenT#1#2{\mathrel{\rlap{$#1#2$}\mkern2mu{#1#2}}}

\newcommand{\Binom}{{\rm Binom}}

\newcommand{\Exp}{{\rm Exp}}

\newcommand{\Uniform}{\mathrm{Uniform}}

\newcommand{\ER}{ Erd\H{o}s--R\'{e}nyi \xspace }

\newcommand{\E}{\mathbb{E}}
\newcommand{\Expect}{\mathbb{E}}

\makeatletter

\@tfor\@tempa:=ABCDEFGHIJKLMNOPQRSTUVWXYZ\do{
  \expandafter\edef\csname t\@tempa\endcsname{
    {\noexpand\widetilde{\@tempa}}
  }
}

\@tfor\@tempa:=ABCDEFGHIJKLMNOPQRSTUVWXYZ\do{
  \expandafter\edef\csname cal\@tempa\endcsname{
    \noexpand\mathcal{\@tempa}
  }
}

\@tfor\@tempa:=abcdefghijklmnopqrstuvwxyz\do{
  \expandafter\edef\csname sf\@tempa\endcsname{
    \noexpand\mathsf{\@tempa}
  }
}
\makeatother

\newcommand{\OT}{\mathsf{OT}}        
\newcommand{\Gap}{\mathsf{Gap}}      
\newcommand{\InnerGap}{\mathsf{InnerGap}}      
\newcommand{\NoSupply}{\mathsf{NoSupply}}      
\newcommand{\HasRider}{\mathsf{HasRider}}      

\newcommand{\drivervec}{\mathsf{S}} 
\newcommand{\ridervec}{\mathsf{D}} 
\newcommand{\servicevec}{\mathsf{R}} 
\newcommand{\sfX}{\mathsf{X}} 
\newcommand{\sfT}{\mathsf{T}} 

\newcommand{\driverset}{\mathcal{S}}
\newcommand{\riderset}{\mathcal{D}}

\newcommand{\driver}{s}
\newcommand{\rider}{d}
\newcommand{\service}{r}

\renewcommand{\u}{u}
\newcommand{\vf}{v^{\mathrm{F}}}
\newcommand{\vnf}{v^{\mathrm{NF}}}

\newcommand{\potential}{\boldsymbol{\psi}}
\newcommand{\potentialF}{\psi^{\mathrm{F}}}
\newcommand{\potentialNF}{\psi^{\mathrm{NF}}}

\newcommand{\base}{\service} 
\newcommand{\extra}{b} 
\newcommand{\shift}{\tau}
\newcommand{\newrange}{\ell}

\newcommand{\smoothness}{\eta}
\newcommand{\rs}{s}

\newcommand{\betaPar}{\xi}  
\newcommand{\MPar}{\gamma}  
\newcommand{\deltaM}{\delta_m}  
\newcommand{\muM}{\mu_m}  
\newcommand{\rhoS}[1]{\rho_m^{#1}} 
\newcommand{\alphaMB}{\beta_{\betaPar, \MPar, \smoothness}} 
\newcommand{\CMB}{\alpha_{\betaPar, \MPar, \smoothness}} 
\newcommand{\dimension}{k}

\newcommand{\tcell}{\mathrm{cell}_{\mathrm{T}}}
\newcommand{\highdimalpha}{\alpha_{\dimension,\betaPar,\MPar}}
\newcommand{\highdimN}{N_{\dimension,\betaPar,\MPar}}
\newcommand{\highdimbeta}{\beta_{\dimension,\MPar} \, \varepsilon^{-\dimension}}
\newcommand{\patterncells}{\mathcal{A}}
\newcommand{\activepatterncells}{\patterncells^*\left[\riderset_+\right]}
\newcommand{\patternside}{w'_{\dimension, \MPar}}
\newcommand{\epsiloncondition}{w_{\dimension, \MPar}}
\newcommand{\alphagainpattern}{\alpha^{\ref{lmm:gainonepattern}}_{\MPar ,\dimension}}
\newcommand{\volsphere}{\kappa_\dimension}
\newcommand{\alphaproba}{\alpha^{\ref{lmm:lowerboundpatterna}}_{\dimension,\betaPar,\MPar }}

\begin{document}

\title{A uniformity principle for spatial matching}
\author{
    Taha Ameen\thanks{
        T. Ameen is with the Department of Electrical and Computer Engineering and the Coordinated Science Lab, University of Illinois, Urbana IL, USA, \texttt{tahaa3@illinois.edu}. T. Ameen is supported by NSF Grant CCF 19-00636.
    },~
    Flore Sentenac\thanks{
        F. Sentenac is with the Department of Information Systems and Operations Management, HEC Paris Business School, France, \texttt{sentenac@hec.fr}.
    } ~and 
    Sophie H.\ Yu\thanks{
        S.\ H.\ Yu is with the Operations, Information and Decisions Department,   the Wharton School of Business, University of Pennsylvania, Philadelphia PA, USA,  \texttt{hysophie@wharton.upenn.edu}.
    }
}
\date{}

\maketitle

\begin{abstract} 
    Platforms matching spatially distributed supply to demand face a fundamental design choice: given a fixed total budget of \emph{service range}, how should it be allocated across supply nodes ex ante, i.e. before supply and demand locations are realized, to maximize fulfilled demand? We model this problem using bipartite random geometric graphs where $n$ supply and $m$ demand nodes are uniformly distributed on $[0,1]^\dimension$ ($\dimension \ge 1$), and edges form when demand falls within a supply node's service region, the volume of which is determined by its service range. Since each supply node serves at most one demand, platform performance is determined by the expected size of a maximum matching. We establish a \emph{uniformity principle}: whenever one service range allocation is more uniform than the other, the more uniform allocation yields a larger expected matching. This principle emerges from diminishing marginal returns to range expanding service range, and limited interference between supply nodes due to bounded ranges naturally fragmenting the graph. For $\dimension=1$, we further characterize the expected matching size through a Markov chain embedding and derive closed-form expressions for special cases. Our results provide theoretical guidance for service-range allocation and incentive design in ride-hailing, on-demand labor markets, and drone delivery platforms, highlighting the benefits of reducing disparities in supply-side flexibility.
\end{abstract}

\tableofcontents
\maketitle


\newpage 


\section{Introduction} \label{sec-introduction}
Matching supply to demand underlies many service systems, from ride-hailing ~\cite{ozkan2020dynamic, bimpikis2019spatial, eom2025batching, wang2024demand,inbook} and on-demand labor platforms \cite{leung2018learning,kanoria2021facilitating,ashlagi2014stability} to emergency response~\cite{revelle1995integrated} and edge-computing networks~\cite{nygren2010akamai}. A common structural feature in these settings is \emph{locality}, i.e. supply can typically serve only demand that is ``close enough.''
In ride-hailing, this locality is literal -- a driver must be geographically close to a rider in order to provide service. In feature-based marketplaces such as TaskRabbit, Fiverr or Upwork, locality is induced by attributes instead of geography --  workers and jobs are compatible if they they share relevant characteristics such as skill type, price range, availability windows and so on. 

Motivated by locality, a substantial body of literature has studied \emph{spatial matching}, which embeds agents in a metric space and defines compatibilities or costs via distance. This framework has been applied to classical facility-location and coverage models~\cite{church1974maximal, revelle1995integrated}, dynamic ride-hailing and delivery systems with explicit spatial structure~\cite{ozkan2020dynamic, bimpikis2019spatial, eom2025batching, wang2024demand, inbook, chen2024courier}, and more general works on matching in metric or feature spaces~\cite{bansal2014randomized,gupta2019stochastic,Kanoria2025,Chen2025}, where distance may encode geography or attribute dissimilarity.

A central design lever is each supply node’s \emph{service range}, which determines the size of its service region and hence the demand nodes it can potentially match with. Operationally, the service range is a decision variable constrained by resources: expanding a supply node's service range implies higher costs, such as longer pickup times in ride-hailing~\cite{ozkan2020dynamic}, higher energy consumption for EVs or drones~\cite{murray2015flying}, and financial incentives for agents to accept less convenient tasks~\cite{campbell2006incentive}. Thus, a platform cannot simply maximize the service range of all agents simultaneously, and typically operate within an aggregate budget constraint on the total service range. 

In this paper, we consider systems where a unit of supply has capacity to serve at most one unit of demand: once demand-supply compatibilities are established, the platform decides how to optimally pair units of demand and supply. Therefore, the total service range represents the total capacity of the supply side to reach demand, while individual service ranges govern the local density of demand-supply compatibilities and hence the overall matching capacity of the system;  Figure~\ref{fig: intro-het-vs-hom} illustrates this. 
We ask a simple structural question: \textit{given a fixed total budget of service range, how should it be allocated across supply nodes ex ante, i.e. before the locations of supply and demand are realized, to maximize the fulfilled demand?}

To answer this question, we model the system as a bipartite random geometric graph~\cite{penrose2003rgg}. 
There are $n$ supply nodes and $m$ demand nodes embedded in a feature space, taken to be the unit hypercube with dimension $\dimension \geq 1$. The features of demand and supply nodes are drawn independently from the uniform distribution. A supply node is compatible to serve a demand node if the demand node lies within a ball centered at the supply node; this ball is the supply node’s service region with volume determined by its service range.
This construction induces a compatibility graph that is a bipartite geometric graph, with edges connecting each supply node to the demand nodes that lie within its service region.     
We assume each supply node can serve at most one demand node. We therefore measure the effectiveness of a service range allocation by the expected maximum number of demand nodes that can be served, equivalently the expected size of a maximum matching in the induced compatibility graph. 

Our main result establishes a \emph{uniformity principle}: for a fixed budget on the total service range, making the vector of service ranges more uniform (in the sense of majorization~\cite{marshall2011majorization}) increases the expected maximum matching size.\footnote{Our results assume a fixed total service range (linear cost), but extend to general settings where the budget constrains the total cost, provided the cost function for each supply node is convex in its service range (i.e., increasing marginal costs). In such settings, a uniform allocation maximizes the total achievable service range under a fixed budget, and hence amplifying the uniformity principle.} Equivalently, the expected matching is Schur-concave in the service range: reallocating a small amount of range from a high-range node to a low-range node improves performance, even though concentrating range among a few super-nodes may intuitively seem attractive for servicing isolated demand nodes that are far apart. This principle emerges from diminishing marginal returns to expanding the service range and limited interference from neighboring nodes when the compatibility graph is relatively sparse.

We show that the uniformity principle holds under uniform supply and demand distributions in all dimensions $\dimension \geq 1$. We further extend it to settings where supply and demand locations follow (possibly different) smooth, Lipschitz-continuous distributions, and we provide counterexamples showing that the principle can fail for general distributions. 
This principle augments existing literature on the idea that ``more uniform'' allocations can improve system performance, a phenomenon seen in stochastic allocation and resource-sharing problems in queueing~\cite{liyanage1993allocation,moyal2008convex,feng2018arrangement}. 

\begin{figure}
    \centering
    \subfigure[Non-uniform allocation]{
    \includegraphics[width=0.45\linewidth]{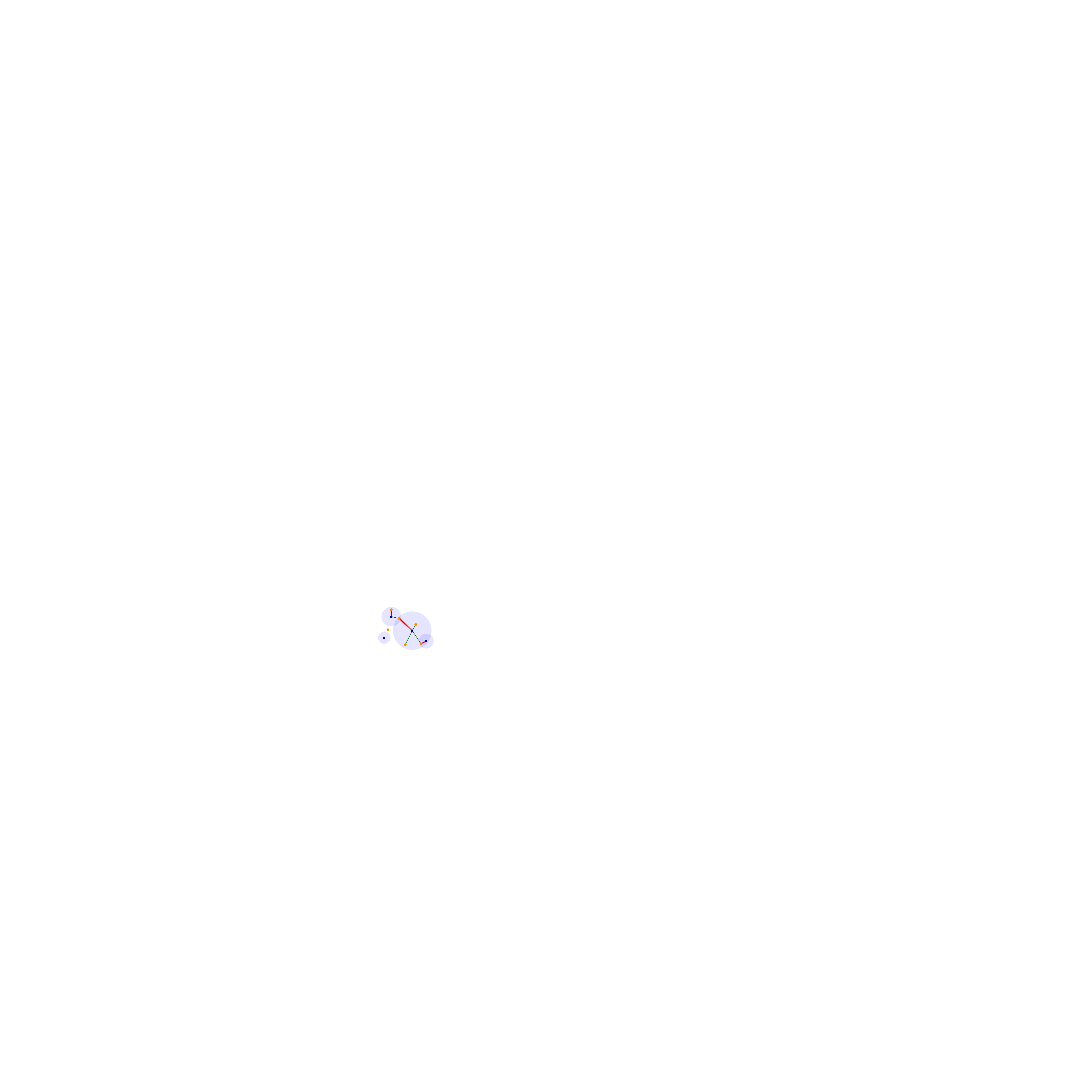}
    \label{fig: intro-het}
    }
    \hfill
    \subfigure[Uniform allocation]{
    \includegraphics[width=0.45\linewidth]{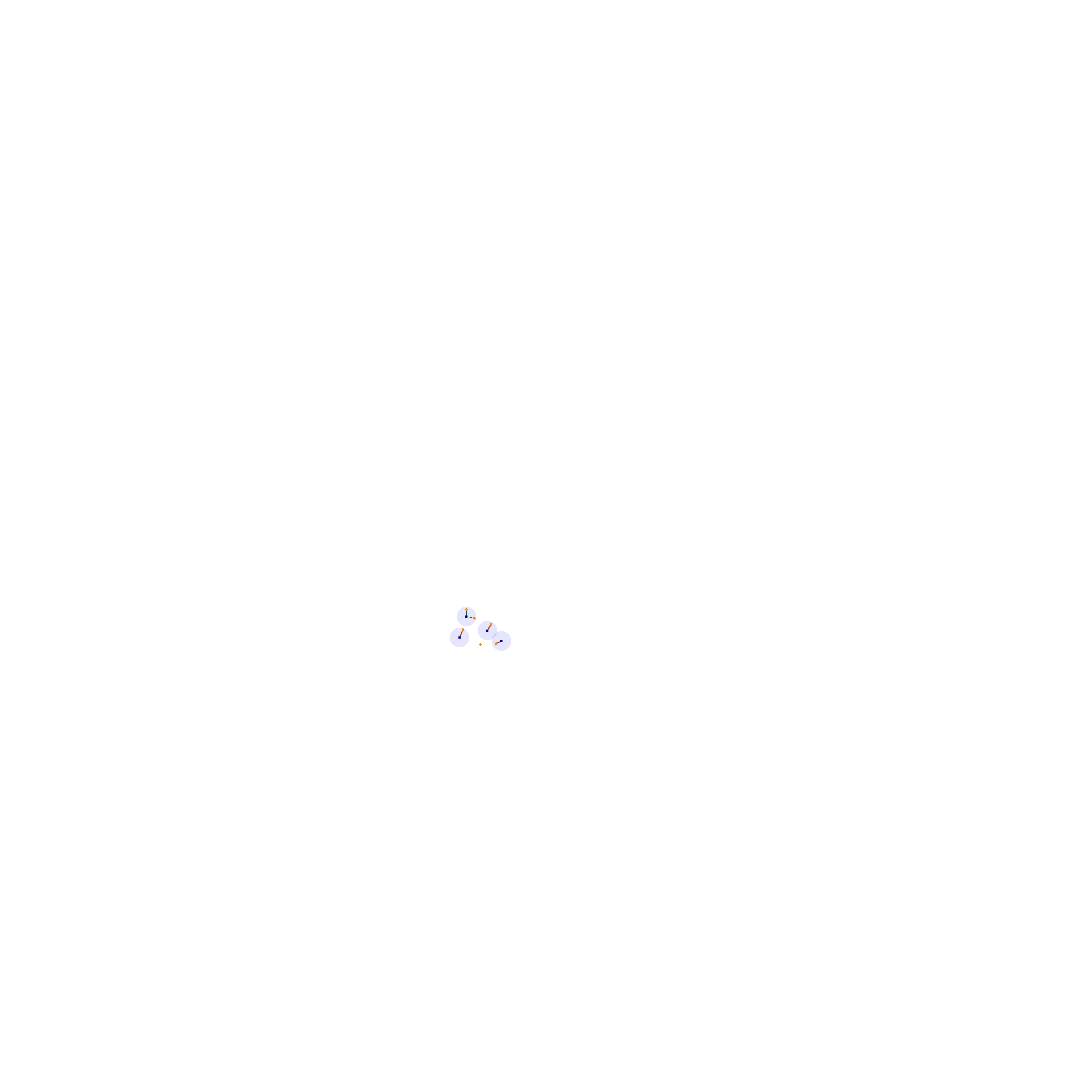}
    \label{fig: intro-hom}
    }
    \caption{Example of a non-uniform (a) and a uniform (b) allocation of service range. Orange circles denote demand nodes and blue squares denote supply nodes; shaded regions indicate supply service regions. Green segments show edges in the compatibility graph, and red segments highlight a maximum matching. The maximum matching size is $3$ in (a) and $4$ in (b). }
    \label{fig: intro-het-vs-hom}
\end{figure}

Our uniformity principle also provides insight into optimal assignment of \emph{flexibility} in matching systems. Flexibility is a longstanding theme in operations, classically engineered ex ante through centralized design choices such as chaining manufacturing plants~\cite{jordan1995principles}, cross-training workers~\cite{wallace2004resource},
or pooling resources~\cite{bassamboo2010flexibility}. In modern on-demand platforms, by contrast, flexibility is often decentralized and driven by contracts and incentives,  for example, by allowing supply units to self-schedule their availability or by combining different worker types, as well as by inducing riders to accept less convenient pickup options on the demand side~\cite{gurvich2019operations,dong2020managing,lobel2024frontiers,yan2025trading}. For example, Uber's Quest program offers bonuses to drivers who complete additional trips during a shift: by encouraging drivers to accept ``one more ride,'' it effectively expands the set of requests they are willing to consider, without changing the underlying geography. Similarly, bonus challenges can induce taskers on TaskRabbit to accept extra jobs, broadening the set of requests they are willing to serve and effectively expanding their service range (Figure~\ref{fig: taskrabbit}).


To study the effect of flexibility, we also analyze a \emph{dual service-range model} in which each supply node is either inflexible or flexible, so that the service range takes one of two values. Related dual service-range models have been considered previously~\cite{freundmartinzhao2024twosided}. In particular, they explicitly compare one-sided flexibility (only supply nodes are flexible) with two-sided flexibility (both demand and supply nodes are flexible), while their primary analysis focuses on \ER random graphs without spatial structure. In contrast, we consider a spatial setting and provide a theoretical analysis of how the allocation of flexibility within the supply side affects the expected maximum number of served demand nodes. Our results yield a simple design principle: it is preferable to incentivize supply nodes so as to reduce disparities in flexibility across nodes.

Beyond this uniformity principle, we derive a characterization of the expected maximum matching in this model and quantify the gains from flexibility. In particular, we describe explicitly how the benefit scales with the service range of flexible and inflexible nodes, and the fraction of flexible nodes.

\begin{figure}[t]
    \centering
    \includegraphics[width=0.99\linewidth]{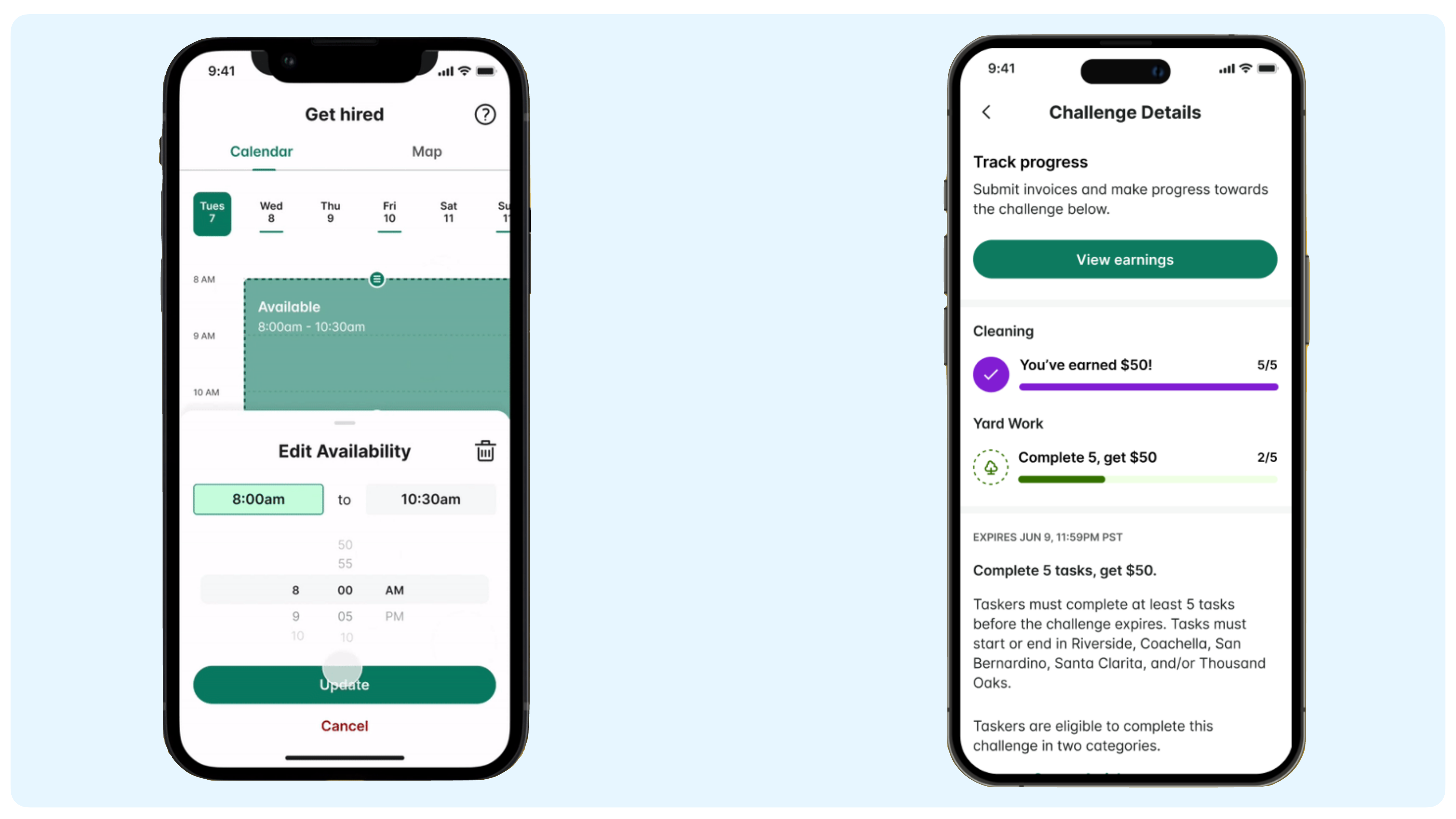}
    \caption{(Left) agents can choose their own level of flexibility. (Right) platform can provide incentives to influence service range. Images borrowed from TaskRabbit's blog post~\cite{blog_taskrabbit}.
    }
    \label{fig: taskrabbit}
\end{figure}

\subsection{Related literature}
 
\paragraph{Spatial matching.}
Literature on spatial matching spans classical models of facility location and coverage optimization \cite{church1974maximal,revelle1995integrated}, as well as more recent analyses of dynamic ride-hailing and delivery systems with explicit spatial structure \cite{ozkan2020dynamic,bimpikis2019spatial,eom2025batching}. This geometric perspective has increasingly been extended to dynamic matching in metric or feature spaces \cite{balkanski2023power,Kanoria2025,Chen2025,yang2025online,sentenac2025online}, where distance may represent either geography or attribute dissimilarity.
However, these dynamic spatial and metric matching papers typically take the service range as exogenous, and focus on online policies whose performance is evaluated against an offline benchmark, rather than on how the compatibility graph itself should be designed.
In contrast, we consider an offline problem: given a fixed snapshot of supply and demand (e.g., accumulated over a batching window), how does the ex-ante allocation of service range across supply nodes affect the maximum matching size?
In this sense, our results can be viewed as analyzing the performance of a batching policy within a fixed time window: by relating the expected matching size to the service range allocation, our results provide insight into optimizing the design of the compatibility graph itself.

\paragraph{Matchings on random graphs.}
    For Erd\H{o}s-R\'{e}nyi (ER) graphs, the seminal work of Karp and Sipser~\cite{KarpSipser} characterized the fraction of matched nodes in sparse regimes via a greedy algorithm; follow-up work on locally tree-like limits~\cite{bordenave2012matchingsinfinitegraphs} further describes asymptotic matching sizes via local weak convergence. In contrast, random geometric graphs (RGGs) are not locally tree-like and contain many short cycles, so techniques based on tree limits do not transfer. Thus, the maximum matching size in RGGs is less understood: \cite{sentenac2025online} obtain a formula for uniform RGGs with $n=m$ in one dimension; \cite{holroyd2008poissonmatching} study distance tails in Poisson perfect matchings; and \cite{gao2025randomized} give lower bounds on the size of maximum matching for Gaussian distributions. {We establish the uniformity principle for all dimensions $\dimension \geq 1$, proving that homogeneous service ranges provide an upper bound on the expected maximum matching size for any heterogeneous allocation with the same total budget.} For $\dimension =1$, we also develop a framework to compute exact expressions for heterogeneous service ranges via Markov chain embeddings.


\paragraph{Optimal transport.}
When service ranges are unconstrained on $[0,1]^{\dimension}$, the bipartite matching problem reduces to Euclidean optimal transport (OT) between two i.i.d.\ uniform samples. In particular, if the transportation cost is the sum of Euclidean distances raised to the power $\dimension$,
\(
\min_{\sigma\in S_n}\sum_{i=1}^n \|X_i-Y_{\sigma(i)}\|^{\dimension}
\;=\; n\,\OT_{n,\dimension},
\)
then classical results imply the \emph{expected optimal total cost} scales as
$\mathbb{E}\!\left[n\,\OT_{n,\dimension}\right]=\Theta(\sqrt{n})$ for $\dimension=1$,
$\mathbb{E}\!\left[n\,\OT_{n,\dimension}\right]=\Theta(\log n)$ for $\dimension=2$,
and $\mathbb{E}\!\left[n\,\OT_{n,\dimension}\right]=\Theta(1)$ for $\dimension\ge 3$~\cite{ajtai1984optimal,talagrand1994transportation}.
In contrast, our model forces locality through the service range, yielding a total ``cost" of $\Theta(1)$ for all $\dimension\ge 1$.
Rather than fulfilling all possible demand with unbounded transportation distances, we characterize how to optimally allocate a fixed budget to maximize expected matches, showing that uniform allocation dominates and that matching sizes increase with the total budget. 
This bounded-range regime captures a practical effect, where platforms let some demand go unfulfilled (or rolled on to the next batch), rather than serving them at a prohibitive cost.

\paragraph{Flexibility and resource allocation.} The introduction mentions foundational works on flexibility in operations management, either centralized \cite{jordan1995principles,wallace2004resource,bassamboo2010flexibility} or decentralized \cite{gurvich2019operations,dong2020managing,lobel2024frontiers,yan2025trading}. Among these works, the most closely related to ours \cite{freundmartinzhao2024twosided} studies the design of two-sided flexibility incentives on a matching platform, modeled as a sparse bipartite random graph. Nodes on each side are either regular or flexible, and edges appear independently with probabilities that depend on the types of their endpoints, so that pairs involving flexible nodes are more likely to be connected. Given a fixed flexibility budget, formalized as a constraint on the total fraction of flexible nodes, they ask whether the platform should invest flexibility only on one side or split it across both sides in order to maximize the expected size of a maximum matching, and characterize parameter regimes in which each allocation is optimal.

Their framework is non-spatial: compatibilities are modeled through independent edge probabilities and do not encode geometric or feature-based locality. Our work is complementary in that we study a spatially embedded matching environment in which supply and demand are located in a metric space and edges arise when demand points fall within the service ranges of nearby supply units, yielding a bipartite random geometric graph. This induces strong local correlations: nearby demand nodes share many candidate suppliers, short cycles are abundant, and the graph is far from locally tree-like, which makes the analysis more challenging.

\paragraph{Uniformity principle.}
The idea that ``more uniform'' allocations can improve system performance is present in the theory of majorization and Schur-concavity \cite{marshall2011majorization}, and has been applied to stochastic allocation and resource-sharing problems. In particular, \cite{liyanage1993allocation} show that when system performance is (stochastically) Schur-convex in a vector of allocated resources, balanced allocations are optimal, and subsequent work on arrangement-increasing and Schur-convex performance measures extends this perspective to a broad class of reliability and resource-allocation settings \cite{feng2018arrangement}. 
Similar themes arise in queueing systems: in real-time queues with deadlines, Schur-convexity comparisons can be used to show that service disciplines closer to the Earliest-Deadline-First policy improve lateness performance measures \cite{moyal2008convex}.
In flexible-service systems, well-spread flexibility (e.g., expander-like sparse connectivity between queues and servers) promotes effective resource pooling and prevents bottlenecks, leading to strong throughput and delay guarantees \cite{tsitsiklis2013queueing,tsitsiklis2017flexible}.
Related uniformity phenomena also appear in supply-chain capacity allocation, where uniform allocation rules can eliminate strategic order inflation and lead to well-behaved equilibria in certain regimes \cite{cachon1999allocation,cho2014capacity}. In a spatial ride-sharing context, \cite{bimpikis2019spatial} show that both platform profits and consumer surplus are maximized when the pattern of ride demand is more ``balanced'' across locations. Our uniformity principle contributes to this line of results by showing, in a random spatial matching environment with exogenous demand, that the expected maximum matching size is Schur-concave in the vector of service ranges, so that making supply-side flexibility more uniform across locations is optimal for a given total coverage budget.

\subsection{Model} \label{sec-model}

We study the spatial matching problem where supply nodes and demand nodes are distributed on the unit hypercube $[0,1]^\dimension$ of dimension $\dimension \geq 1$. Let $[n] \triangleq \{1,\ldots,n\}$ and $[m] \triangleq \{1,\ldots,m\}$ denote the index sets for supply and demand nodes respectively. Define $\betaPar \triangleq m/n$, which measures the demand-to-supply ratio or market imbalance. 
For each supply node $i \in [n]$, denote $(s_i, r_i)$ as its vector of characteristics, where $s_i \in [0,1]^\dimension$ is its location and $r_i \in \mathbb{R}_{\geq 0}$ is its service range parameter. For each demand node $j \in [m]$, denote $d_j \in [0,1]^\dimension$ as its location. Supply node $i$ can serve demand node $j$ if 
\begin{align*}
      \|s_i - d_j\|_2  \, \leq \,  \left( {r_i}/{n} \right)^{1/\dimension} \, ,
\end{align*}
Equivalently, supply node $i$ serves only demand nodes within a ball of volume $\kappa_{\dimension} \, r_i/n$ centered at $s_i$, where $\kappa_{\dimension} = \pi^{\dimension/2}/\Gamma(1+\dimension/2)$ is the volume of a unit ball in $\mathbb{R}^\dimension$, and the scaling by $n$ ensures operation in the sparse regime of random geometric graphs with $\Theta(1)$ degree per supply node on average. Let
\[
    \drivervec\triangleq \left(s_i\right)_{i=1}^n\,, \quad \ridervec \triangleq \left(d_j\right)_{j=1}^m\,, \quad  \servicevec \triangleq \left(r_i\right)_{i=1}^n
\]
denote the vectors of supply locations, demand locations, and service ranges, respectively. We also use $\mathcal{S} = \{s_i\}_{i=1}^n$ and $\mathcal{D} = \{d_j\}_{j=1}^m$ to denote the multisets of supply and demand locations. Given $\drivervec$, $\ridervec$, and $\servicevec$, let $G((\drivervec, \servicevec), \ridervec)$ denote the bipartite graph with vertex set $[n] \cup [m]$ and edge set
\[
    \mathcal{E} \triangleq \left\{(i,j)\in [n]\times[m] : \  \|s_i-d_j\|_2 \le \left( {r_i}/{n} \right)^{1/\dimension} \right\} \,.
\]
\begin{assumption}\label{assump:random-locations}
Supply and demand node locations are drawn independently and uniformly at random from the unit hypercube, i.e. 
\(
    \left\{s_i\right\}_{i=1}^n , \left\{d_j\right\}_{j=1}^m \stackrel{\text{i.i.d.}}{\sim} \Uniform([0,1]^\dimension) \,. 
\)
\end{assumption}
Given a service range vector $\servicevec \in \reals_{\geq 0}^n$, let $G \sim \mathbb{G}(m,\servicevec)$ be the random geometric bipartite graph with supply and demand locations drawn according to Assumption~\ref{assump:random-locations}. Define
\begin{align}\label{eq:objective}
    \mu_{m} (\servicevec) \triangleq  \Expect_{G\sim  \mathbb{G}(m , \servicevec) } \left[M(G) \right]
\end{align}
as the expected maximum matching size under the random model $\mathbb{G}(m, \servicevec)$, where $M(G)$ denotes the size of a maximum matching in $G$. 

\begin{remark}\label{rem:radiusvsrange}
In our model, the service range parametrizes the volume of the service region, which is proportional to the expected number of compatible demand nodes.
An alternative choice is to have the service range parametrize the \emph{radius} of the service region, so that supply node $i$ can serve demand node $j$ if $\|s_i-d_j\|_2 \le r_i/n^{1/\dimension}$, where the scaling by $n^{1/\dimension}$ ensures a $\Theta(1)$ degree per supply node on average. This interpretation is natural in applications such as ride-hailing, where distance has a direct operational meaning. Note that these two parametrizations are equivalent when $\dimension=1$. We study the radius model through simulation in Section~\ref{sec-simulations} and show that for $\dimension = 2$, it leads to qualitatively different behavior.
\end{remark}

\subsection{Main results}\label{sec:main-results}
\subsubsection{Uniformity principle}
\begin{definition}[Majorization]\label{def:maj}
For $\mathbf{x},\mathbf{y}\in\reals^n$, let $\mathbf{x}^\downarrow$ and $\mathbf{y}^{\downarrow}$ be their non-increasing rearrangements, so that $x^{\downarrow}_1 \geq x^{\downarrow}_2 \geq \cdots \geq x^{\downarrow}_n$ and $y^{\downarrow}_1 \geq y^{\downarrow}_2 \geq \cdots \geq y^{\downarrow}_n$. We say $\mathbf{x}$ \emph{majorizes} $\mathbf{y}$, and write $\mathbf{x}\succeq \mathbf{y}$, if
\[
    \sum_{i=1}^\ell x_i^\downarrow \ge \sum_{i=1}^\ell y_i^\downarrow\quad (\ell=1,\dots,n-1), \qquad \sum_{i=1}^n x_i^\downarrow=\sum_{i=1}^n y_i^\downarrow.
\]
\end{definition}

\begin{theorem}[Uniformity principle]\label{thm:uniformity}
Given any constant $\MPar \geq \max\{1,\betaPar^{-1}\}$, for any two service range vectors $\servicevec, \servicevec' \in [0, \MPar]^n$ satisfying $\servicevec \succeq \servicevec'$ and $\|\servicevec ^\downarrow - {\servicevec'}^\downarrow\|_1 > \theta n$, where 
\begin{enumerate}[label=$\mathrm{(U\arabic*)}$]
    \item\label{U:1} $(\dimension = 1)$ $\theta > C_{\betaPar,\MPar} \left(\frac{\log n}{n}\right)^{1/3}$ for some constant $C_{\betaPar,\MPar}>0$ that depends only on $\betaPar$ and $\MPar$,  
    \item\label{U:2} $(\dimension \ge 2)$ $\theta > C_{\betaPar,\MPar, \dimension} \left(\frac{1}{n}\right)^{\frac{1}{4(k+1)}}$ for some constant $C_{\betaPar,\MPar,\dimension}>0$ that depends only on $\betaPar$, $\MPar$ and $\dimension$,
\end{enumerate}
the expected matching size satisfies $ \mu_{m} (\servicevec') > \mu_{m} (\servicevec)$ for all sufficiently large $n$.
\end{theorem}

\begin{corollary}[Uniform allocation dominates]\label{cor:uniform-dominates}
Given a service range vector $\servicevec \in [0, \MPar]^n$, let $\overline{\servicevec} = (\overline{r}, \overline{r}, \ldots, \overline{r})$ denote the uniform allocation where $\overline{r} = \|\servicevec\|_1/n$. If $\|\servicevec - \overline{\servicevec}\|_1 > \theta n$ where either \ref{U:1} or \ref{U:2} hold, 
then for sufficiently large $n$, we have
\(
    \mu_{m} (\overline{\servicevec}) > \mu_{m} (\servicevec).
\)
\end{corollary}

\begin{proof}
Since $\overline{\servicevec}$ has all equal entries, it is majorized by any non-uniform vector $\servicevec$ with the same sum, 
{
    i.e.}
$\servicevec \succeq \overline{\servicevec}$. The result then follows from Theorem~\ref{thm:uniformity}.
\end{proof}

\begin{remark}[Beyond uniform distribution]\label{rmk:general}
The uniformity principle extends beyond uniform distributions. For $\dimension = 1$, our result is more general (see~Proposition~\ref{prop:schur-concavity}); it allows for distinct supply and demand distributions. Specifically,  it holds for distributions that are $\eta$-Lipschitz continuous for some constant $\eta > 0$ and have densities bounded below by a constant depending only on $\eta$. We conjecture that this extension also holds for $\dimension \geq 2$ and leave this as future work.

However, the uniformity principle does not hold for all distributions. Two examples are provided in Figure~\ref{fig:counterexample} with $\varepsilon, \ell_1, \ell_2 = \Theta(1)$: In panel~(a), given a total service range budget of $\Theta(1)$, uniform allocation yields $\service_i = \Theta(1/n)$ for all $i \in [n]$, resulting in a maximum matching of size $0$; In panel~(b), given a total service range budget less than $\ell_2$, uniform allocation yields $r_i < \ell_2$ for all $i \in [n]$, also resulting in a maximum matching of size $0$. Thus, the uniformity principle requires regularity conditions on the distributions; an open problem is to identify the minimal such conditions.
\end{remark}

\begin{figure}
    \centering
    \subfigure[Counterexample 1]{
    \includegraphics[width=0.45\linewidth]{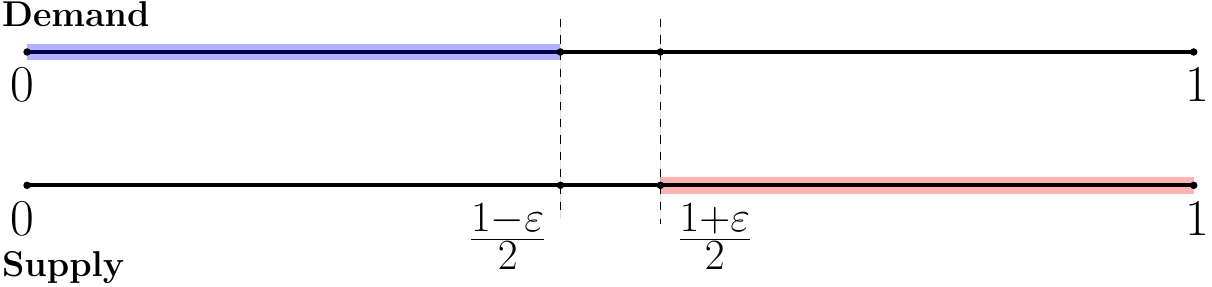}
    \label{fig: counterexample-1}
    }
    \hfill 
    \subfigure[Counterexample 2]{
    \includegraphics[width=0.45\linewidth]{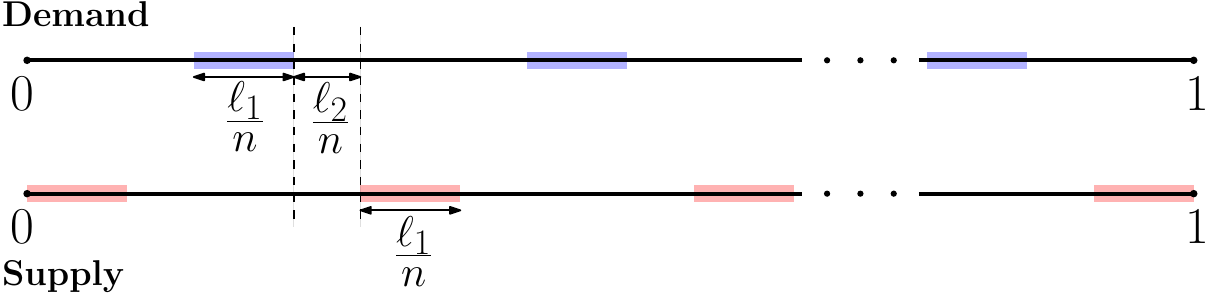}
    \label{fig: counterexample-2}
    }
    \caption{Demand and supply are uniformly distributed in their respective shaded regions.}
    \label{fig:counterexample}
\end{figure}

\subsubsection{Dual service range model}
We present a framework to compute the maximum matching size using an embedded Markov chain. We focus on a dual service range model when \( m = n \) and \(\dimension = 1\), where each supply node is either inflexible (has service range $\base$) or flexible (has service range $\base+\extra$). The vector of service ranges $\servicevec_p$ has $p$ fraction of flexible nodes, and $1-p$ fraction of inflexible nodes.\footnote{All such $\servicevec_p$ are statistically equivalent, since demand and supply locations are i.i.d. uniform.} 
Given $(\drivervec,\servicevec_p)$ and $\ridervec$, denote
\(
    G \sim \mathbb{G}(n,\servicevec_p)
\)
as the bipartite graph $G\big((\drivervec,\servicevec_p),\ridervec\big)$ under the compatibility rule $|\driver_i-\rider_j|\le \service_i/n$. This is the \emph{dual service range} model with parameters $\base,\extra,p$. Define
\[
    \nu_n(\base,\extra,p) \; \triangleq \; \Expect\big[\,|\calM(G)|\,\big],
\]
the expected size of a maximum matching in $G$. 

We compute $\nu_n(\base,\extra,p)$ via a Markov chain approximation of a greedy matching algorithm. To that end, Let $(u_t)_{t\in \mathbb{N}}$, $(v_t)_{t\in \mathbb{N}}$, and $(w_t)_{t\in \mathbb{N}}$ be mutually independent random processes with
\[
    u_t \stackrel{\text{i.i.d.}}{\sim} \Exp(p), \quad v_t \stackrel{\text{i.i.d.}}{\sim} \Exp(1-p), \quad w_t \stackrel{\text{i.i.d.}}{\sim} \Exp(1) \,.
\]
Consider a discrete-time Markov chain
\( 
    \potential(t) \triangleq  \begin{bmatrix} x(t), ~ y(t) \end{bmatrix}
\)
on $\reals^2$, with one-step transition $\nabla \potential(t) \triangleq \potential(t+1) - \potential(t)$ defined at each time $t$ using $u_t, \, v_t, \, w_t$ as shown in Figure~\ref{fig-Markov-Chain-Regions}. A formal description is provided in~\eqref{eq: markov-chain-dynamics}. The chain $\potential(t)$ convergences to a unique stationary distribution (\Cref{prop:psi-ergodic}). We use $\potential(t)$ to characterize the expected size of a maximum matching in this model.

\begin{figure}[t]
    \centering
    \includegraphics[width=0.8\linewidth]{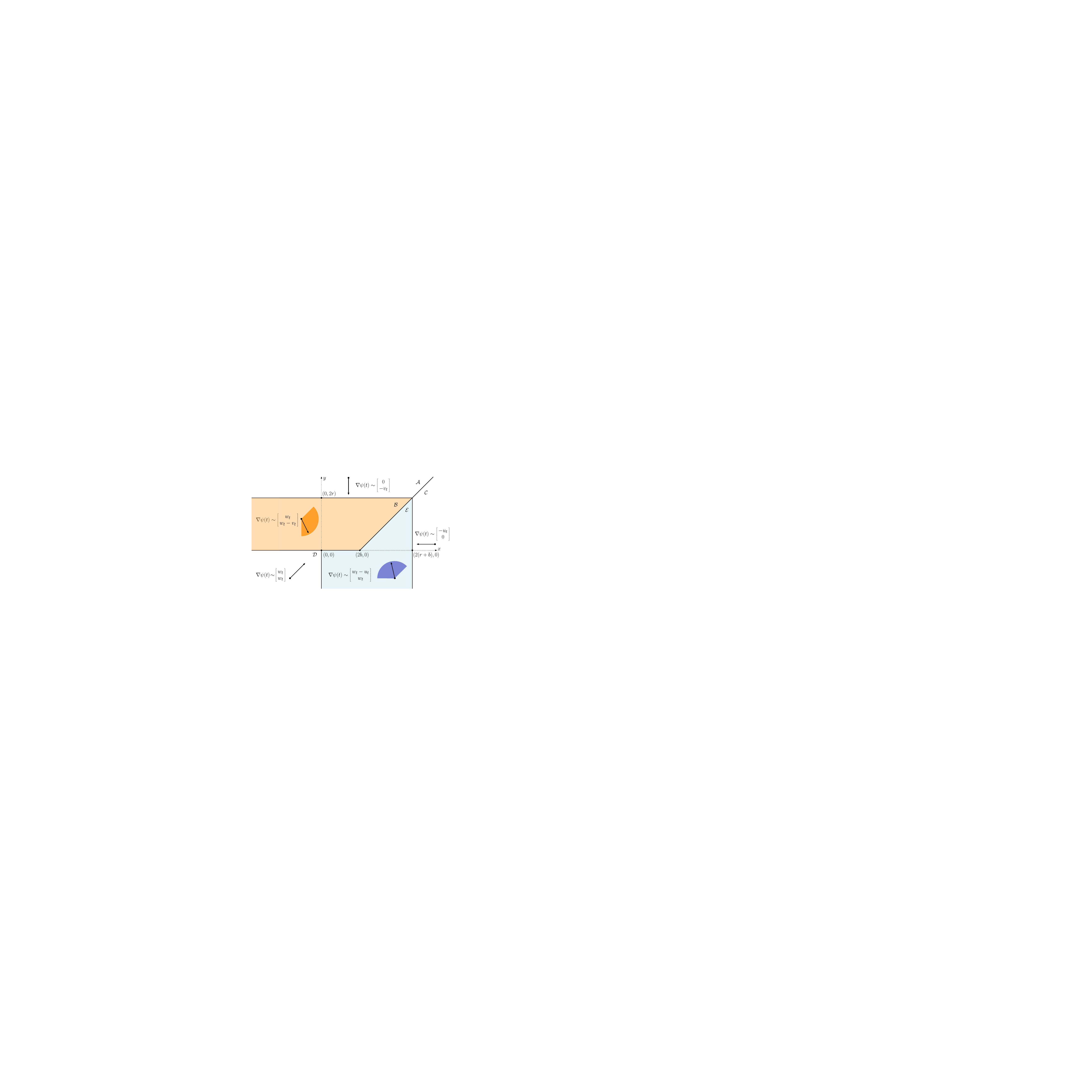}
    \caption{Markov dynamics of $\potential(t)$ with dots and arrows in vector field notation: a dot represents the current location of $\potential(t)$, and the arrow denotes its movement direction from that dot. For regions $\mathcal{B}$ and $\mathcal{E}$, the shaded sector represents the set of potential movement directions for $\potential$.} 
    \label{fig-Markov-Chain-Regions}
\end{figure}

\begin{theorem} \label{thm: sizes-of-maximum-matching}
    Let $\nu_n(\base, \extra, p)$, $\nu_n^{\mathrm{NF}}(\base, \extra, p)$, $\nu_n^{\mathrm{F}}(\base, \extra, p)$ respectively denote the expected number of matched supply, inflexible supply, and flexible supply nodes in the dual service range model. Then, 
    \begin{align}
        \frac 1n \,\nu_n(\base,\extra,p) &= \left( \frac{1 - 2\,F_{\calD}}{1 - F_{\calD}} \right) \,+\, o(1) \, , \label{eq: match-total}
        \\
        \frac 1 n \, \nu_n^{\mathrm{F}}(\base,\extra,p) &= \left(\frac{F_{\calE}}{ F_{\calE} + F_{\calC}}\right) p \,+\, o(1) \, , \label{eq: match-flex}\\
        \frac 1 n \, \nu_n^{\mathrm{NF}}(\base,\extra,p) &= \left( \frac{F_{\calB}}{ F_{\calB} + F_{\calA}} \right) (1-p) \,+\, o(1)\, ,  \label{eq: match-non-flex}      
    \end{align}
    where $F_{R}$ is the stationary measure of the Markov chain $\potential(t)$ on region $R \in \{\calA, \calB, \calC, \calD, \calE\}$.
\end{theorem}

When one of $\base, \extra$ or $p$ equal zero, we present a closed-form formula for the fraction of matched nodes, by explicitly computing the stationary distribution.

\begin{theorem} \label{thm: extreme-cases}
In the dual service range model, 
\begin{enumerate}
    \item[\( \mathrm{(i)} \)] If $\base = 0$, the expected size of a maximum matching satisfies
    \begin{align}
        \frac 1 n \, \nu_n(0, \extra, p) = 
            \frac{e^{2 p \extra} - e^{2\extra}}{e^{2 p \extra} - \frac 1 p e^{2\extra}}   \,+\, o(1) \, .
    \end{align}
    \item[\( \mathrm{(ii)} \)] If $\extra = 0$ or $p=0$, the expected size of a maximum matching satisfies
    \begin{align}
        \frac 1 n \, \nu_n(\base, 0 ,p) =  \, \nu_m(\base, \extra , 0) = \frac{\base}{\base + 1/2}  \,+\,  o(1) \, . \label{eq: size-of-max-matching-eps-equals-0}
    \end{align}
\end{enumerate}
\end{theorem}

\begin{remark}
The case $p = 0$ reduces to the setting where all supply nodes have the same service range $\base$, where \eqref{eq: size-of-max-matching-eps-equals-0} recovers~\cite[Theorem 1]{sentenac2025online}; When $\base = 0$, the setting is an imbalanced market with $n$ demand nodes and $m =\betaPar n$ supply nodes, with $\betaPar= p$. Note that a similar Markov chain embedding argument extends to a heterogeneous model with more than two service ranges. 
\end{remark}

\begin{remark}[Bounds for the dual service range model] \label{rmk: bounds-for-dual-service range}
    Some simple derivations, detailed in~\Cref{apx-remark}, show that our theorems entail the following bounds:
    \begin{itemize}
        \item (Upper bound) 
        \begin{align} \label{eq: upper-bound}
            \frac 1 n \, \nu_n(\base,\extra,p) \leq  \frac{\base + p\extra}{\base+ p \extra + 1/2} + o(1) \, .
        \end{align}
        \item (Lower bound) 
        \begin{align} 
            \frac 1 n \, \nu_n(\base,\extra,p) 
            \geq
            \sup_{p < q \leq 1} \!\left\{ 
                (1-q) \, \frac{ e^{2\base\frac{1-q}{1-p}} - e^{ 2\base} }{e^{2\base\frac{1-q}{1-p}} \!-\! \frac{1-p}{1-q}\cdot  e^{2\base} } \!+\! q \, \frac{ e^{2(\base+\extra)\frac{p}{q}} - e^{ 2(\base+\extra)} }{ e^{2(\base+\extra)\frac{p}{q}} \!-\! \frac{q}{p}\cdot e^{ 2(\base+\extra)} } 
            \right\}  \,+\, o(1) \, . \label{eq: lower-bound-with-q}
        \end{align}
        As a special case, letting $q \to p$ gives 
        \begin{align}
            \frac 1 n \, \nu_n(\base,\extra,p) \geq   (1-p) \, \frac{\base}{\base+ 1/2} + p \, \frac{\base+\extra}{\base+\extra+ 1/2}  + o(1) \, .   \label{eq: lower-bound-with-p}
        \end{align}
    \end{itemize}
\end{remark}

\subsection{Notation and paper organization}

For $x \in \reals$, denote $\max\{x,0\}$ as $x_+$. For a vector $\mathbf{v} \in \reals^\dimension$, and scalar $a \in \mathbb{R}$, let $\mathbf{v} \oplus a$ denote the augmented vector $(\mathbf{v}, a) \in \mathbb{R}^{\dimension+1}$ obtained by appending $a$ to $\mathbf{v}$. We use $B(x,r)$ to denote the $\ell^2$-ball of radius $r$ centered at $x$. We denote $X \sim \Exp(p)$ if $X$ has an exponential distribution with mean $1/p$. For any measurable set $A \subseteq \reals^\dimension$, we denote by $\mathrm{Vol}(A)$ its Lebesgue measure.
For two positive sequences $\{x_n\}$ and $\{y_n\}$, we write $x_n = O(y_n)$ if $x_n \le C \, y_n$ for an absolute constant $C$ and for all $n$; and we write $x_n = \Theta(y_n)$ if both $x_n = O(y_n)$ and $y_n = O(x_n)$. 

The rest of the paper is organized as follows. In Section~\ref{sec-uniformity-principle}, we establish the uniformity principle (Theorem~\ref{thm:uniformity}); the proof in one dimension is given in Section~\ref{sec-one-dimension}, and the extension to higher dimensions is developed in Section~\ref{sec-high-dimension}. Section~\ref{sec-markov-chain} then studies the dual-service range model and proves Theorems~\ref{thm: sizes-of-maximum-matching} and~\ref{thm: extreme-cases}. Simulation results are presented in Section~\ref{sec-simulations}, and concluding remarks appear in Section~\ref{sec-conclusion}. Proofs of supporting lemmas and auxiliary results are deferred to the appendices.


\section{Analysis for the uniformity principle} \label{sec-uniformity-principle}

In this section we prove the uniformity principle~(\Cref{thm:uniformity}), {with technical lemmas stated here whose proofs are deferred to~\Cref{sec-deferred-proofs}.} We analyze \( \dimension=1 \) and \( \dimension \geq 2 \) separately because the underlying RGGs exhibit a structural dichotomy: when $\dimension=1$, all connected components are of size \( O(\log n) \) with high probability, while for \( \dimension \geq 2 \) a giant component of size \( \Theta(n) \) emerges once the average degree crosses a constant threshold (\cite[Chs. 10–11]{penrose2003rgg}). 
The \( \dimension = 1 \) proof localizes to small components and yields sharper bounds, whereas the argument for \( \dimension \geq 2 \) controls long-range dependencies induced by the giant component, leading to a slightly loose bound. 

\subsection{Analysis for~\texorpdfstring{\( \dimension = 1 \)}{}} \label{sec-one-dimension}
For \( \dimension = 1 \), we establish the uniformity principle for supply and demand distributions on \( [0,1] \) whose densities are Lipschitz continuous and bounded from below. This covers the uniform distribution as a special instance. The following proposition implies~\Cref{thm:uniformity} for~\ref{U:1} and~\prettyref{rmk:general}.

\begin{proposition}[Uniformity principle, $\dimension =1$] 
\label{prop:schur-concavity}
Let
\(
    \left\{s_i\right\}_{i=1}^n \stackrel{\mathrm{i.i.d.}}{\sim}\mathbb{D}_1
\)
and
\(
    \left\{d_j\right\}_{j=1}^m \stackrel{\mathrm{i.i.d.}}{\sim} \mathbb{D}_2
\)
be mutually independent. Assume that \( \mathbb{D}_1 \) and \( \mathbb{D}_2 \) have densities $g_1$ and $g_2$ respectively on $[0,1]$, that are both $\smoothness$-Lipschitz and bounded in
\(
[1/\smoothness, \, \smoothness]
\)
for some $\smoothness \geq 1$. For any $\MPar \geq \max\{1,\betaPar^{-1}\}$ and any two service range vectors \( \servicevec , \, \servicevec' \in [0, \, \MPar]^{n} \) satisfying
\[
\servicevec \succeq \servicevec' ~~~~ \mbox{ and } ~~~~  \| {\servicevec}^{\downarrow} - {\servicevec'}^{\downarrow} \|_1> \theta n   ~~~~  \mbox{ for some } \theta > 0 \, ,
\]
the expected matching size\footnote{
    In this proposition and its proof, we slightly abuse notation by drawing the positions from \( \mathbb{D}_1 \) and \( \mathbb{D}_2 \) rather than \( \Uniform([0,1]) \) in all the relevant definitions (e.g. the definition of \( \muM\left( \servicevec\right) \) ).
} 
satisfies for all \( n \geq 8\MPar \), 
\[
    \muM ( \servicevec' ) - \, \muM ( \servicevec ) \, \geq \,  \CMB \, \theta^{3} n -  \alphaMB \,\log n \, ,
\] 
where
\begin{align}
    \CMB \triangleq \frac{7}{384} \left[1-\frac{\smoothness^4+1}{(\smoothness^2+1)^2}\right] e^{-(8  - 4 \betaPar)\MPar \smoothness} (\MPar \smoothness)^{-3} \,, 
    ~~~~ 
    \alphaMB \triangleq 8 e^{4\MPar\smoothness\betaPar} \, \MPar\smoothness  + 20 \MPar \smoothness  \, . \label{eq:def-CMB-NMB-alphaMB}
\end{align}
\end{proposition}

\begin{proof}[Proof]
First, we present the following proposition that shows \emph{reallocating} some service range from a larger-range supply node to a smaller-range one strictly increases the expected matching size.

\begin{proposition}[Gain from a pair of supply nodes]\label{prop:domination2drivers-shift}
Given any constant \( \MPar\geq \max\{1,\betaPar^{-1}\} \), let \( \servicevec\in[0,\MPar]^{\,n-2} \) be a vector of service ranges. For any two additional ranges \( \base_1, \,  \base_2\) with \( 0 \leq \base_1 \leq \base_2 \leq \MPar \), any  \( \shift\in[0,(\base_2 - \base_1)/2] \), and all \( n \geq 8 \MPar \),
\begin{align}
    \muM \left(\servicevec \oplus \base_1 + \shift \oplus \base_2 - \shift\right)
    -
    \muM \left(\servicevec\oplus \base_1 \oplus \base_2\right)
    \geq 
     {8} \CMB  \, \shift^3 - \alphaMB  \, \log(n)/n \, , \label{eq:two-driver-shift}
\end{align}
\end{proposition}
We now apply~\Cref{prop:domination2drivers-shift} iteratively along a sequence of \(\sfT\)-transforms that maps \( \servicevec \) to \( \servicevec' \), and then accumulate the resulting gains. Formally, a \(\sfT\)-transform is defined as follows.

\begin{definition}[\(\sfT\)-transform]\label{def:Ttransform-subsec}
    For indices \( i \neq j \) and \( t\in[0,1] \), the \(\sfT\)-transform \( T_{ij}(t) \) {of a vector \( \mathbf{z} \)} replaces \( (z_i,z_j) \) by \( \left((1-t)z_i+t z_j, (1-t)z_j+t z_i\right) \), leaving other coordinates unchanged. Equivalently, it increases (resp. decreases) the smaller (resp. larger) coordinate by \( \tau = t|z_i-z_j| \).
\end{definition}

Consider a \(\sfT\)-transform step acting on a pair \( (\base_1,\base_2)\in[0,\MPar]^2 \) with \( \base_1 \leq \base_2 \), replacing \( (\base_1,\base_2) \) by \( (\base_1+\tau,\base_2-\tau) \) while keeping other coordinates unchanged. By~\prettyref{prop:domination2drivers-shift},
\[
    \muM(\cdot\oplus \base_1+\tau\oplus \base_2-\tau) - \muM(\cdot\oplus \base_1\oplus \base_2)
    \geq  
    {8}\, \CMB  \, \tau^3 - \alphaMB \, {\log (n)}/{n}.
\]
We now chain these local shifts along a finite \(\sfT\)-transform decomposition of the total change; the next lemma supplies this decomposition and records the overall moved amount.

\begin{lemma} \label{lem:HLP}
Let \(\mathbf x,\, \mathbf y \in \reals_{\geq 0}^n \) have the same mean and \( \mathbf x\succeq \mathbf y \). Then:
\begin{itemize}
    \item[\( \mathrm{(a)} \)] (Hardy–Littlewood–Pólya) 
    There is a sequence of $T_\star$ steps of \(\sfT\)‑transforms mapping \(\mathbf x \) to \( \mathbf y \), where $T_{\star} \leq n-1$. 
    \item[\( \mathrm{(b)} \)] If \( \tau_r \) is the moved amount in step \( r \), then \( \sum_{r=1}^{T_{\star}} \tau_r =  \frac12 \|\mathbf x^\downarrow -\mathbf y^\downarrow \|_1 \). 
\end{itemize}
\end{lemma}
By \prettyref{lem:HLP}  and summing over  $T$ steps of \(\sfT\)‑transforms mapping \(\mathbf x \) to \( \mathbf y \), we have
\begin{align}
    \muM\left(\servicevec'\right) - \muM\left(\servicevec\right)
    \geq  {8} \,  \CMB  \sum_{\ell=1}^{T_*} \tau_\ell^3 - T_* \, \alphaMB  \frac{\log n}{n} 
    \geq  {8} \, \CMB  \sum_{\ell=1}^{T_*}  \tau_\ell^3 - \alphaMB \log n  \,, \label{eq:telescope_1}
\end{align}
where the last inequality holds by \prettyref{lem:HLP}(a). Since \( x\mapsto x^3 \) is convex and the total shift is fixed,
\begin{align}
    \sum_{\ell=1}^{T_*}\tau_\ell^{ 3} 
    \, \geq \,  T_*\left( \frac{\sum_{\ell}\tau_\ell}{T_*}\right)^{3}
    \, = \,    \frac{\left(\sum_{\ell}\tau_\ell\right)^{3}}{T_*^{2}} 
    \, \stepa{=} \,  \frac{ \big( \|\servicevec^\downarrow - \servicevec'^\downarrow \|/2  \big)^3 }{n^2} 
    \, \stepb{\geq} \, \frac{(n\theta)^{3}}{ {8} n^2}
    \, = \,    { \frac18 } \,  \theta^{3} n \,, \label{eq:tau_cube} 
\end{align}
where (a) uses~\prettyref{lem:HLP}(b) and (b) uses \( \norm{\servicevec^\downarrow - \servicevec'^\downarrow}_1 \geq \theta n \). 
By~\prettyref{eq:telescope_1} and~\prettyref{eq:tau_cube}, our result follows. 
\end{proof}

\subsubsection{Proof of~\prettyref{prop:domination2drivers-shift}}
\paragraph{\(\sfT\)-transform gain decomposition} 
We start by decomposing the expected matching gain from a given \(\sfT\)-transform into individual success probabilities and a connectivity correction term. Consider a graph \( G((\drivervec,\servicevec),\ridervec) \).
For a new supply node \( x = (\driver_x,\service_x)\in[0,1]\times\mathbb{R}_{\geq 0} \), the augmented graph \( G \oplus x \) is the bipartite graph induced by supply node locations \( \drivervec \oplus \driver_x \), service ranges \( \servicevec \oplus \service_x \), and demand node locations \( \ridervec \).  In the random geometric model, define the single-driver success probability
\[
    \deltaM(\service, \servicevec) \triangleq \Expect_{\driver \sim \mathbb{D}_1 \independent G \sim \mathbb{G}(m,\servicevec)}\, \left[M\left(G\oplus (s,r)\right)-M(G) \right],
\]
and the connectivity probability
\[
    \rhoS{\servicevec}(\service_x, \service_y) \triangleq 
    \Expect_{\driver_x,\driver_y \, \stackrel{\text{i.i.d.}}{\sim} \, \mathbb{D}_1 \independent G \sim \mathbb{G}(m,\servicevec)} \, \left[\mathds{1}\{(\driver_x, \service_x)\text{ and }(\driver_y, \service_y)\text{ connected in }G\oplus (\driver_x, \service_x)\oplus (\driver_y, \service_y)\}\right].
\]
The following lemma establishes the decomposition.

\begin{lemma}[Decomposition of \(\sfT\)-transform gain]\label{lem:main_decomposition-shift}Under the conditions of~\Cref{prop:domination2drivers-shift},
\begin{align}
    & \muM\left(\servicevec \oplus  (\base_1+\shift) \oplus (\base_2-\shift) \right)-\muM \left(\servicevec \oplus \base_1 \oplus  \base_2 \right) \nonumber \\
    & ~~~~  \geq \deltaM(\base_1+\shift,\servicevec) + \deltaM(\base_2-\shift,\servicevec) - \deltaM(\base_1,\servicevec) - \deltaM(\base_2,\servicevec) - \rhoS{\servicevec} \left(\base_1+\shift, \base_2-\shift\right). \label{eq:main_decomposition-shift}
\end{align}
\end{lemma}

\paragraph{Strong concavity of success probability.} 
The core technical contribution, proved in~\Cref{sec:proof-strongconcave1D-shift}, establishes that the success probability function exhibits strong concavity up to a vanishing correction term. The proof relies on a geometric analysis of the node positions. 
    
\begin{proposition}[Strong concavity]\label{prop:strongconcave1D-shift}
Under the conditions of~\Cref{prop:domination2drivers-shift},
\begin{align*}
    \deltaM(\base_1+\shift,\servicevec) + \deltaM(\base_2-\shift,\servicevec) - \deltaM(\base_1,\servicevec) - \deltaM(\base_2,\servicevec)
    \geq {8\,} \CMB \, \tau^3 - {14 \smoothness \, \MPar }/{n} \, , 
\end{align*}
where $\CMB $ is defined in~\eqref{eq:def-CMB-NMB-alphaMB}.
\end{proposition}
\paragraph{Bounding connectivity effects.} Next, we show that bounded service ranges naturally fragment the geometric graph, making interference between random supply nodes negligible.

\begin{lemma}[Rare connectivity]\label{lem:rarelysamecomponent}
Under the conditions of~\Cref{prop:domination2drivers-shift},
\[
    \rhoS{\servicevec} \left(\base_1+\shift, \base_2-\shift\right) \le  \frac{4\smoothness \MPar }{n}\left(\left\lceil 2e^{ 4 \smoothness\MPar m/n}\log n\right\rceil+2\right) +  \frac{1}{2\MPar\, n} \, .
\]
\end{lemma}

\begin{figure}[t]
    \centering
    \subfigure[Strong concavity of $\widetilde{\delta}_m(\cdot,\servicevec) $]{
    \includegraphics[width=0.4\linewidth]{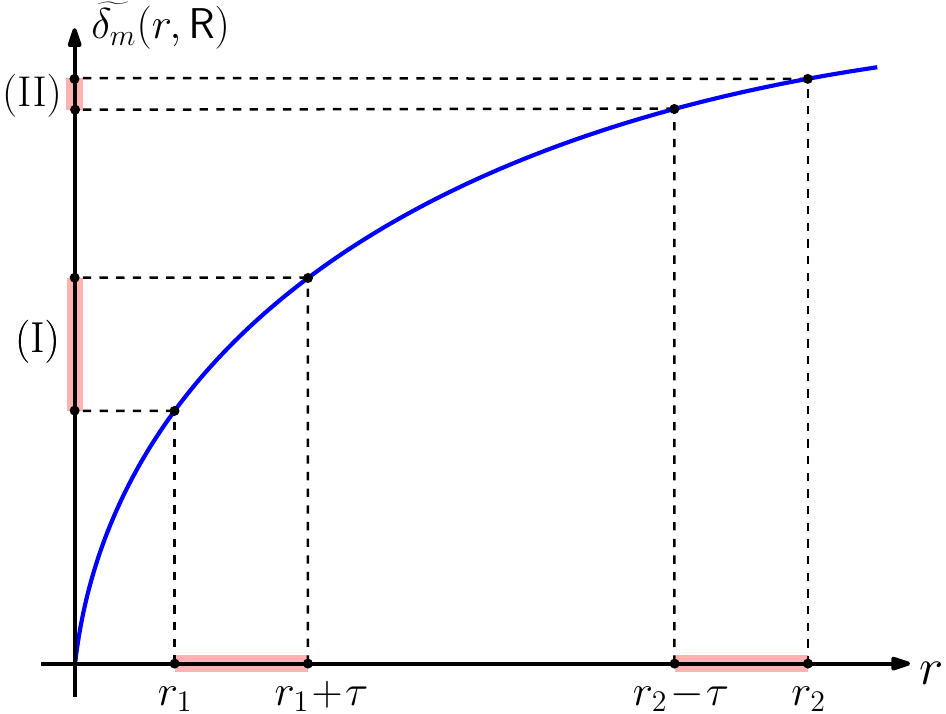}
    \label{fig:concavity}
    }
    \hfill 
    \subfigure[Illustration: $\Gap_{\shift}(\riderset,x)$ and no supply imply $\Gap_{\shift}(\riderset_+,x)$]{
    \includegraphics[height= 1.87 in]{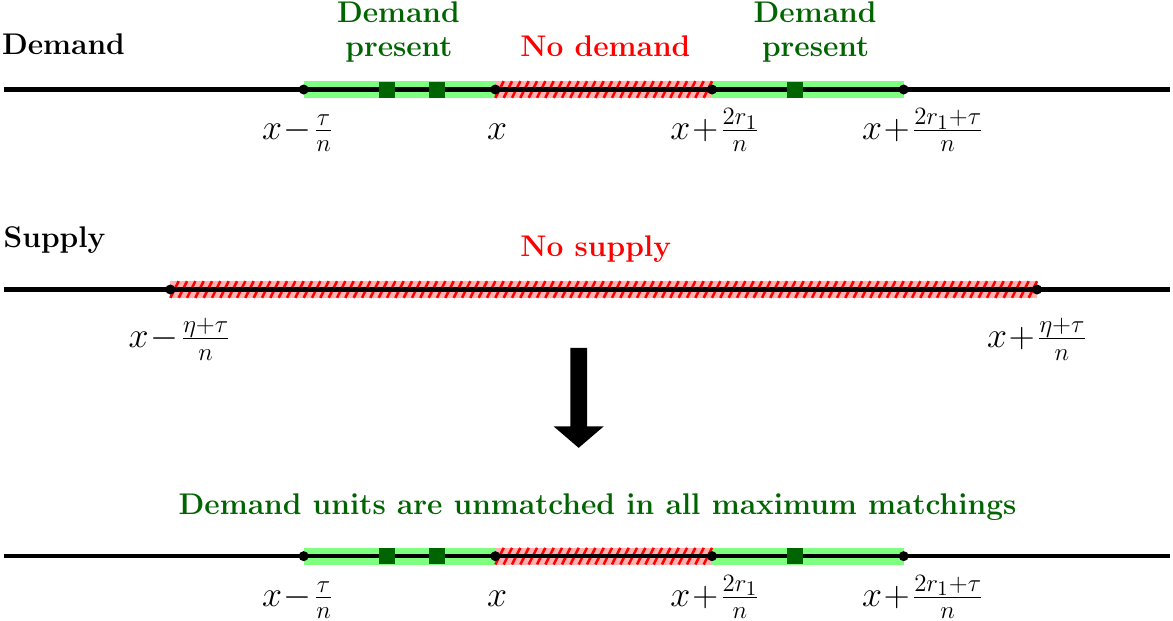}
    \label{fig:gap}
    }
    \caption{Illustrations used in proofs}
    \label{eq:proofconcavityillustration}
\end{figure}

The full proof of this lemma is deferred to~\Cref{sec:proofoflemmararelysamecomponent}. It proceeds by showing that, with high probability, the connected components of the graph span only a fraction \( O(\log n/n) \) of \( [0,1] \). 
Combining~\eqref{eq:main_decomposition-shift},~\prettyref{lem:rarelysamecomponent}, and~\prettyref{prop:strongconcave1D-shift}, we obtain
\begin{align*}
    &\muM\left(\servicevec \oplus (\base_1+\shift) \oplus (\base_2-\shift)\right)-\muM\left(\servicevec \oplus \base_1 \oplus \base_2\right)  \geq {8} \,  \CMB \, \tau^3 - \frac{14 \smoothness\MPar }{n}  - \rhoS{\servicevec}(\base_1+\shift,\base_2-\shift) \\
    &~~~~ \geq  {8} \, \CMB  \, \shift^3 -  \frac{4\smoothness \MPar }{n}\left(\left\lceil 2e^{ 4 \smoothness\MPar m/n}\log n\right\rceil+2\right) -  \frac{1}{2\MPar \, n} -\frac{14 \smoothness\MPar }{n} \geq {8}  \, \CMB  \, \shift^3  - \alphaMB \, \frac{\log n}{n},
\end{align*}
where the last inequality holds because \( \MPar \geq 1 \) and~\eqref{eq:def-CMB-NMB-alphaMB}.

\subsubsection{Proof of~\prettyref{prop:strongconcave1D-shift}}\label{sec:proof-strongconcave1D-shift}
    
For a graph \( G = G((\drivervec,\servicevec), \ridervec) \),~let
\begin{align}
        \riderset_+(G) \triangleq \bigcup_{\pi: \, \text{maximum matching of }G} \{\rider_j: j \text{ is unmatched in } \pi\} \, . \label{eq:D_+}
\end{align}
denote the set of demand node locations that are unmatched in some maximum matching. The following lemma gives an expression for $\deltaM(\service,\servicevec)$ in terms of $\riderset_+$. The idea is that if a supply node with range $r$ can be matched to a demand node in $\riderset_+$, then the matching size can be augmented.

\begin{lemma} \label{lmm:service_calR_+}
For any \( \newrange \geq 0 \) and service range vector \( \servicevec \in \reals_{\ge 0}^n\),
\begin{align}
    \deltaM(\newrange,\servicevec) = \Expect_{G\sim \mathbb{G}(m,\servicevec)} \left[ \, \int_{0}^{1} \mathds{1}\left\{\riderset_+(G)\cap\Big(\Big[x- \frac{\newrange}{n}, x+ \frac{\newrange}{n}\Big]\cap[0,1]\Big)\neq\emptyset\right\} g_1(x)  \, \diff x\right]. \label{eq:service_calR_+}
\end{align}
\end{lemma}

Define the forward-window surrogate
\begin{align} \label{eq:forward-window-surrogate}
    \widetilde{\delta}_m(\newrange ,\servicevec) \triangleq \mathbb{E}\bigg[\int_{0}^{1} \mathbf{1}\Big\{\riderset_+(G)\cap \Big(\Big[x, x+ \frac{2 \newrange }{n}\Big]\cap[0,1]\Big)\neq\emptyset\Big\} \,  g_1(x) \, \diff x\bigg],
\end{align}
where the interval \( [x, x+ 2\newrange/n] \) has the same length as its symmetric counterpart but shifts forward. By~\prettyref{lmm:service_calR_+}, and as \( g_1(\cdot) \) is \( \smoothness \)-Lipschitz,
\begin{align}
    \big|\deltaM(\newrange,\servicevec)- \widetilde{\delta}_m(\newrange,\servicevec) \big| \leq {3\, \smoothness\newrange}/{n}, ~~~~ \mbox{for each } \newrange \geq 0. \label{eq:delta_tilde_general}
\end{align}
Consequently, to prove~\Cref{prop:strongconcave1D-shift}, it suffices to show the following lemma.
\begin{lemma} \label{lem:concavitytwoterms}
    Under the same conditions as~\Cref{prop:strongconcave1D-shift},
    \begin{align}
        \underbrace{
            \widetilde{\delta}_m(\base_1 + \shift,\servicevec) - \widetilde{\delta}_m(\base_1,\servicevec)}_{\triangleq \, \termI 
        }
        \,-\, 
        \underbrace{
            \big( \widetilde{\delta}_m(\base_2, \servicevec) - \widetilde{\delta}_m(\base_2 - \shift, \servicevec) \big) }_{\triangleq \, \termII     
        } 
        \geq {8} \,\CMB \shift^3 - \frac{2\smoothness \, \base_2}{n} \, . \label{eq:shifted_lower_bound_general}
    \end{align}
\end{lemma}
The key is to show \( \widetilde{\delta}_m(\cdot,\servicevec) \) is strongly concave up to a vanishing correction term. We show:
\begin{align*}
    \termI - \termII  \ge  \mathbb{E} \int_0^1 \mathbf{1} \left\{ \Gap_\shift(\riderset_+(G),   x) \right\}   g_1(x) \,  \diff x -  \frac{2(\base_2-\base_1-\shift)}{n} \,,
\end{align*}
where for any \( x \in [0,1] \) and any set \( \mathcal{L} \) of locations, the event \( \Gap_{\shift}(\mathcal{L},x) \) is defined as
\[
    \Gap_\shift(\calL, x) \triangleq \left\{
    \calL\cap \Big[x - \frac{\shift}{n}, x \Big] \neq \emptyset,    \,\,
    \calL\cap \Big[x, x + \frac{2\base_1}{n} \Big] = \emptyset,    \,\,
    \calL\cap \Big[x + \frac{2\base_1}{n}, x + \frac{2\base_1+2\shift}{n}\Big] \neq \emptyset
    \right\}.
\]
To understand the proof, it is helpful to interpret $\termI$ as the expected value of the integrand of the indicator of the event that increasing the service range of a forward-looking node located at \( x \) from \( 2\base_1/n \)  to \( 2(\base_1+\shift)/n \) increases the maximum matching, and similarly for \( \termII \), corresponding to the increase from \( 2(\base_2-\shift)/n \) to \( 2\base_2/n \). 
For each \( x \in [0,1] \) (up to boundary effects) where the event holds in \( \termII \), there exists a corresponding position \( x + \frac{2(\base_2 - \base_1 - \shift)}{n} \in [0,1] \) where it also holds in \( \termI \). Positions \( x \) satisfying \( \Gap_\shift(\calL, x) \) are positions where the event holds in \( \termI \) but are not mapped to under this correspondence. 
We show that such gap patterns \( \Gap_\shift(\calL, x) \) are prevalent in the realized graph, as they occur whenever the set \( \riderset \) exhibits this structure and no element of \( \driverset \) lies within the corresponding region. The proof is illustrated in~\Cref{eq:proofconcavityillustration}, with details in~\Cref{sec:proofofconcavitytwoterms}.

\subsection{Analysis for \texorpdfstring{$\dimension\geq 2$}{}} \label{sec-high-dimension}

Comparing against the case $\dimension = 1$, two additional hurdles arise. First, the graph $G$ contains a giant component, which requires careful handling via trimming. Second, the the $\Gap_\shift(\riderset_+, \cdot )$ event used to prove~\Cref{prop:domination2drivers-shift} for $\dimension = 1$, does not generalize. To overcome this, in the proof of~\Cref{prop:strongconcavehighD}, we define \emph{special pattern cells}, which are cells with useful structural properties. The main result of this section is the following proposition, which implies~\cref{U:2} of~\Cref{thm:uniformity}. 
\begin{proposition}[Uniformity principle, $\dimension \ge 2$]  \label{prop:schur-concavity-highD-untrimmed}
Given $\dimension \geq 2$ and any constant $\MPar \geq \max\{1,\betaPar^{-1}\}$, for any two service range vectors $\servicevec,\servicevec' \in [0,\MPar ]^n$ satisfying
\[
    \servicevec \succeq \servicevec' ~~~~ \mbox{ and } ~~~~ \| {\servicevec}^{\downarrow} - {\servicevec'}^{\downarrow} \|_1> \theta n ~~~~ \mbox{for some } \theta > 0,
\]
the expected matching size satisfies for all $n \ge { \highdimN }$,
\[
    \muM\left( \servicevec'\right) -      \muM\left( \servicevec\right) > \highdimalpha \, \theta^{4} \, n
    - \beta_{\dimension,\MPar} \big(  \big(  \highdimalpha \, \theta^{4}/3 \big) \wedge \epsiloncondition\big)^{-\dimension}\,,
\]
where $\highdimN, \, \highdimalpha,\, \beta_{\dimension,\MPar},  \, \epsiloncondition$ are positive constants that depend only on their subscripts.
\end{proposition}

\noindent \textit{Proof.}
We first introduce a trimming operation inspired by the one in~\cite{magnolia}, to deal with the giant component that can be present in $G$ when \( \dimension \geq 2 \).

\begin{figure}[t]
    \centering
        \subfigure[Graph $G$]{
            \includegraphics[width=0.35\linewidth]{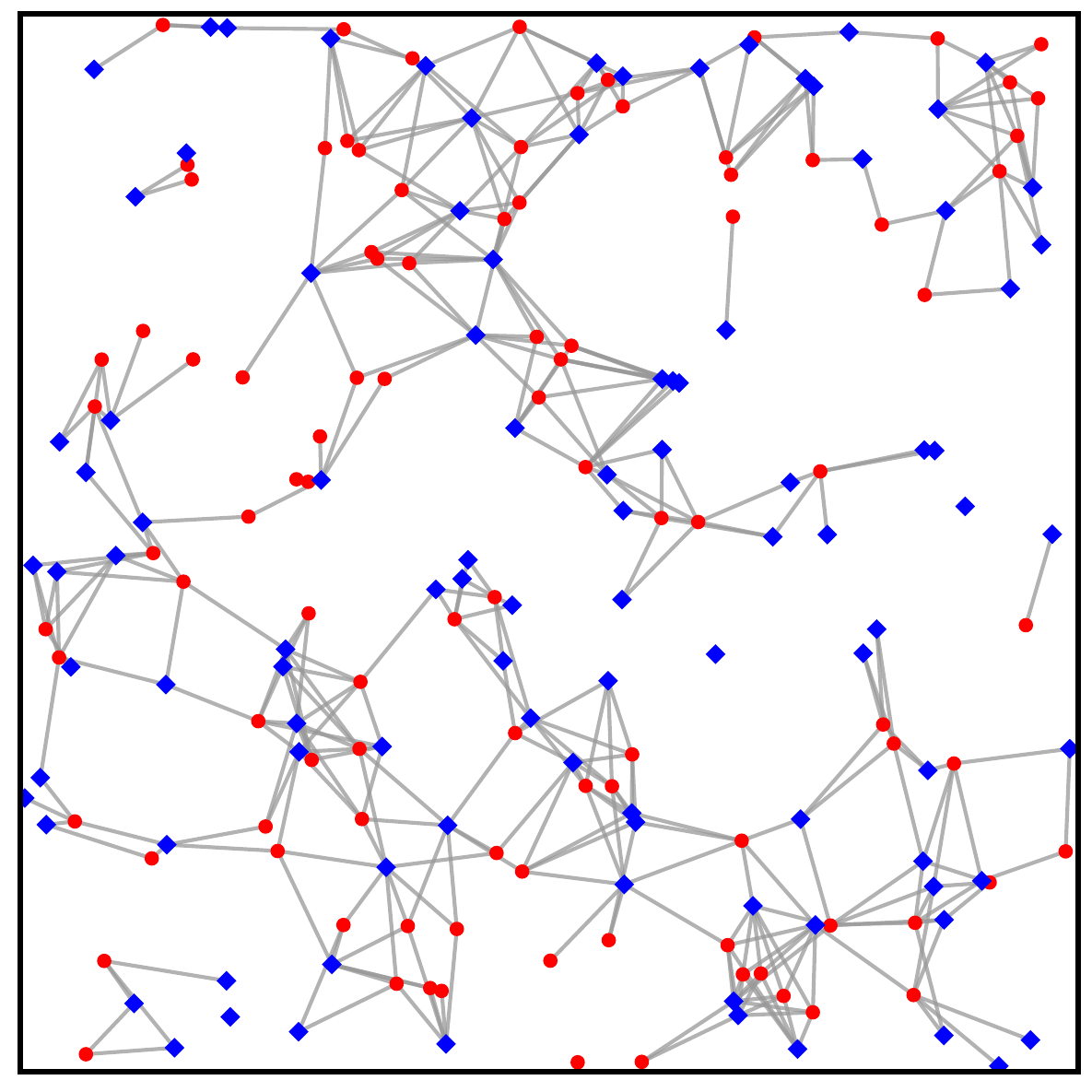}
            \label{fig-graph}
        } 
        \subfigure[Trimmed $G_\varepsilon$]{
            \includegraphics[width=0.35\linewidth]{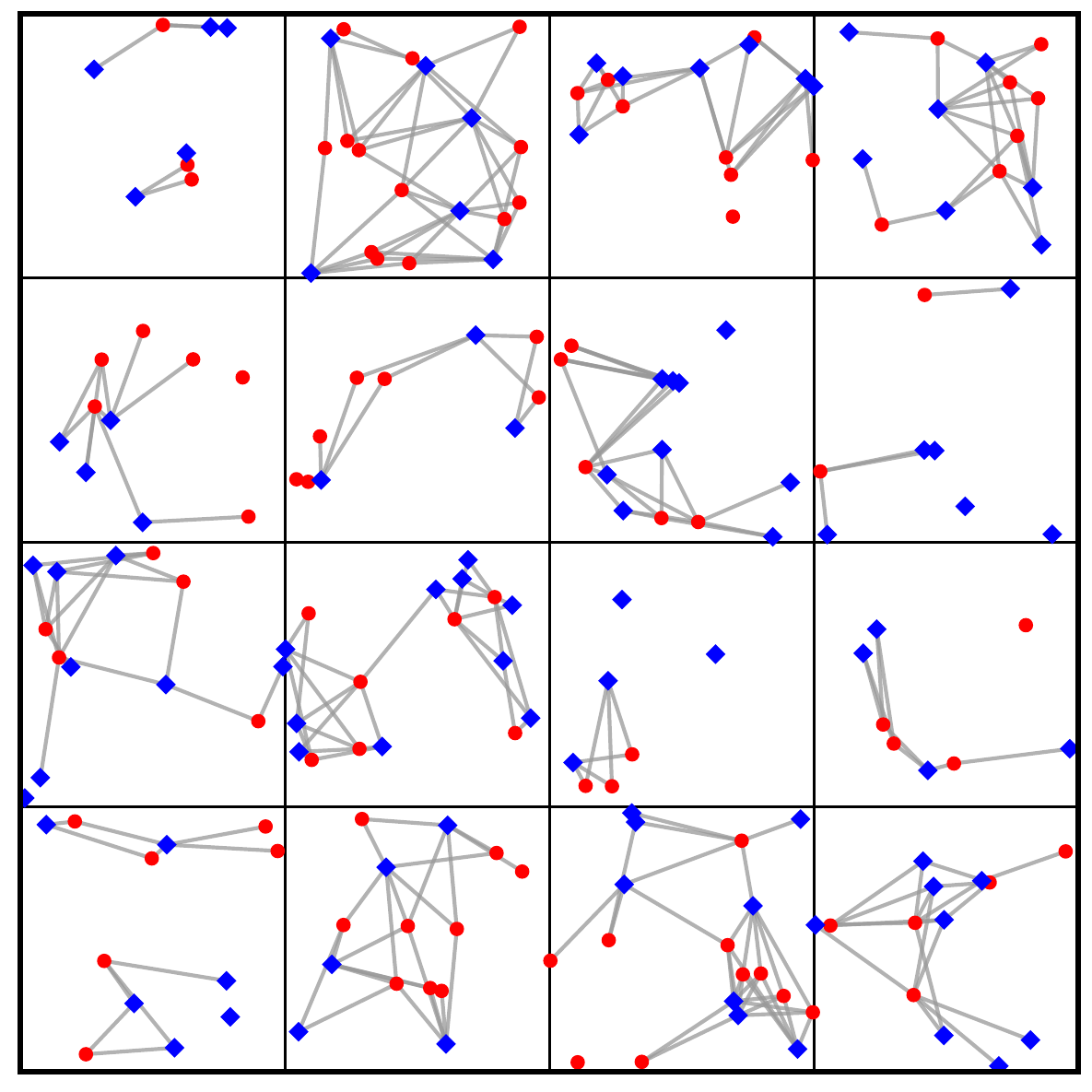}
            \label{fig-trimming}
        }
        \\
        \caption{Partition of $[0,1]^2$ into trimming cells.}
        \label{fig:trimmingcellillustration}
\end{figure}

\paragraph{Partition for trimming.}  Fix some constant $\varepsilon > 0$.
        Partition $[0,1]^\dimension$ into hypercubes (\emph{trimming cells}) of side-length
        \[
            L_T \triangleq 2 \, \varepsilon^{-1} \dimension \, \big(\MPar/n)^{1/\dimension},
        \]
        where the last cell in each dimension may be smaller from rounding. Construct a new graph $G_\varepsilon$ from $G$ by removing every edge whose endpoints are in different cells. This operation is shown in~\Cref{fig:trimmingcellillustration}. For any point $\rider$, denote by $\tcell(\rider)$ the trimming cell that contains it.

Define the expected matching size in the trimmed graph as:
\begin{align}
    \mu_{m}^\varepsilon (\servicevec) \triangleq  \Expect_{G\sim  \mathbb{G}(m , \servicevec) } \left[M(G_\varepsilon) \right]
\end{align}
The next lemma shows trimming has a small impact on the expected size of the maximum matching.
\begin{lemma}\label{lmm:trimmingbound}
   For any $\servicevec \in [0,\MPar]^n$ and  constant $\varepsilon>0$, we have
  $ | \mu_{m}^\varepsilon (\servicevec)- \mu_{m} (\servicevec)|\leq \varepsilon n.$
\end{lemma}
\begin{proof}
The matching size is changed by no more than the number of supply nodes that are  within distance 
$(\MPar/n)^{1/\dimension}$ of a cell boundary are affected by trimming. Thus,
\begin{equation*}
    | \mu_{m}^\varepsilon (\servicevec)- \mu_{m} (\servicevec)| \, \leq \, \sum_{i=1}^n \mathbb{P}\big(\rider_i \text{ is within } \left( \MPar/n\right)^{1/\dimension} 
    \text{ of the {boundary} of a cell}\big) \le \varepsilon n \,,
\end{equation*}
where the last inequality holds because for any $\dimension$, there are at most $\frac{\varepsilon n^{1/\dimension}}{2\dimension  \MPar^{1/\dimension}}$ cell boundaries, and hence the volume of the region within distance $(\MPar/n)^{1/\dimension}$ of the boundaries is less than $\varepsilon$.  
\end{proof}

\paragraph{\(\sfT\)-transforms in the trimmed graph}
We introduce the following proposition, which is the counterpart to~\Cref{prop:domination2drivers-shift}, and establishes the gain from $\sfT$-transforms in the trimmed graph. 

\begin{proposition}[Gain from a pair of supply nodes]\label{prop:domination2supply nodestrimmed}
    Given any constant $\MPar \geq \max\{1,\betaPar^{-1} \}$,  let  $\servicevec\in [0, \MPar]^{n-2} $ be a vector of service ranges. For any two additional ranges $\base_1, \base_2$ with $ 0 \leq \base_1 \leq \base_2 \leq \MPar$, any $\shift\in[0,(\base_2 - \base_1)/2]$, any $\varepsilon \leq \epsiloncondition$ and all $n \ge\highdimN$, 
    \begin{align}
        \mu_m^\varepsilon\left(\servicevec \oplus \left(\base_1+\shift\right) \oplus \left(\base_2-\shift\right)\right)-\mu_m^\varepsilon \left(\servicevec \oplus \base_1 \oplus \base_2 \right) \geq  \, \highdimalpha' \, \shift^{4} - {\highdimbeta}/{n}, \label{eq:average_better}
    \end{align}  
    where each constant $\highdimalpha'$ and $\beta_{\dimension,\MPar}$ depends only on the variables in the subscript.
\end{proposition}

Analogous to the proof of \prettyref{prop:schur-concavity}, by summing over $T_*$ steps of $\sfT$-transforms, 
\begin{align*}
    \muM^\epsilon(\servicevec') - \muM^\epsilon(\servicevec)
    & \geq  
     \, \highdimalpha' \sum_{\ell=1}^{T_*} \shift_{\ell}^{4}
    -   T_*  \frac{\highdimbeta}{n}  
    \stepa{\ge}  \highdimalpha'\sum_{\ell=1}^{T_*}  \shift_{\ell}^{4}
    -  \highdimbeta  
    \\
    &\stepb{\ge}   \highdimalpha \, \theta^{4} n 
    -  \highdimbeta \, , 
\end{align*}
where (a) holds by~\prettyref{lem:HLP}(a). For (b), since \(x\mapsto x^{4}\) is convex on \(\mathbb{R}_{\ge 0}\), Jensen's inequality gives
\[
    \sum_{\ell=1}^{T_*}\tau_\ell^{4}
    \geq T_*\Big(\frac{\sum_{\ell=1}^{T_*}\tau_\ell}{T_*}\Big)^{4}
    = \frac{\left(\sum_{\ell=1}^{T_*}\tau_\ell\right)^{4}}{T_*^{4-1}}
    \geq \frac{\left(\|\servicevec^\downarrow - \servicevec'^\downarrow\|_1/2\right)^{4}}{n^{4-1}}
    \geq 2^{-4} \, \theta^{4} n,
\]
where we used~\prettyref{lem:HLP}(b) and \(\|\servicevec^\downarrow -\servicevec'^\downarrow \|_1\ge \theta n\), and then set \(\highdimalpha = 2^{-4} \highdimalpha'\).
Finally, by Lemma~\ref{lmm:trimmingbound},
\[
\mu_m(\servicevec')-\mu_m(\servicevec)\ge \mu_m^\varepsilon(\servicevec')-\mu_m^\varepsilon(\servicevec)-2\varepsilon n.
\]
Setting \(\varepsilon = \big(\highdimalpha\,\theta^{4}/3\big)\wedge \epsiloncondition\) yields the stated bound.
\hfill $\Box{}$

\subsubsection{Proof of \prettyref{prop:domination2supply nodestrimmed}}

The proof follows the same structure as that of \prettyref{prop:domination2drivers-shift}. We use the same notation, generalizing $\delta_m(\service, \servicevec)$ to $\delta^\varepsilon_m(\service, \servicevec)$ and $\rho_m^{\servicevec}(\service_x, \service_y) $ to $\rho_m^{\servicevec,\varepsilon}(\service_x, \service_y) $ by replacing $G$ with $G_\varepsilon$ in the definitions. 
Throughout, trimming is performed after any augmentation: i.e. $(G\oplus x)_{\varepsilon}$ denotes the trimmed version of the augmented graph $G\oplus x$. Specifically, 
\[
    \delta^\varepsilon_m(\service, \servicevec) \triangleq \Expect_{\driver \sim \mathrm{Uniform}([0,1]^\dimension) \independent G \sim \mathbb{G}(m,\servicevec)}\, \Big[ M\big(  \big( G\oplus (s,r) \big)_{\varepsilon}\big)-M(G_{\varepsilon}) \Big],
\]
and
\[
    \rhoS{\servicevec}(\service_x, \service_y) \triangleq 
    \Expect \Big[\mathds{1} \Big\{(\driver_x, \service_x)\text{ and }(\driver_y, \service_y)\text{ connected in } \big( G\oplus (\driver_x, \service_x) \oplus (\driver_y, \service_y) \big)_{\varepsilon} \Big\}\Big] \, ,
\]
where the expectation is with respect to the randomness $\driver_x,\driver_y \, \stackrel{\mathrm{i.i.d.}}{\sim} \, \mathrm{Uniform}([0,1]^\dimension)$ which is independent of $G \sim \mathbb{G}(m,\servicevec)$.
The main technical challenge is to show the success probability function $\deltaM^\varepsilon(\cdot,\servicevec)$ is strongly concave.

\begin{proposition}[{Strong concavity}] \label{prop:strongconcavehighD} 
    Under the conditions of~\Cref{prop:domination2supply nodestrimmed},
    \begin{align}
        \deltaM^\varepsilon(\base_1+\shift,\servicevec) + \deltaM^\varepsilon(\base_2-\shift,\servicevec) - \deltaM^\varepsilon(\base_1,\servicevec) - \deltaM^\varepsilon(\base_2,\servicevec)
        \, \geq \,  \shift^{4} \highdimalpha' \, , \label{eq:concavity-shift-highdim}
    \end{align}
    where $\highdimalpha'>0$ only depends on $\dimension$, $\betaPar$ and $\MPar$. 
\end{proposition}

The following is the counterpart to~\prettyref{lem:main_decomposition-shift}. Its proof is omitted, since it is the same.

\begin{lemma}[Decomposition of $\sfT$-transform gain]  \label{lem:main_decomposition-shift-high-dim} 
    Under the conditions of~\Cref{prop:domination2supply nodestrimmed},
    \begin{align}
        & \muM^\varepsilon\left(\servicevec \oplus  (\base_1+\shift) \oplus (\base_2-\shift) \right)-\muM^\varepsilon \left(\servicevec \oplus \base_1 \oplus  \base_2 \right) \nonumber \\
        & ~~~~  \geq \deltaM^\varepsilon(\base_1+\shift,\servicevec) + \deltaM^\varepsilon(\base_2-\shift,\servicevec) - \deltaM^\varepsilon(\base_1,\servicevec) - \deltaM^\varepsilon(\base_2,\servicevec)
        - \rho_m^{\servicevec,\varepsilon}(\service_x, \service_y). \label{eq:main_decomposition-shift-high-dim}
    \end{align}
\end{lemma}
Next, we show that due to the trimming effect which fragments the geometric graph,  interference between supply nodes is negligible in the trimmed graph for $\dimension\ge 2$. 
\begin{lemma}[Rare connectivity]\label{lmm:rarelysamecomponenthighdim}
    For supply nodes $x=(\driver_x,\service_x)$ and $y=(\driver_y,\service_y)$ with $\driver_x,\driver_y \stackrel{\mathrm{i.i.d.}}{\sim}\Uniform([0,1]^\dimension)$ independent of $G\sim \mathbb{G}(m,\servicevec)$ with $\servicevec \in \reals_{\ge 0}^{n-2}$, 
    \[
        \rho_m^{\servicevec, \varepsilon}(\service_x, \service_y) \le   \highdimbeta / {n} \,, \text{ where } \beta_{\dimension,\MPar} \triangleq \left(2\dimension \right)^\dimension \MPar \, .
    \]
\end{lemma}
\begin{proof}
    By definition of the trimming operation, two nodes can belong to the same component only if they belong to the same cell, and, by design, the volume of a cell is $\left(\frac{2\dimension \, \MPar^{1/\dimension}}{\varepsilon n^{1/\dimension}}\right)^{\dimension}$.
\end{proof}
Combining~\prettyref{lmm:rarelysamecomponenthighdim} and~\prettyref{lem:main_decomposition-shift-high-dim} and~\prettyref{prop:strongconcavehighD}, we obtain~\prettyref{prop:domination2supply nodestrimmed}.

\subsubsection{Proof of \prettyref{prop:strongconcavehighD}}\label{sec:strongconcavityhighD} 
For any point $\rider_j \in [0,1]^\dimension$, let $\tcell(\rider_j)$ denote the trimming cell containing $\rider_j$. For any $\newrange \in [0, \MPar]$, define $B_\newrange(\rider_j)$ as the part of the ball at $\rider_j$ of radius $(\newrange/n)^{1/\dimension}$ in the trimming cell 
, i.e.

\[
    B_\newrange(\rider_j) \triangleq B\big(\rider_j,(\newrange/n)^{1/\dimension}\big) \cap \tcell(\rider_j).
\]
Recall \prettyref{eq:D_+} and take $\riderset_+ \equiv \riderset_+(G_\varepsilon)$ that is the set of demand node locations that are unmatched in some maximum matching. Analogous to \prettyref{lmm:service_calR_+}, we have the following lemma.  
\begin{lemma} \label{lem: high-dim-delta-formula}
For any $\newrange \ge 0$ and $\servicevec \in [0,\MPar]^n$, 
   \begin{align}\label{eq:equalitymatchprobavolume}
    \delta^\varepsilon_m\left(\newrange,\servicevec\right) = \Expect_{G \sim \mathbb{G}(m,\servicevec)} \left[  \mathrm{Vol}_{\newrange}(\riderset_+)\right] \,, \quad \text{where } \mathrm{Vol}_{\newrange}\left(\riderset_+\right) \triangleq \mathrm{Vol} 
    \bigg( \bigcup_{\rider_j \in \riderset_+}  B_\newrange(\rider_j)
    \bigg) \,.
\end{align} 
\end{lemma}
\noindent \textbf{Partition for pattern.} 
Partition $[0,1]^\dimension$ into {a set $\calA$ of} hypercubes (\emph{pattern cells}) of side length 
\begin{align}
    L_p \triangleq  \frac{1}{n^{1/\dimension}} \, \patternside, ~~~~ \mbox{where } \patternside \triangleq 6 \MPar ^{1/\dimension}. \label{eq: def-w-k}
\end{align}
Among them, a pattern cell that satisfies a structural property is called \emph{special pattern cell}: i) it is  entirely contained in a trimming cell: ii) it contains exactly two points of $\riderset_+$ in a ball of radius $R'''$ centered at the cell center, with one of them in a ball of radius $R$ and the other inside an annulus of with inner and outer radii $R'$ and $R''$ respectively (see~\Cref{fig:Uaillustration} for an illustration, and~\eqref{eq:patterndef} for a formal definition). We denote by $\calA^*[\riderset_+]$ the subset of special pattern cells in $\calA$. Then, we can decompose the contribution of $\calA$ to $\delta^\varepsilon_m \left(\newrange,\servicevec\right)$ as the the part contributed by $\calA^*[\riderset_+]$ and the part contributed by $\calA \backslash \calA^*[\riderset_+]$.
\begin{align}\label{eq:decomporho1rho2}
  \delta^\varepsilon_m\left(\newrange,\servicevec\right)
    & =  
    \underbrace{\mathbb{E}_{G \sim \mathbb{G}(m,\servicevec)}\bigg[\sum_{A \in \patterncells   
        \setminus \activepatterncells}  \!\!\!\! \text{Vol}\bigg( \bigcup_{\rider_j \in \riderset_+} B_\newrange(\rider_j)\cap A \bigg)\bigg]}_{ \triangleq \, \rho_{1}(\newrange)
    } \nonumber \\
    &~~~~ +
    \underbrace{\mathbb{E}_{G \sim \mathbb{G}(m,\servicevec)}\bigg[\sum_{A \in \activepatterncells}
        \!\!\!\! \text{Vol}\bigg( \bigcup_{\rider_j \in \riderset_+} B_\newrange(\rider_j)\cap A \bigg)\bigg]}_{ \triangleq \, \rho_{2}(\newrange)
    } \,,
\end{align}
The purpose of decomposing into pattern cells is that the probability of increasing the matching by adding a new node uniformly in a special pattern cell is a strongly concave function of the service range. Outside the special pattern cells, this probability is concave in the service range, i.e. $\rho_1(\newrange)$ is concave (see \prettyref{lmm:concaverho2}), while $\rho_2(\newrange)$ is strongly concave.
\begin{figure}[t]
    \centering
    \includegraphics[width=0.8\linewidth]{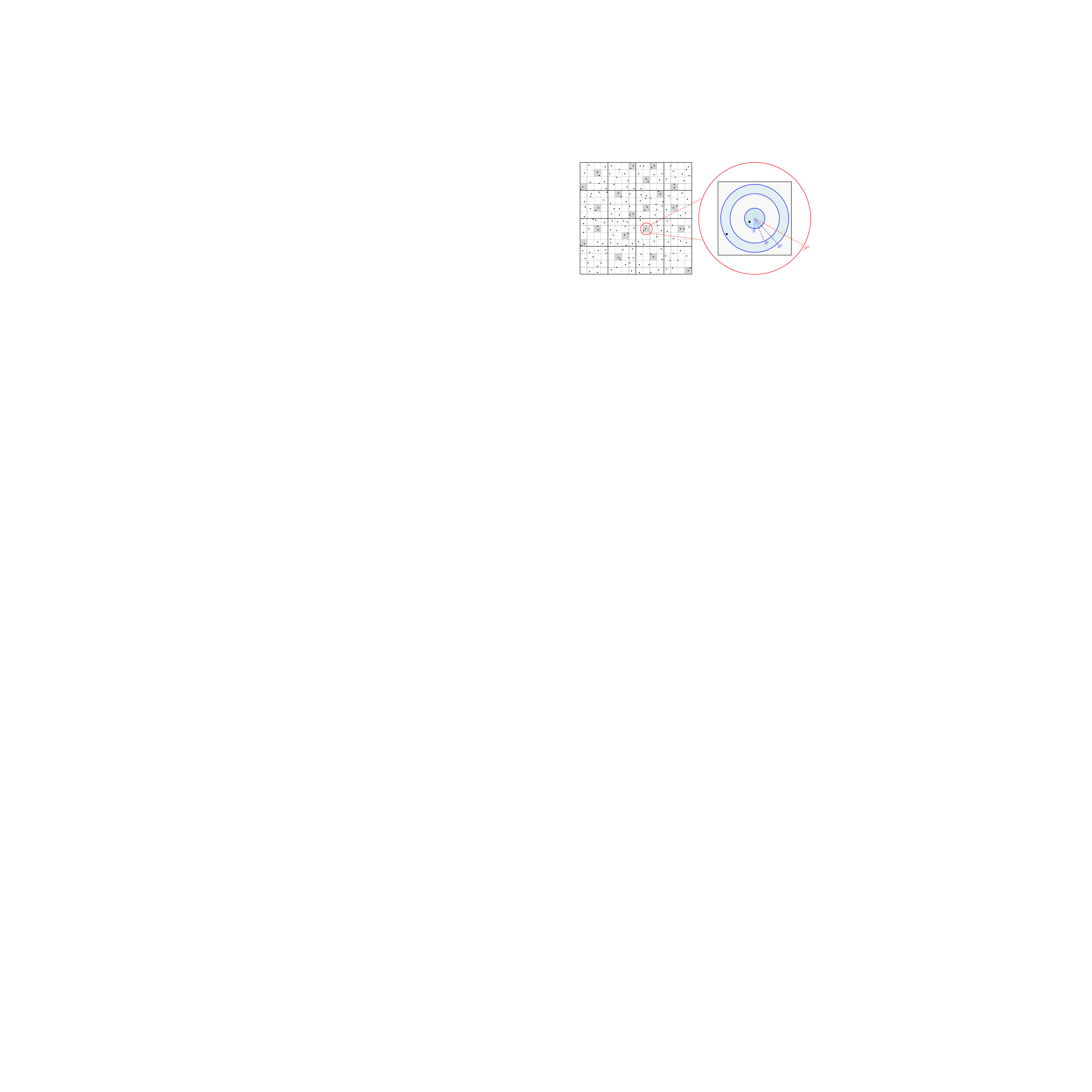}
    \caption{Example of a pattern cell decomposition for $\dimension=2$. The coarse grid corresponds to the trimming cells, whereas the fine grid corresponds to the pattern cells. Special pattern cells are shaded in, and one such cell $A$ is zoomed into. The definition of the pattern is given in~\eqref{eq:patterndef}.}
    \label{fig:Uaillustration}
\end{figure}
\begin{lemma}\label{lmm:concaverho2}
$\rho_{1}(\newrange)$ is concave in $\newrange$ for $0\le \newrange \le \MPar$.
\end{lemma}
 In order to show that $\rho_{2}(\newrange)$ is strongly concave, for any fixed cell $A$ and $\riderset_+$, we define 
\begin{align}
    \chi_A(\newrange,\riderset_+) \triangleq  \mathrm{Vol}
    \bigg(
        \bigcup_{\rider_j \in \riderset_+} B_\newrange(\rider_j)\cap A
    \bigg). \label{eq:gamma_A_ell}
\end{align}
Then, we introduce the following two technical lemmas. 
\begin{lemma}\label{lmm:gainonepattern}
For any configuration $\riderset_+$, if $A \in \activepatterncells$, then for some
$\alphagainpattern>0$
\begin{align*}
    \big(\chi_A (\base_1+\shift,\riderset_+ ) - \chi_A(\base_1,\riderset_+) \big)
    \,-\,
    \big(\chi_A(\base_2,\riderset_+) - \chi_A (\base_2-\shift,\riderset_+ )\big)
    \, \geq \, 
    \shift^{2}\frac{\alphagainpattern}{n}\, .
\end{align*}
\end{lemma}

\begin{lemma}\label{lmm:lowerboundpatterna}
For any $\varepsilon \leq \epsiloncondition$ and $n \ge \highdimN$, there exists $\alphaproba>0$ such that
\[
    \sum_{A \in \patterncells} \mathbb{P}_{G \sim \mathbb{G}(m,\servicevec)}\left[A \in \activepatterncells\right]
    \geq  \tau^2\alphaproba n.
\]
\end{lemma}
Thus, for any $\varepsilon\leq \epsiloncondition$ and any $n> \highdimN$, 
\begin{align*}
    \rho_{2}(\base_1+\tau)+\rho_{2}(\base_2-\tau)-\rho_{2}(\base_1)-\rho_{2}(\base_2)
    & 
    \stepa{\geq} \shift^{2}\frac{\alphagainpattern}{n}\sum_{A\in \patterncells}\mathbb{P}_{G \sim \mathbb{G}(m,\servicevec)}\left[A \in \activepatterncells\right] 
    \\
    &
    \stepb{\geq} 
    \shift^{4} \highdimalpha'  \,,
\end{align*}
where $\highdimalpha' \triangleq \alphaproba\alphagainpattern$.
Here, (a) holds by \prettyref{lmm:gainonepattern}, and (b) holds by~\prettyref{lmm:lowerboundpatterna}.  Together with \prettyref{eq:decomporho1rho2},~\prettyref{lmm:concaverho2} and the above inequality, the proof of~\prettyref{prop:strongconcavehighD} is complete.

\subsubsection{Special pattern cells}
This section describes the special pattern cells and their use in the proofs for Lemmas~\ref{lmm:concaverho2}-\ref{lmm:lowerboundpatterna}. Define
\[
    R \triangleq \frac{1}{10}\left(\frac{r_2}{n}\right)^{1/\dimension} \, ,
    \quad
    R' \triangleq \left(1-\frac{1}{10^\dimension}\right)^{1/\dimension}\left(\frac{r_2}{n}\right)^{1/\dimension} \, ,
    \quad
    R'' \triangleq \left(\frac{r_2}{n}\right)^{1/\dimension} \, ,
\]
as well as
\[
    R''' \triangleq \frac{\sqrt{\dimension}\,\patternside}{n^{1/\dimension}} + 2\left(\frac{\MPar }{n}\right)^{1/\dimension} \, .
\]
Here, $\patternside$ is defined in~\eqref{eq: def-w-k}. Note that the parameters $R,R',R'',R'''$ depend on $\MPar ,\dimension,\base_1, \base_2, \shift,n$. We omit the dependency in the subscript for readability. For a given cell $A$, we denote $x_A$ its center. For any realization of $\riderset_+$, a pattern cell $A$ is in $\activepatterncells$ if: 
\begin{align}\label{eq:patterndef}
    \left\{\, A \subseteq {\tcell}(x_A) \right\}
    \cap
    \left\{
        \begin{aligned}
            &\left|\mathcal \riderset_+ \cap B(x_A, R''')\right| \,=\, 2,\\[0.3em]
            &\left|\mathcal \riderset_+ \cap B(x_A, R)\right| \,=\, 1,\\[0.3em]
            &\left|\mathcal \riderset_+ \cap \big(B(x_A, R'') \setminus B(x_A, R')\big)\right| \,=\, 1
        \end{aligned}
    \right\}.
\end{align}
Such a cell $A$ is termed a special pattern cell. Property $A \in \activepatterncells$ is illustrated in \Cref{fig:Uaillustration}.  
We now collect some properties that special pattern cells satisfy -- these properties are used extensively in the proofs of Lemmas~\ref{lmm:concaverho2}-\ref{lmm:lowerboundpatterna}. The proof of~\prettyref{lem:constantdef} itself is postponed to~\Cref{sec-proof-of-properties}.

\begin{lemma} \label{lem:constantdef} 
For each \( n\in\mathbb N \), the parameters $R, R', R'', R'''$ satisfy:
\begin{enumerate}[label= \( \mathrm{(P\arabic*)} \), ref=P\arabic*]
    \item \label{prop:volumes}
    \[
        \mathrm{Vol}\Bigl( B\bigl(x_A,  R''\bigr) \setminus B\big(x_A,  R'\big) \Bigr)
        = \mathrm{Vol}\Bigl( B\bigl(x_A,  R\bigr) \Bigr)= \kappa_{\dimension} \frac{r_2}{ 10^\dimension n} \geq  \kappa_{\dimension} \frac{2\tau}{10^\dimension n} .
    \]
    where $\kappa_{\dimension}$ is the volume of a ball of unit radius in $\mathbb{R}^\dimension$.
    \item \label{prop:no_intersection} 
    Consider any two points $x_0,x_1$ with $x_0 \in B\left(x_A,R\right)$ and $x_1 \in B(x_A,R'')\setminus B(x_A,R')$. Define, for $\ell\ge 0$,
    \[
        I(\ell;x_0,x_1)\triangleq \mathrm{Vol}\Big( B\big(x_0,(\ell/n)^{1/\dimension}\big)\,\medcap\,B\big(x_1,(\ell/n)^{1/\dimension}\big)\Big).
    \]
    Then the map $\ell\mapsto I(\ell;x_0,x_1)$ is convex on $[\base_2-2\shift,\base_2]$, and there exists a constant $\alphagainpattern>0$ such that
    \begin{align} \label{eq:second-difference}
        I(\base_2)-2I(\base_2-\shift)+I(\base_2-2\shift)
        \ge \frac{\alphagainpattern}{n}\,\shift^{\,2} \, .
    \end{align}
    \item \label{prop:insidecell} 
    \[
        B\left(x_A,\,R'' + (\MPar /n)^{1/\dimension}\right)  \subseteq  A, \qquad A  \subseteq  B(x_A, R''').
    \]
    \item \label{prop:distance}
    Any point $\rider \in [0,1]^\dimension \setminus B(x_A,R''')$ satisfies\footnote{The distance from a point $d$ to a cell $A$ is defined as $\operatorname{dist}(d,A) = \inf_{a\in A} \| d-a\|_2$.}
    \[
        \operatorname{dist}(\rider,A) > 2\left(\frac{\MPar }{n}\right)^{1/\dimension}.
    \]
\end{enumerate}
\end{lemma}


\section{Analysis for the dual service range model} \label{sec-markov-chain}    

In this section, we prove our theorems on the dual service range model, i.e.~\Cref{thm: sizes-of-maximum-matching} and~\Cref{thm: extreme-cases}. Proofs for supporting lemmas are deferred to~\Cref{apx-deferred-proofs-for-dual-service}.

\subsection{Analysis for Markov embedding} \label{sec:dual_embedding}

We now establish a relationship between the matching problem and the Markov chain $(\potential(t))_{t\geq 0}$. At the core of this relationship is a generative process. We present this connection and introduce some terminology, and then use the generative process to prove~\Cref{thm: sizes-of-maximum-matching}.

\paragraph{Terminology.} 
For each $t \ge 0$, we track three \textit{active} agents: one demand node $\u(t)$, one flexible supply node $\vf(t)$, and one inflexible supply node $\vnf(t)$. At time $t$, the supply node (flexible or inflexible) is
\begin{itemize}
    \item[--] \textit{behind}, if it is too far left of the demand node: $\vnf(t) < \u(t) - \base/n$ or $\vf(t) < \u(t) - (\base + \extra)/n$.
    \item[--] \textit{in range}, if it is within matching range, i.e. $|\vnf(t) - \u(t)| \leq \base/n$ or $|\vf(t) - \u(t)| \leq (\base + \extra)/n$.
    \item[--] \textit{ahead}, if it is too far right of the demand node: $\vnf(t) > \u(t) + \base/n$ or $\vf(t) > \u(t) + (\base+\extra)/n$.
\end{itemize}
Further, we say that the flexible supply node has \emph{priority} over the inflexible supply node at time $t$, if $\vf(t) + \extra/n < \vnf$; otherwise the inflexible supply node has priority. Note that the priority rule compares the \textit{deadline}, i.e. the right-most point a supply node can cover, of the flexible and inflexible supply node, and prioritizes the supply node whose deadline is more imminent. We will say that the demand node is \textit{advanced} at time $t$, to mean $\u(t+1) = \u(t) + \mathrm{Exp}(1)/n$.  Similarly, a flexible (resp. inflexible) supply node is advanced at time $t$ if $\vf(t+1) = \vf(t) + \mathrm{Exp}(p)/n$ (resp. $\vnf(t+1) = \vnf(t) + \mathrm{Exp}(1-p)/n$). If an agent is not advanced, then it is kept, i.e. it remains the same at time $(t+1)$ as $t$.

\paragraph{The generative process.} Scan the unit interval $[0,1]$ from left-to-right, and generate a set of demand nodes, inflexible supply nodes, and flexible supply nodes as Poisson point processes on $[0,1]$ of rate $1$, $1-p$ and $p$ respectively by advancing them successively and matching via a greedy rule described below. Without loss of generality, $\u(0) = \vnf(0) = \vf(0) = 0$ may be assumed.

\begin{itemize}
    \item Case (A): The inflexible supply node is behind and it has priority over the flexible supply node. No match is made, and only the inflexible supply node is advanced.
    \item Case (B): The inflexible supply node is in-range and it has priority over the flexible supply node. The demand node is matched to the inflexible supply node, and both of them are advanced.
    \item Case (C): The flexible supply node is behind and it has priority over the inflexible supply node. No match is made, and only the flexible supply node is advanced.
    \item Case (D): Both supply nodes are ahead. The demand node is not matched, and only it is advanced.
    \item Case (E): The flexible supply node is in-range and it has priority over the inflexible supply node. The demand node is matched to the flexible supply node, and both of them are advanced.
\end{itemize}

\begin{figure}[t]
    \centering
    \subfigure[Case (A)]{
    \includegraphics[width = 0.3 \textwidth]{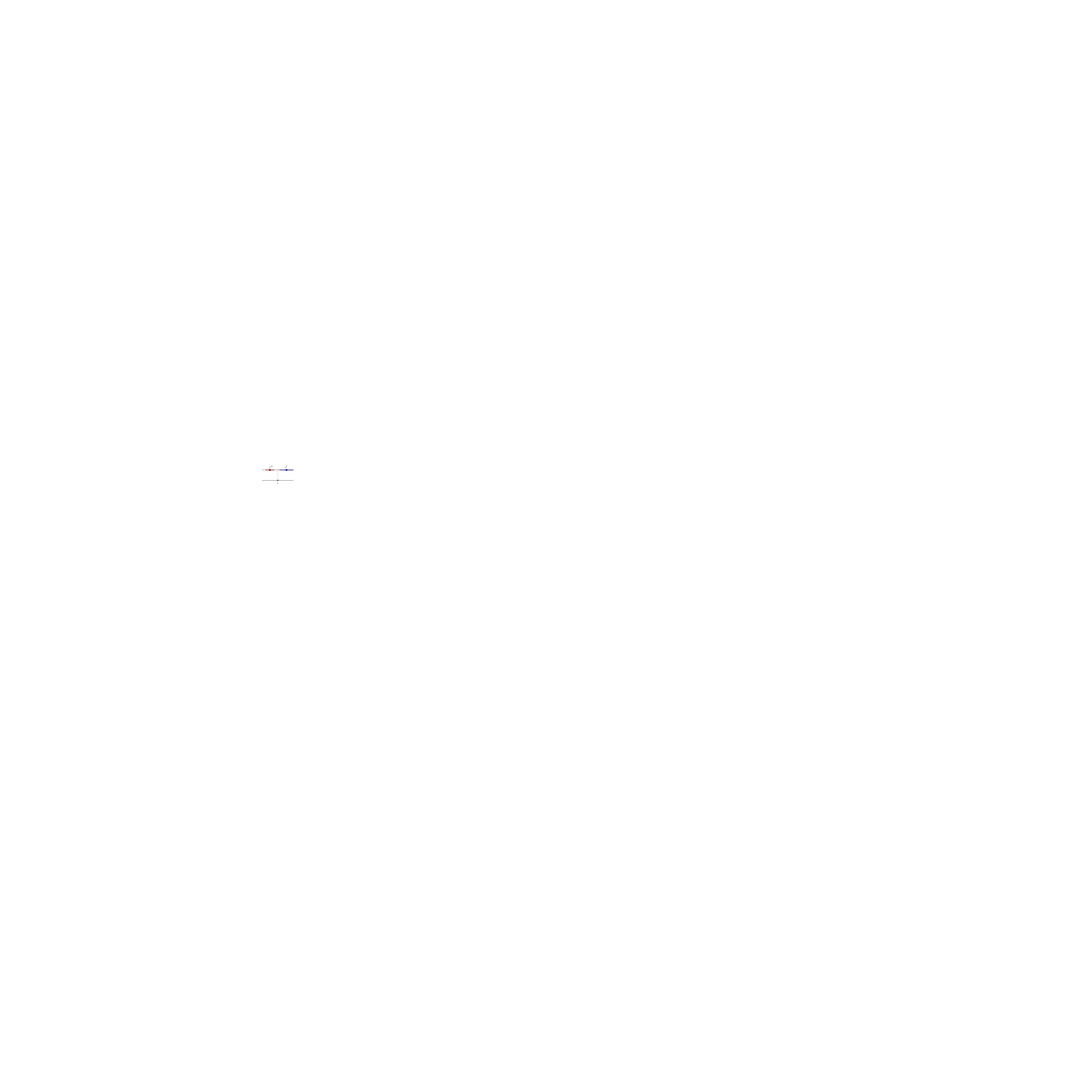}
    \label{fig-Generative-Configurations-A}
    } 
    \hfill
    \subfigure[Case (B)]{
    \includegraphics[width = 0.3 \textwidth]{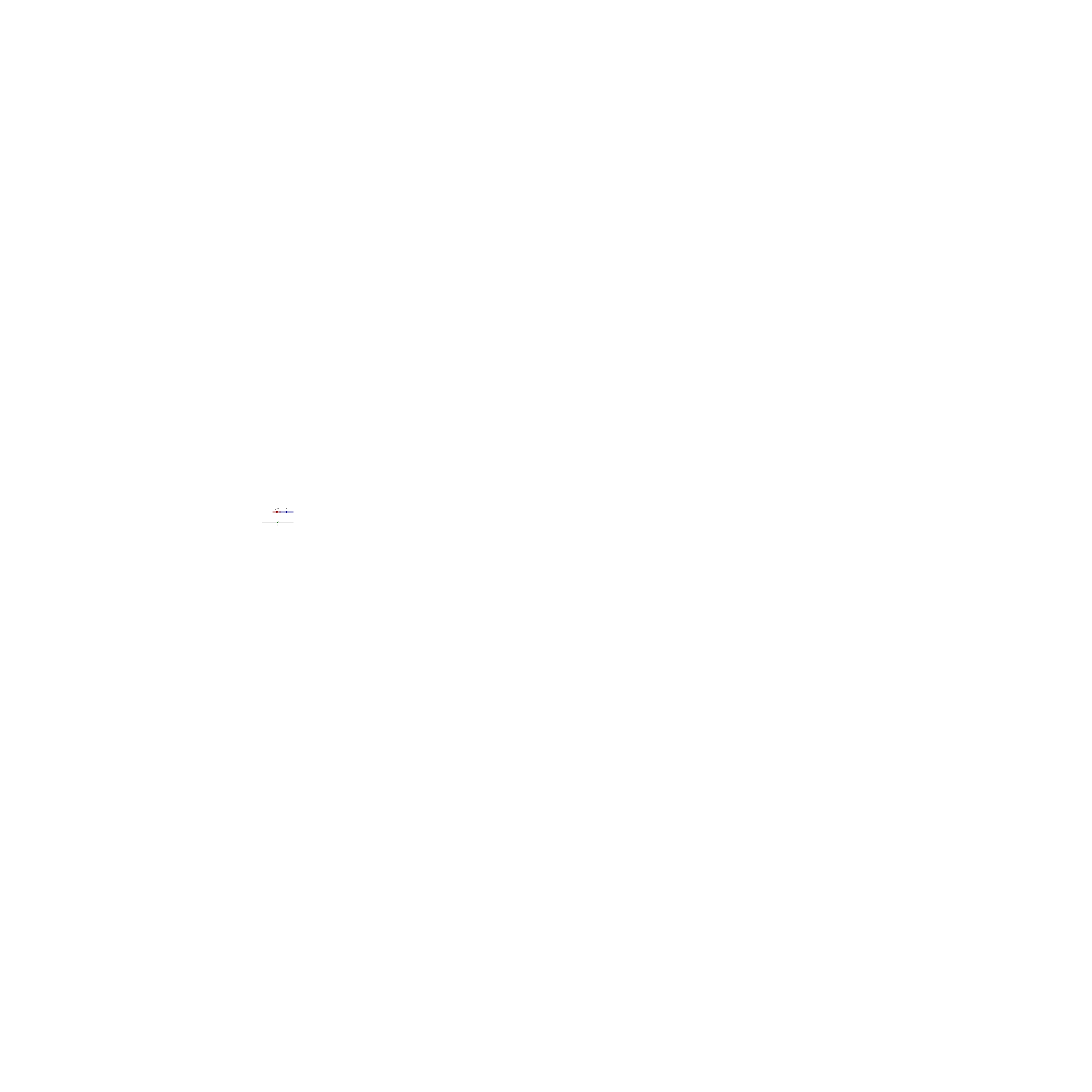}
    \label{fig-Generative-Configurations-B}
    } 
    \hfill
    \subfigure[Case (C)]{
    \includegraphics[width = 0.3 \textwidth]{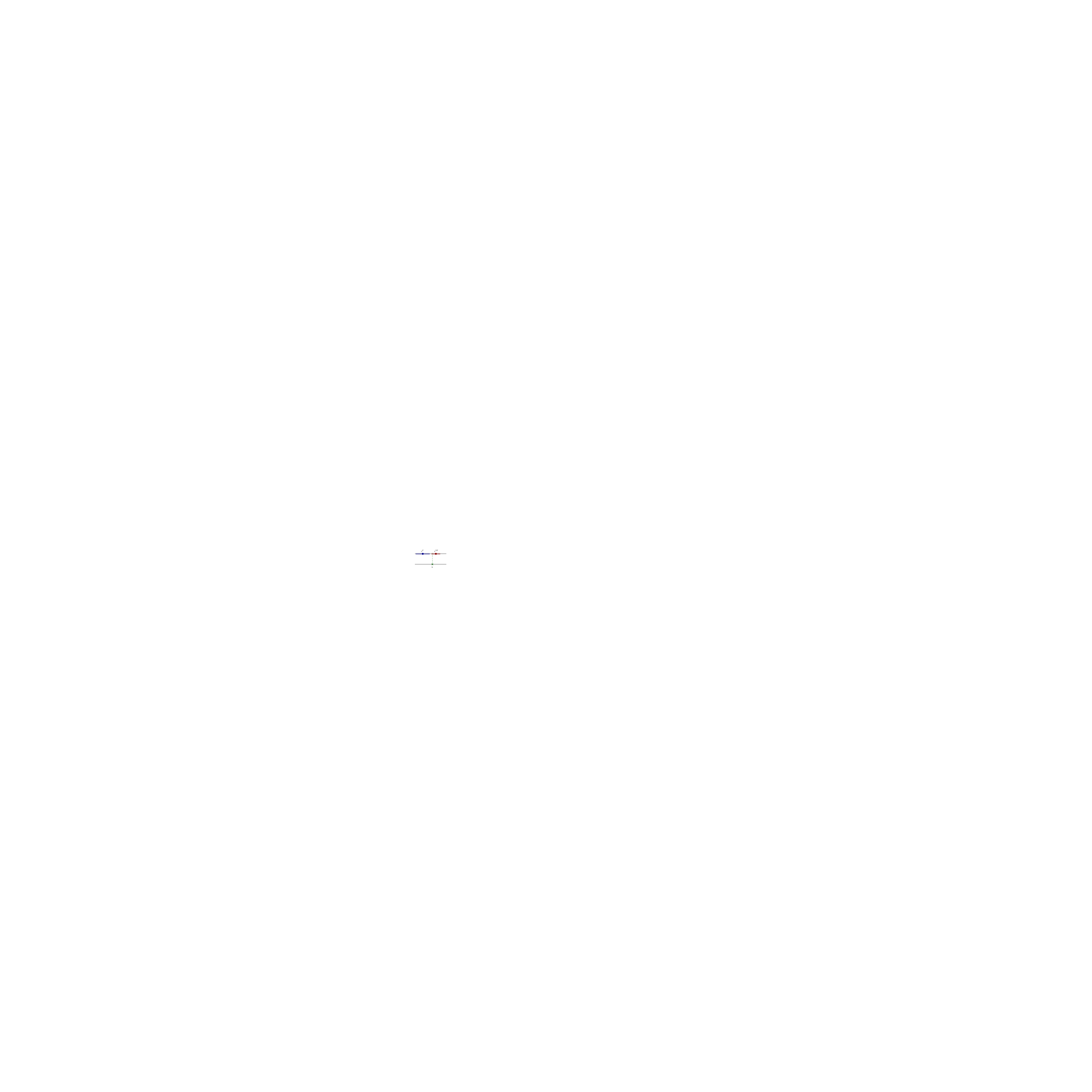}
    \label{fig-Generative-Configurations-C}
    } 
    \vfill
    \subfigure[Case (D)]{
    \includegraphics[width = 0.3 \textwidth]{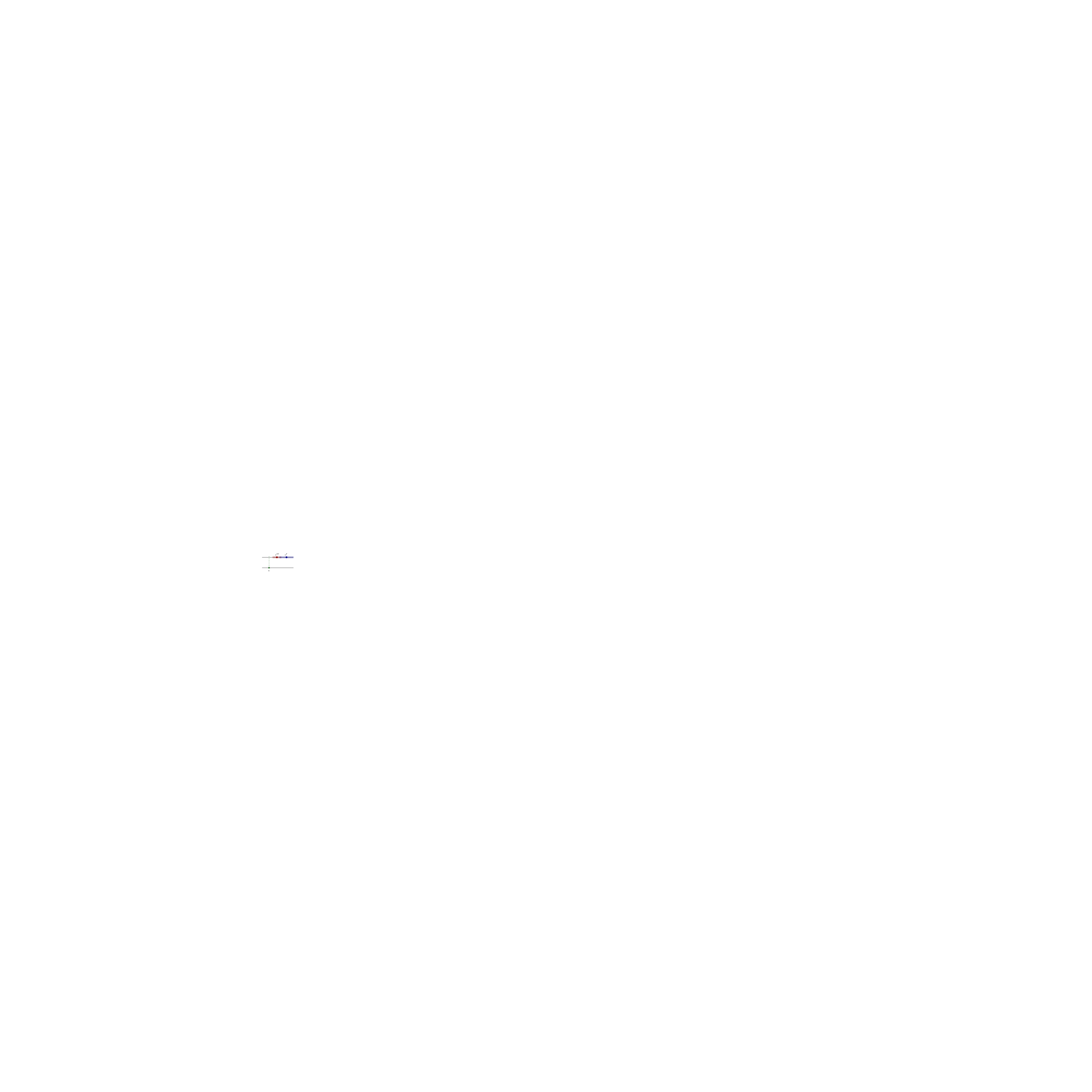}
    \label{fig-Generative-Configurations-D}
    } 
    \qquad
    \subfigure[Case (E)]{
    \includegraphics[width = 0.3 \textwidth]{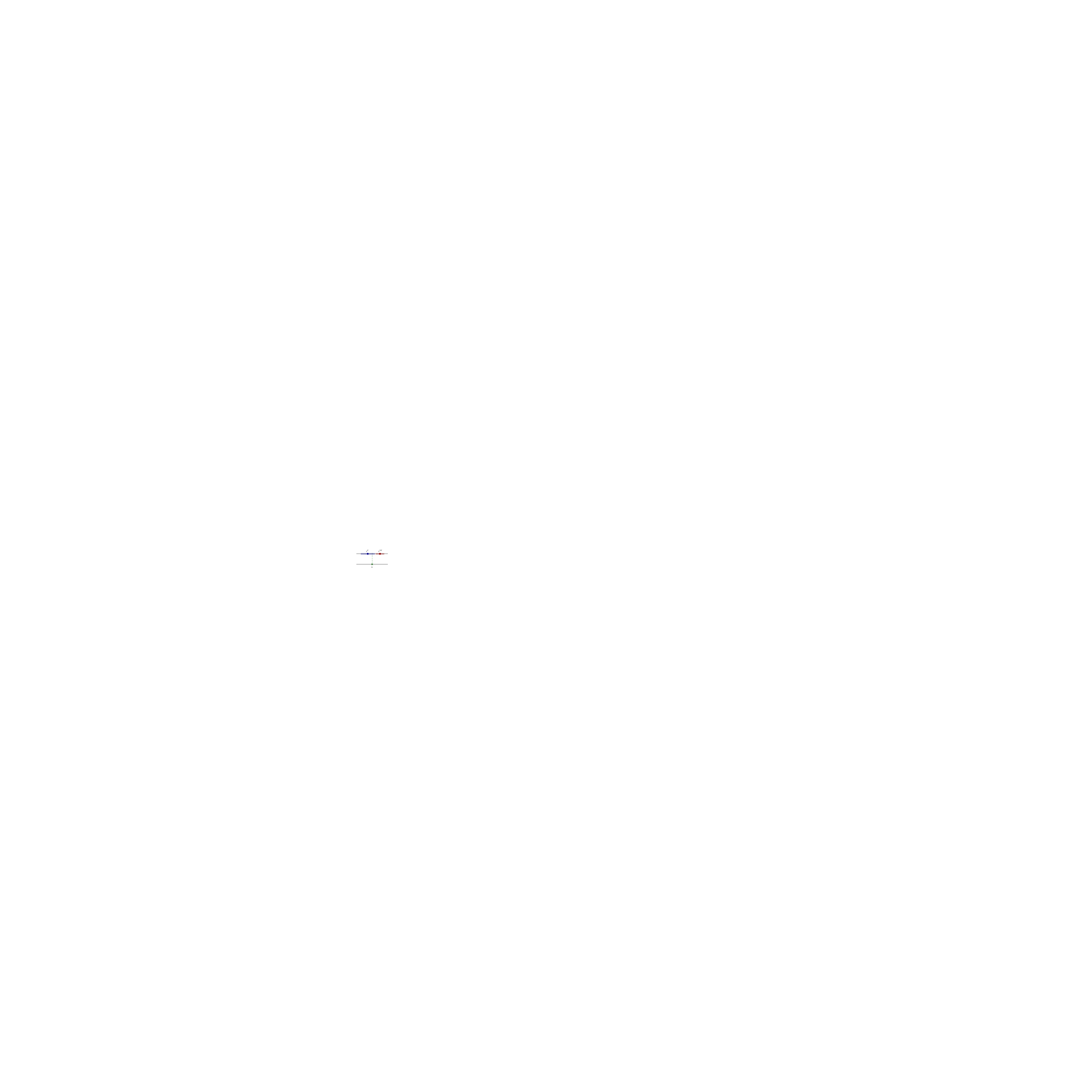}
    \label{fig-Generative-Configurations-E}
    } 
    \caption{Possible configurations between active demand node, and active flexible and inflexible supply nodes in various cases. Shaded regions around a supply node represent its service region.}
    \label{fig-Generative-Configurations}
\end{figure}

\begin{table}[t]
    \centering
    \renewcommand{\arraystretch}{1.2}
    \begin{tabular}{cccc}
    \toprule
    \textbf{Region} $R$ & $\nabla \potential(t)$ \textbf{in} $R$ & \textbf{Visualization} & \textbf{Interpretation} \\
    \midrule
        $\calA$  & $[0, \, -v_t]$  & Move down by $\mathrm{Exp}(1-p)$  & Discard inflexible supply node \\
        $\calB$  & $[w_t, \, w_t-v_t]$  & Move in shaded segment (Fig.~\ref{fig-Markov-Chain-Regions})  & Match inflexible supply node \\
        $\calC$  & $[-u_t, \, 0]$  & Move left by $\mathrm{Exp}(p)$  & Discard flexible supply node \\
        $\calD$  & $[w_t,\, w_t]$  & Move up-right by $\sqrt{2} \cdot \mathrm{Exp}(1)$  & Discard demand node \\
        $\calE$  & $[w_t - u_t, \,w_t]$  & Move in shaded segment (Fig.~\ref{fig-Markov-Chain-Regions})  & Match flexible supply node \\
    \bottomrule
    \end{tabular}
\caption{Markov dynamics of $\potential(t)$ in each region of $\reals^2$}
\label{tab: markov_dynamics}
\end{table}

The generative process terminates when all active agents exceed $1$. The five cases (A)--(E) are illustrated in Figure~\ref{fig-Generative-Configurations}. Instead of tracking the triplet $(\u(t), \vf(t), \vnf(t))$, it is convenient to track only the scaled lead times
\begin{equation}\label{eq:psi-def}
    \potentialF(t)  \,\triangleq\, n \bigl( \u(t) - \vf(t) \bigr) + (\base + \extra) \, , \qquad 
    \potentialNF(t) \,\triangleq \, n \bigl( \u(t) - \vnf(t) \bigr) + \base \, ,
\end{equation}
which measure at time $t$, the signed distance from the demand node to the left-most reachable point of the flexible and inflexible supply node respectively. Then $\potential(t) \triangleq \begin{bmatrix}
    \potentialF(t), & \potentialNF(t)
\end{bmatrix}$
is a Markov process on $\reals^2$, with the cases (A)-(E) above corresponding precisely to the five regions $\calA$ - $\calE$ defined below, and the one-step transitions $\nabla\potential(t)  \triangleq \potential(t+1) - \potential(t)$ defined as
\begin{align} \label{eq: markov-chain-dynamics}
  \nabla  \potential(t)  = 
    \begin{cases}
        [0, \, -v_t] & \text{ for }\potential(t) \in \calA
        \triangleq \{(x,y): y\ge 2\base, \, x\le y + 2\extra\}\\
        [w_t,\, w_t - v_t] & \text{ for } \potential(t) \in \calB 
        \triangleq\{(x,y): 0\le y \le 2\base, \, x\le y + 2\extra\} \\
         [-u_t,\, 0] &  \text{ for } \potential(t) \in \calC
        \triangleq \{(x,y): x\ge 2(\base+\extra), \, y\le x-2\extra\}\\ 
        [w_t,\, w_t] & \text{ for } \potential(t) \in \calD 
        \triangleq \{(x,y): x\le 0, \, y \le 0\} \\ 
        [w_t-u_t,\, w_t] & \text{ for } \potential(t) \in \calE
        \triangleq \{(x,y): 0 \le x \le 2(\base+\extra), \, y\le (x-2\extra)_+\} \\ 
    \end{cases} \,.
\end{align}

Note that the transition rules are independent of $t$; they are described for each region in Table~\ref{tab: markov_dynamics}. Note also the following about each step of the generative process:
\begin{itemize}
    \item steps in $\calB\cup \calE$ are \emph{matches} (of demand nodes to inflexible or flexible supply nodes);
    \item steps in $\calD$ skip demand nodes (such nodes remain unmatched);
    \item steps in $\calA \cup \calC$ discard supply nodes (no new demand node is generated).
\end{itemize}

\begin{proof}[Proof of~\texorpdfstring{\Cref{thm: sizes-of-maximum-matching}}{}]

Let $\widehat M_n$, $\widehat M_n^{\mathrm{F}}$, $\widehat M_n^{\mathrm{NF}}$ respectively be the number of matched supply nodes, flexible supply nodes and inflexible supply nodes, produced by the generative process above. 

\begin{proposition}\label{prop:gen-emb-eq}
Let $p \in (0,1)$. There is a constant $C=C(\base,\extra,p)<\infty$ such that
\begin{align}
    \left|\,\Expect[\widehat M_n] \,-\, \nu_n(\base,\extra,p)\,\right| \ &\leq \ C \sqrt{n \log n}  \label{eq: total-c}\\
    \left|\,\Expect[\widehat M_n^{\mathrm{F}}] \,-\, \nu_n^{\mathrm{F}}(\base,\extra,p)\,\right| \ &\leq \ C \sqrt{n \log n} \label{eq: flex-c} \\
    \left|\,\Expect[\widehat M_n^{\mathrm{NF}}] \,-\, \nu_n^{\mathrm{NF}}(\base,\extra,p)\,\right| \ &\leq \ C \sqrt{n \log n} \label{eq: non-flex-c}
\end{align}
\end{proposition}

Let $\widehat N^{\mathrm{D}}(T)$, $\widehat N^{\mathrm{F}}(T)$, $\widehat N^{\mathrm{NF}}(T)$ be the numbers of demand nodes, flexible supply nodes, and inflexible supply nodes that have been advanced up to step $T$, and let $\widehat M^{\mathrm{D}}(T)$, $\widehat M^{\mathrm{F}}(T)$, $\widehat M^{\mathrm{NF}}(T)$ denote the numbers of matches of each kind up to step $T$. Thus, the fractions of demand nodes, flexible supply nodes, and inflexible supply nodes that are matched by time $T$ are
\[
    \frac{\widehat M^{\mathrm{D}}(T)}{\widehat N^{\mathrm{D}}(T)}, \quad
    \frac{\widehat M^{\mathrm{F}}(T)}{\widehat N^{\mathrm{F}}(T)}, \quad \mathrm{and} \quad 
    \frac{\widehat M^{\mathrm{NF}}(T)}{\widehat N^{\mathrm{NF}}(T)}.
\]
Let $\tau_n$ denote the (random) termination time of the generative process on $[0,1]$ i.e. $\tau_n$ is the first step at which all active agents have been advanced past $1$, and hence $\widehat N^{\mathrm{D}}(\tau_n)$, $\widehat N^{\mathrm{F}}(\tau_n)$, $\widehat N^{\mathrm{NF}}(\tau_n)$ are exactly the total numbers of demand, flexible supply, and inflexible supply points in $[0,1]$ generated by the three independent PPPs of rates $n$, $pn$, and $(1-p)n$. 
Note that
\[
    \widehat M_n = \widehat M^{\mathrm{D}}(\tau_n), \qquad
    \widehat M_n^{\mathrm{F}} = \widehat M^{\mathrm{F}}(\tau_n), \qquad
    \widehat M_n^{\mathrm{NF}} = \widehat M^{\mathrm{NF}}(\tau_n),
\]
since every match uses exactly one demand node and exactly one supply node. Next, we show that the Markov chain $\potential(t)$ follows a law of large numbers (LLN).

\begin{proposition}[Markov chain LLN]\label{prop:psi-ergodic}
The Markov chain $\potential(t)$ admits a unique stationary distribution $\pi$ on $\mathbb{R}^2$, such that for any bounded measurable function $g$,
\[
    \frac1T\sum_{t=0}^{T-1}g \bigl(\potential(t)\bigr)\ \xrightarrow[T\to\infty]{\mathrm{a.s.}}\ \int g \,\diff\pi \, .
\]
\end{proposition}

By construction of the generative process, we have that for each deterministic $T$,
\begin{align*}
    \frac{\widehat N^{\mathrm{D}}(T)}{T} &= \frac1T\sum_{t=0}^{T-1}\mathbf{1}\{\potential(t)\in \calB\cup\calD\cup\calE\},
    &
    \frac{\widehat M^{\mathrm{D}}(T)}{T} &= \frac1T\sum_{t=0}^{T-1}\mathbf{1}\{\potential(t)\in \calB\cup\calE\},\\
    \frac{\widehat N^{\mathrm{F}}(T)}{T} &= \frac1T\sum_{t=0}^{T-1}\mathbf{1}\{\potential(t)\in \calC\cup\calE\},
    &
    \frac{\widehat M^{\mathrm{F}}(T)}{T} &= \frac1T\sum_{t=0}^{T-1}\mathbf{1}\{\potential(t)\in \calE\},\\
    \frac{\widehat N^{\mathrm{NF}}(T)}{T} &= \frac1T\sum_{t=0}^{T-1}\mathbf{1}\{\potential(t)\in \calA\cup\calB\},
    &
    \frac{\widehat M^{\mathrm{NF}}(T)}{T} &= \frac1T\sum_{t=0}^{T-1}\mathbf{1}\{\potential(t)\in \calB\}.
\end{align*}

Applying \Cref{prop:psi-ergodic} to these indicator functions yields, almost surely as $T\to\infty$,
\begin{align*}
    \frac{\widehat N^{\mathrm{D}}(T)}{T}\toas F_{\calB}+F_{\calD}+F_{\calE}, \qquad
    &\frac{\widehat M^{\mathrm{D}}(T)}{T}\toas F_{\calB}+F_{\calE},\\
    \frac{\widehat N^{\mathrm{F}}(T)}{T}\toas F_{\calC}+F_{\calE}, \qquad
    &\frac{\widehat M^{\mathrm{F}}(T)}{T}\toas F_{\calE},\\
    \frac{\widehat N^{\mathrm{NF}}(T)}{T}\toas F_{\calA}+F_{\calB}, \qquad
    &\frac{\widehat M^{\mathrm{NF}}(T)}{T}\toas F_{\calB}.
\end{align*}
Since $\tau_n\to\infty$ almost surely as $n\to\infty$ (indeed, $\tau_n\ge \max\{\widehat N^{\mathrm{D}}(\tau_n),\widehat N^{\mathrm{F}}(\tau_n),\widehat N^{\mathrm{NF}}(\tau_n)\}$ and each of these PPP counts diverges with $n$), the same almost sure limits hold along the random subsequence $T=\tau_n$. In particular,
\begin{align}
    \frac{\widehat M^{\mathrm{F}}(\tau_n)}{\widehat N^{\mathrm{F}}(\tau_n)}
    \toas \frac{F_{\calE}}{F_{\calC}+F_{\calE}},
    \qquad
    \frac{\widehat M^{\mathrm{NF}}(\tau_n)}{\widehat N^{\mathrm{NF}}(\tau_n)}
    \toas \frac{F_{\calB}}{F_{\calA}+F_{\calB}},
    \qquad
    \frac{\widehat M^{\mathrm{D}}(\tau_n)}{\widehat N^{\mathrm{D}}(\tau_n)}
    \toas \frac{F_{\calB}+F_{\calE}}{F_{\calB}+F_{\calD}+F_{\calE}}.
    \label{eq:ratio-limits-at-tau}
\end{align}
At time $\tau_n$, the generative process has revealed all points of the three independent PPPs on $[0,1]$ with rates $n$, $pn$, and $(1-p)n$. Hence,
\[
    \frac{\widehat N^{\mathrm{D}}(\tau_n)}{n}\toas 1,\qquad
    \frac{\widehat N^{\mathrm{F}}(\tau_n)}{n}\toas p,\qquad
    \frac{\widehat N^{\mathrm{NF}}(\tau_n)}{n}\toas 1-p,
\]
by the strong law of large numbers for sums of exponential inter-arrival times.
Combining these limits with \eqref{eq:ratio-limits-at-tau} gives the almost sure limits
\[
    \frac{\widehat M_n^{\mathrm{F}}}{n}
    = \frac{\widehat M^{\mathrm{F}}(\tau_n)}{\widehat N^{\mathrm{F}}(\tau_n)}\cdot \frac{\widehat N^{\mathrm{F}}(\tau_n)}{n}
    \toas \left(\frac{F_{\calE}}{F_{\calC}+F_{\calE}}\right)p,
    \qquad
    \frac{\widehat M_n^{\mathrm{NF}}}{n}
    \toas \left(\frac{F_{\calB}}{F_{\calA}+F_{\calB}}\right)(1-p),
\]
and
\[
    \frac{\widehat M_n}{n}
    = \frac{\widehat M^{\mathrm{D}}(\tau_n)}{\widehat N^{\mathrm{D}}(\tau_n)}\cdot \frac{\widehat N^{\mathrm{D}}(\tau_n)}{n}
    \toas \frac{F_{\calB}+F_{\calE}}{F_{\calB}+F_{\calD}+F_{\calE}}.
\]
All three ratios on the left are bounded in $[0,1]$, hence by dominated convergence,
\begin{align}
    \frac{\Expect[\widehat M_n^{\mathrm{F}}]}{n}\to \left(\frac{F_{\calE}}{F_{\calC}+F_{\calE}}\right)p, \quad
    \frac{\Expect[\widehat M_n^{\mathrm{NF}}]}{n}\to \left(\frac{F_{\calB}}{F_{\calA}+F_{\calB}}\right)(1-p), \quad
    \frac{\Expect[\widehat M_n]}{n}\to \frac{F_{\calB}+F_{\calE}}{F_{\calB}+F_{\calD}+F_{\calE}},
    \label{eq:expected-greedy-limits}
\end{align}
Combining~\eqref{eq:expected-greedy-limits} with~\Cref{prop:gen-emb-eq} gives~\eqref{eq: match-flex} and~\eqref{eq: match-non-flex}. To prove~\eqref{eq: match-total}, define the net imbalance of advanced supply nodes versus demand nodes up to step $T$ by
\[
    \Delta(T) \, \triangleq \, \big(\widehat N^{\mathrm{F}}(T)+\widehat N^{\mathrm{NF}}(T)\big) - \widehat N^{\mathrm{D}}(T).
\]
At a single step, $\Delta$ increases by $1$ in $\calA\cup\calC$, decreases by $1$ in $\calD$, and is unchanged in $\calB\cup\calE$. Hence
\[
    \frac{\Delta(T)}{T}
    =
    \frac1T\sum_{t=0}^{T-1}\Big(\mathbf{1}\{\potential(t)\in \calA\cup\calC\}-\mathbf{1}\{\potential(t)\in \calD\}\Big)
    \ \toas\ (F_{\calA}+F_{\calC})-F_{\calD} \, .
\]
Evaluating at $T=\tau_n$, we also have the identity
\[
    \Delta(\tau_n)
    =
    \big(\widehat N^{\mathrm{F}}(\tau_n)+\widehat N^{\mathrm{NF}}(\tau_n)\big) - \widehat N^{\mathrm{D}}(\tau_n).
\]
Since both $\widehat N^{\mathrm{D}}(\tau_n)/n\toas 1$ and $\big(\widehat N^{\mathrm{F}}(\tau_n)+\widehat N^{\mathrm{NF}}(\tau_n)\big)/n\toas 1$, it follows that $\Delta(\tau_n)/n\toas 0$. Moreover, $\tau_n=\Theta(n)$ almost surely because
\[
    \max \left\{\widehat N^{\mathrm{D}}(\tau_n),\widehat N^{\mathrm{F}}(\tau_n),\widehat N^{\mathrm{NF}}(\tau_n) \right\}
    \ \le\ \tau_n\ \le\ \widehat N^{\mathrm{D}}(\tau_n)+\widehat N^{\mathrm{F}}(\tau_n)+\widehat N^{\mathrm{NF}}(\tau_n),
\]
and each PPP count is $\Theta(n)$ almost surely. Therefore $\Delta(\tau_n)/\tau_n\toas 0$, and so
\(
    F_{\calA}+F_{\calC} \,=\, F_{\calD}.
\)
Using $F_{\calA}+F_{\calB}+F_{\calC}+F_{\calD}+F_{\calE}=1$, we obtain
\[
    \frac{F_{\calB}+F_{\calE}}{F_{\calB}+F_{\calD}+F_{\calE}}
    =
    \frac{1-(F_{\calA}+F_{\calC}+F_{\calD})}{1-(F_{\calA}+F_{\calC})}
    =
    \frac{1-2F_{\calD}}{1-F_{\calD}}.
\]
Substituting into \eqref{eq:expected-greedy-limits} gives $\Expect[\widehat M_n]/n \to (1-2F_{\calD})/(1-F_{\calD})$, and combining with~\Cref{prop:gen-emb-eq} yields \eqref{eq: match-total}. This concludes the proof.
\end{proof}

\subsubsection{Proof of~\texorpdfstring{\Cref{prop:gen-emb-eq}}{}}

The generative process produces three independent Poisson point processes (PPPs) on $[0,1]$:
demands $\Phi^{\mathrm{D}}$ with intensity $n$, flexible supplies $\Phi^{\mathrm{F}}$ with intensity $pn$, and inflexible supplies $\Phi^{\mathrm{NF}}$ with intensity $(1-p)n$. Let $\widehat G_n$ be the interval bigraph induced by these point sets and the rule $|\driver_i-\rider_j|\le \service_i/n$.
The graph $\widehat G_n$ is an \emph{interval graph}, so a maximum matching can be found via a greedy algorithm~\cite{glover1967convex, lipskipreparata1981}. Specifically, Algorithm~\ref{alg: greedy} (which processes demand nodes left-to-right and matches to the unmatched neighbor with the smallest deadline $\driver_i+\service_i/n$) is optimal on $\widehat G_n$. 
\begin{lemma} \label{lem: greedy-is-optimal} Let $\service_i$ be a vector of service ranges. Let $G$ be a bipartite geometric graph on the vertex set $\drivervec \cup \ridervec$ (with $|\drivervec| = |\ridervec| = n$), in which $\driver_i$ has an edge with $\rider_j$ if and only if $|\driver_i - \rider_j| \leq \service_i/n$. The output $\calM$ of Algorithm~\ref{alg: greedy} on $G$ is a maximum matching.
\end{lemma}
Let $\calM(\widehat G_n)$ be the output of Algorithm~\ref{alg: greedy}, and let $\widehat \calM(\widehat G_n)$ denote the matching realized by the generative process on $\widehat G_n$. Our next result shows that the two matchings agree.

\begin{algorithm}[t]                 
  \caption{Greedy Algorithm}
  \label{alg: greedy}

  \begin{algorithmic}[1]
    \Require Bipartite graph $G$ on vertex set $\drivervec \cup \ridervec$

    \State $\drivervec' \gets \drivervec$ \Comment{unmatched supply nodes}
    \State $\ridervec' \gets \ridervec$         \Comment{unmatched demand nodes}
    \State $\calM' \gets \emptyset$ \Comment{matching}

    \For{$t = 1,\cdots,|\ridervec|$}
        \If{$\rider_{(t)}$ has a neighbor in $\drivervec'$}
            \State $t^* \gets \argmin \{\driver_i + \service_i/n : \driver_i \text{ is a neighbor of $\rider_{(t)}$ in $\drivervec'$} \}$ \Comment{greedy matching rule}
            \State $\drivervec' \gets \drivervec' - \driver_{t^*}$
            \State $\ridervec' \gets \ridervec' - \rider_{(t)}$
            \State $\calM' \gets \calM' + (\rider_{(t)}, \driver_{t^*}) $
        \EndIf
    \EndFor

    \State \Return $\calM'$
  \end{algorithmic}
\end{algorithm}

\begin{lemma} \label{lem:agreement}
For any realization of $\widehat G_n$, we have $\calM(\widehat G_n)=\widehat \calM(\widehat G_n)$.
\end{lemma}

Consequently, $|\widehat \calM(\widehat G_n)|=|\calM(\widehat G_n)|$, and since the greedy algorithm is optimal, this equals the maximum matching size in $\widehat G_n$. Let $\widehat M_n$, $\widehat M_n^{\mathrm{F}}$ and $\widehat M_n^{\mathrm{NF}}$ be, respectively, the total number of matches, the number of matches using flexible supplies, and the number of matches using inflexible supplies, produced by the generative process. Then pathwise
\[
    \widehat M_n = |\calM(\widehat G_n)|,
    ~~~~~~
    \widehat M_n^{\mathrm{F}} = |\{e\in \calM(\widehat G_n): e \text{ uses a flexible supply}\}|,
    ~~~~~~
    \widehat M_n^{\mathrm{NF}} = |\calM(\widehat G_n)| - \widehat M_n^{\mathrm{F}}.
\]
Define the PPP counts
\[
    N^{\mathrm{D}} \triangleq |\Phi^{\mathrm{D}}|,   
    ~~~~~~
    N^{\mathrm{F}} \triangleq |\Phi^{\mathrm{F}}|,   
    ~~~~~~
    N^{\mathrm{NF}} \triangleq |\Phi^{\mathrm{NF}}|, 
    ~~~~~~
    N^{\mathrm{S}} \triangleq N^{\mathrm{F}}+N^{\mathrm{NF}}.
\]
We transform $\widehat G_n$ in two stages to obtain a graph with the law of the true model $G_n$ (i.e., $n$ demands, $n$ supplies, and $ pn $ flexible supplies chosen among the $n$).

\medskip
\noindent \emph{Stage 1 (exactly $n$ demands and $n$ supplies).} If $N^{\mathrm{D}}>n$ (resp.\ $<n$), delete $N^{\mathrm{D}}-n$ (resp.\ add $n-N^{\mathrm{D}}$) i.i.d.\ $\mathrm{Unif}[0,1]$ demand points; do the analogous operation on the supply side to end with exactly $n$ supply points. Because PPP points are i.i.d.\ uniform within the window, this operation changes the instance by at most $|N^{\mathrm{D}}-n|+|N^{\mathrm{S}}-n|$ vertex insertions/deletions. Since adding/removing a single vertex changes the maximum matching size by at most one, we obtain
\begin{align}\label{eq:stage1-total}
    \big|\,|\calM(\widehat G_n)| - |\calM(G_n')|\,\big|
    \, \leq \, |N^{\mathrm{D}}-n|+|N^{\mathrm{S}}-n|, 
\end{align}
where $G_n'$ denotes the graph after Stage~$1$. Changing the type (flexible/inflexible) of a supply node changes the number of matched nodes of the same type by at most $1$. Therefore, 
\begin{align}\label{eq:stage1-type}
    \big|\,\widehat M_n^{\mathrm{F}} - M_{\mathrm{F}}(G_n')\,\big| \le |N^{\mathrm{D}}-n|+|N^{\mathrm{S}}-n|,
    ~~~~~~
    \big|\,\widehat M_n^{\mathrm{NF}} - M_{\mathrm{NF}}(G_n')\,\big| \le |N^{\mathrm{D}}-n|+|N^{\mathrm{S}}-n|,
\end{align}
where $M_{\mathrm{F}}(\cdot)$ and $M_{\mathrm{NF}}(\cdot)$ denote the numbers of matched flexible and inflexible supplies in a maximum matching of the argument graph.

\medskip
\noindent \emph{Stage 2 (enforce $ p n $ flexible supplies).} Let $K_n$ be the number of flexible supplies among the $n$ supply points of $G_n'$. We change the \emph{types} of exactly $|K_n-  p n |$ supplies (arbitrarily) so that the final flexible count is $p n$. Changing the type of one supply (i.e., changing only the incident edges of that supply by switching its radius between $\base$ and $\base+\extra$) alters the size of a maximum matching by at most $1$, and similarly changes $M_{\mathrm{F}}$ and $M_{\mathrm{NF}}$ by at most $1$. Therefore,
\begin{align}\label{eq:stage2-bounds-parent}
    \big|\,|\calM(G_n')| - |\calM(G_n'')|\,\big| \leq |K_n- p n |,
\end{align}
and 
\begin{align} \label{eq:stage2-bounds-children}
    \big|\,M_{\mathrm{F}}(G_n') - M_{\mathrm{F}}(G_n'')\,\big| \leq |K_n - p n | \, ,
    ~~~~
    \big|\,M_{\mathrm{NF}}(G_n') - M_{\mathrm{NF}}(G_n'')\,\big| \leq |K_n - p n| \, .
\end{align}
Here $G_n''$ is the graph after Stage~2. Because the $n$ supply locations are i.i.d.\ $\mathrm{Unif}[0,1]$ and types are now obtained by choosing exactly $ p n $ of the $n$ supplies uniformly to be flexible, $G_n''$ has the same law as the original graph $G_n$ with parameters $(\base,\extra,p)$.

Standard Poisson tail bounds give that with probability at least $1-4/n$,
\begin{align}\label{eq:counts-conc-new}
    |N^{\mathrm{D}}-n|\leq 2\sqrt{n\log n}
    ~~~~~~ \text{and} ~~~~~~
    |N^{\mathrm{S}}-n|\leq 2\sqrt{n\log n} \, .
\end{align}
Further, conditional on the $n$ supply locations after Stage~1, let the types be i.i.d.\ $\mathrm{Bernoulli}(p)$ (this can be achieved by adopting the natural marking from the union of the independent PPPs and marking any added supplies independently). Then $K_n\sim\mathrm{Binomial}(n,p)$, and by a Chernoff bound, with probability at least $1-2/n$,
\begin{align}\label{eq:binom-conc}
    |K_n - pn| \,\leq\, 2\sqrt{n\log n} \, .
\end{align}
On the intersection of \eqref{eq:counts-conc-new} and \eqref{eq:binom-conc}, combining \eqref{eq:stage1-total}, \eqref{eq:stage2-bounds-parent} and~\eqref{eq:stage2-bounds-children} yields
\[
    \big|\,|\calM(\widehat G_n)| - |\calM(G_n'')|\,\big|
    \,\leq\, |N^{\mathrm{D}}-n|+|N^{\mathrm{S}}-n|+|K_n- p n |
    \,\leq\, C_1\sqrt{n\log n},
\]
for a constant $C_1=C_1(\base,\extra,p)$. Outside this event, the difference is trivially at most $n$. Taking expectations and using that the bad event has probability $O(1/n)$, we obtain
\begin{align}\label{eq:depo-exp-new-total}
    \big|\,\Expect[|\calM(\widehat G_n)|]-\Expect[|\calM(G_n''|]\,\big|
    \, \leq \, C\sqrt{n\log n}.
\end{align}
The same argument applied to \eqref{eq:stage1-type},~\eqref{eq:stage2-bounds-parent} and~\eqref{eq:stage2-bounds-children} gives
\begin{align}\label{eq:depo-exp-new-type}
    \big|\,\Expect[\widehat M_n^{\mathrm{F}}]-\Expect[M_{\mathrm{F}}(G_n'')]\,\big|
    \ \leq\ C\sqrt{n\log n},
    \qquad
    \big|\,\Expect[\widehat M_n^{\mathrm{NF}}]-\Expect[M_{\mathrm{NF}}(G_n'')]\,\big|
    \ \leq\ C\sqrt{n\log n}.
\end{align}

To conclude, note that by definition, $\nu_n(\base,\extra,p) \triangleq \Expect[|\calM(G_n'')|]$, $\nu_n^{\mathrm{F}}(\base,\extra,p) \triangleq \Expect[M_{\mathrm{F}}(G_n'')]$, and $\nu_n^{\mathrm{NF}}(\base,\extra,p) \triangleq \Expect[M_{\mathrm{NF}}(G_n'')]$. Using $|\calM(\widehat G_n)|=\widehat M_n$ established above and combining with \eqref{eq:depo-exp-new-total}--\eqref{eq:depo-exp-new-type}, we obtain the desired bounds:
\begin{align*}
\left|\,\Expect[\widehat M_n] - \nu_n(\base,\extra,p)\,\right|
\ \leq \ C\sqrt{n\log n}
\end{align*}
and
\begin{align*}
\left|\,\Expect[\widehat M_n^{\mathrm{F}}] - \nu_n^{\mathrm{F}}(\base,\extra,p)\,\right|
\ \leq \ C\sqrt{n\log n},\qquad
\left|\,\Expect[\widehat M_n^{\mathrm{NF}}] - \nu_n^{\mathrm{NF}}(\base,\extra,p)\,\right|
\ \leq\ C\sqrt{n\log n}.
\end{align*}
This concludes the proof of~\Cref{prop:gen-emb-eq}.

\subsubsection{Proof of~\texorpdfstring{\Cref{prop:psi-ergodic}}{}}
We show that the Markov chain $\potential(t)$ is positive Harris recurrent; see~\Cref{apx-preliminaries} for a review of terminology, and tools used in the proof. First, we produce a petite set, and verify $\varphi$-irreducibility of $\potential(t)$.
\begin{lemma} \label{lem:small}
Fix $\rs>0$ and $d \in \left( 0, \min\{1,\base/2,\extra \}\right)$. Consider the rectangles 
\[
    R \triangleq [d, \, 2d] \times [3d, \,  4d], ~~~~ 
    K_\rs \triangleq \left[ -\rs, \, 2(\base+\extra)\right] \times \left[-\rs, \, 2\base \right].
\]
Let $\varphi = \mathrm{Vol}\vert_R$, so that for any Borel set $A \subseteq \reals^2$, $\varphi(A) = \mathrm{Vol}(A\cap R)$. Then,
\begin{itemize}
    \item[$\mathrm{(i)}$] $K_\rs$ is a petite set for $\potential(t)$ with respect to $\varphi$.
    \item[$\mathrm{(ii)}$] $\potential(t)$ is $\varphi$-irreducible. 
\end{itemize}
\end{lemma}
Next, we show that the Foster-Lyapunov condition holds, i.e. we construct a radially unbounded function $V$ satisfying the drift inequality toward the petite set. 
\begin{lemma}[Foster–Lyapunov]\label{lem:drift}
Consider the function
\[
    V(x,y)\ \triangleq\ 1+ \alpha\,x_+ + \beta\,y_++ \delta\,(x_-^2+y_-^2) \, , ~~~~~~ \alpha,\beta,\delta>0.
\]
There exist parameters $\alpha,\beta,\delta$ and $\rs>0$ such that, with
$K_\rs\triangleq[-\rs,\,2(\base+\extra)]\times[-\rs,\,2\base]$,
\begin{align}\label{eq:geom-drift}
    \mathbb E\left[V\big(\potential(t+1)\big)-V\big(\potential(t)\big)\,\middle|\,\potential(t)=(x,y)\right]
    \leq -\eta\,\mathbf 1_{K_\rs^c}(x,y) +  b_0\,\mathbf 1_{K_\rs}(x,y)
    ~~~~ \forall(x,y)\in\mathbb R^2 \, ,
\end{align}
for some positive finite constants $\eta$ and $b_0$.
\end{lemma}

Lemmas~\ref{lem:small} and~\ref{lem:drift} together imply that the chain is positive Harris recurrent (see Theorem~\ref{thm:PHR} in~\Cref{apx-preliminaries}). Since $\potential(t)$ is positive Harris recurrent, the conditions of~\Cref{thm: convergence} hold, i.e. for any bounded measurable function $g$, we have that
\[ 
    \frac 1 T \sum_{i=0}^{T-1} g(\potential(t)) \toas \int g \, \diff\pi \, .
\]
This concludes the proof of~\Cref{prop:psi-ergodic}.

\subsection{Analysis for extreme cases}
\label{sec:dual_exact}

In this section, we prove~\Cref{thm: extreme-cases} by deriving an exact formula for the expected size of a maximum matching in the two extreme cases shown respectively in~\Cref{fig-Markov-Chain-Regions-c-to-0} and~\Cref{fig-Markov-Chain-Regions-eps-to-0}.

\begin{proof}[Proof of~\Cref{thm: extreme-cases}]
Consider a partition of $\reals^2$, defined by splitting regions $\calA$, $\calB$ and $\calE$ into two parts each. This finer partition helps visualize the extreme cases and the stationary distribution, and is given by
\begin{align*}
    \calA_1 &= \{(x,y): y\geq 2\base,\ x\leq y \}, 
    \qquad \qquad \qquad ~~~~~~~~~\,
    \calA_2 = \{(x,y): y\geq 2\base,\ y \leq x\leq y + 2\extra\} 
    \\
    \calB_1 &= \{(x,y): 0\leq y \leq 2\base,\ x\leq y \},
    \qquad  \qquad ~~~~~~~~~~\,
    \calB_2 = \{(x,y): 0\leq y \leq 2\base,\ y \leq x \leq y + 2\extra\}
    \\
    \calC &= \{(x,y): x\geq 2(\base+\extra),\ y\leq x-2\extra\}, \qquad \quad ~~~~~~
    \calD = \{(x,y): x\leq 0,\ y \leq 0\},
    \\
    \calE_1 &= \{(x,y): 2\extra \leq x \leq 2(\base+\extra),\ y\leq x-2\extra \},
    \qquad~\;
    \calE_2 = \{(x,y): 0 \leq x \leq 2\extra,\ y \leq 0 \},  
\end{align*}
Note that $\calA_1 \cup \calA_2 = \calA $, $\calB_1 \cup \calB_2 = \calB$ and $\calE_1 \cup \calE_2  =\calE$. 

\begin{figure}[t]
    \centering
    \includegraphics[width=0.99\linewidth]{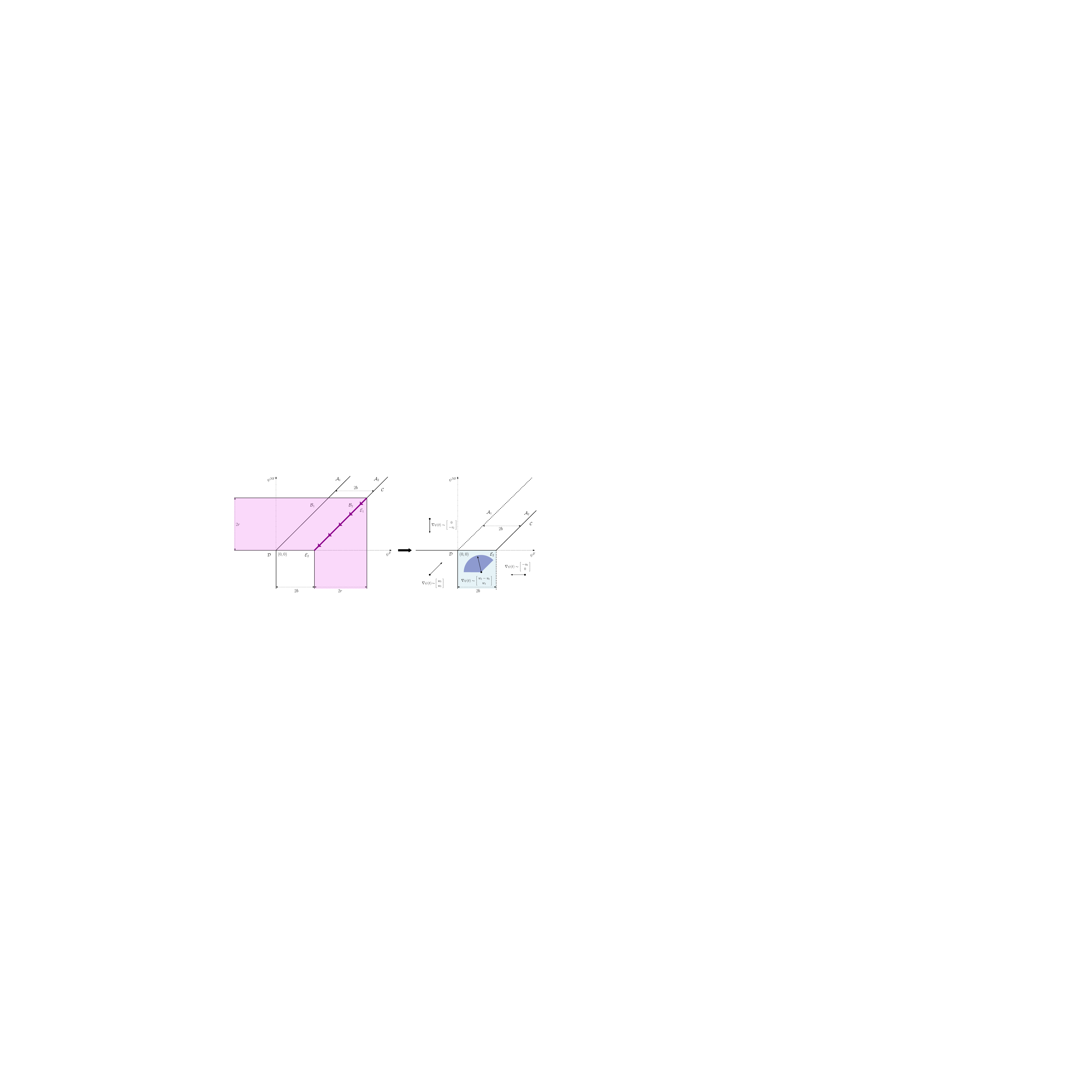}
    \caption{Dynamics with $\base = 0$, obtained by shrinking shaded region along the marked direction.}
    \label{fig-Markov-Chain-Regions-c-to-0}
\end{figure}

\medskip 
\begin{enumerate}
\item[$\mathrm{(i)}$] \textit{The case of $\base = 0$.} ~This case can be visualized as the limit of shrinking the strip $\calB_1 \cup \calE_1$ along the line joining $(2(\base+\extra),2\base)$ to $(2\extra,0)$. This transformation and the resulting Markov process $\potential^{(\base=0)}$ is shown in Figure~\ref{fig-Markov-Chain-Regions-c-to-0}. Only the regions $\calA_1, \calA_2, \calC, \calD, \calE_2$ survive. The stationary distribution can be explicitly computed in this regime as shown in Lemma~\ref{lem-stationary-density-c-equal-0}.

\begin{figure}[t]
    \centering
    \includegraphics[width=0.99\linewidth]{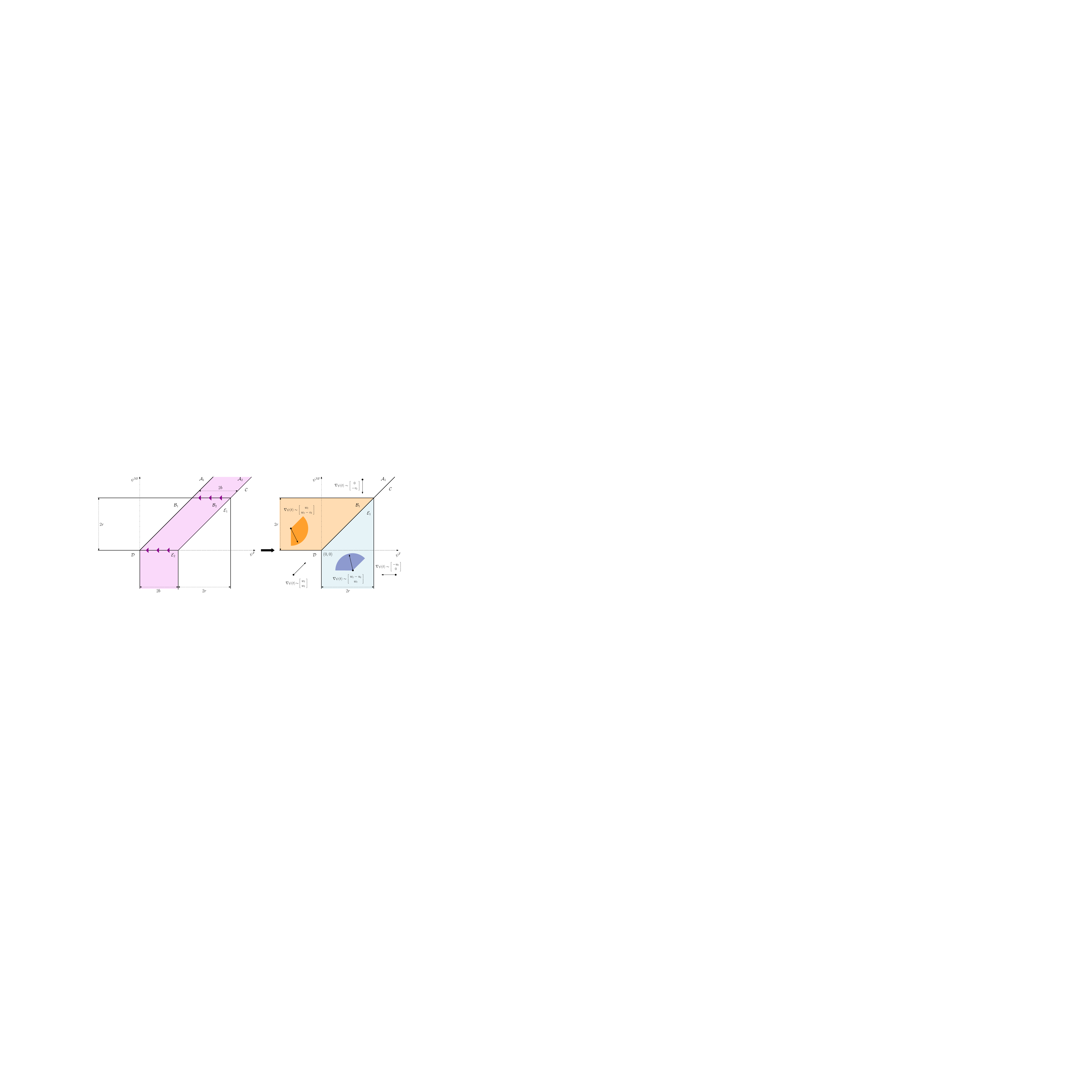}
    \caption{Dynamics with $\extra = 0$, obtained by collapsing shaded region along the marked direction.}
    \label{fig-Markov-Chain-Regions-eps-to-0}
\end{figure}

\begin{lemma}[Stationary density, $\base = 0 $] \label{lem-stationary-density-c-equal-0}
Fix $\extra>0$ and $p\in(0,1)$. Let  $\pi:\mathbb{R}^2\to[0,\infty)$ be 
\[
    \pi(x,y)=
    \begin{cases}
        C\,e^{p x-(1+p)y}, & (x,y)\in \calA_1,\\
        C\,e^{-(1-p)x-py}, & (x,y)\in \calA_2,\\
        C\,e^{2\extra}\,e^{-(2-p)x+(1-p)y}, & (x,y)\in \calC,\\
        C\,e^{p x+(1-p)y}, & (x,y)\in \calD,\\
        C\,e^{-(1-p)x+(1-p)y}, & (x,y)\in \calE_2,
    \end{cases}
    \qquad
    \mbox{ where }C \triangleq \frac{e^{2\extra}\,p(1-p)^2}{(2-p)\,e^{2\extra}-p\,e^{2p\extra}}.
\]
Then $\pi$ is a stationary density for the Markov chain $\potential^{(\base=0)}$.
\end{lemma}
We may then directly compute
\[ 
    F_{\calD} = \int_{-\infty}^0 \int_{-\infty}^0 \frac{e^{2\extra}p(1-p)^2}{(2-p)e^{2\extra} - pe^{2p\extra}}\, e^{px + (1-p)y} \, \diff x \, \diff y = \frac{(1-p)e^{2\extra} }{(2-p)e^{2\extra} - p e^{2p\extra}}.
\]
Substituting this in~\Cref{thm: sizes-of-maximum-matching}, and using~\Cref{prop:gen-emb-eq} yields the formula for the case $\base=0$.
\item[\( \mathrm{(ii)} \)] {\textit{The case of $\extra = 0$.}} ~This can be visualized as shrinking the strip $\calE_2 \cup \calB_2 \cup \calA_2$ of width $2\extra$ to width zero. This transformation and the resulting Markov process $\potential^{(\extra=0)}$ is shown in Figure~\ref{fig-Markov-Chain-Regions-eps-to-0}. Only the regions $\calA_1, \calB_1,\calC, \calD, \calE_1$ survive. The stationary distribution can be explicitly computed in this regime as shown in~\Cref{lem-stationary-density-eps-equal-0}.

\begin{lemma}[Stationary density, $\extra = 0 $] \label{lem-stationary-density-eps-equal-0} 
Fix $\base>0$ and $p\in(0,1)$. Let $\pi:\mathbb{R}^2\to[0,\infty)$ be
\[
    \pi(x,y)=
    \begin{cases}
        C e^{2\base}\,e^{p x-(1+p)y}, & (x,y)\in \calA_1,\\[2pt]
        C e^{p x-p y}, & (x,y)\in \calB_1,\\[2pt]
        C e^{2\base}\,e^{-(2-p)x+(1-p)y}, & (x,y)\in \calC,\\[2pt]
        C e^{p x+(1-p)y}, & (x,y)\in \calD,\\[2pt]
        C e^{-(1-p)x+(1-p)y}, & (x,y)\in \calE_1.
    \end{cases}
    \qquad \mbox{ where } C = \frac{p(1-p)}{2(1+\base)}.
\]
Then $\pi$ is a stationary density for the Markov chain $\potential^{(\extra=0)}$. 
\end{lemma}
We may then directly compute
\[ 
    F_{\calD} = \int_{-\infty}^0 \int_{-\infty}^0 \frac{p(1-p)}{2(1+\base)} \, e^{px +(1-p) y} \, \diff x \, \diff y = \frac{1}{2(1+\base)}.
\]
Substituting this in~\Cref{thm: sizes-of-maximum-matching}, and using~\Cref{prop:gen-emb-eq} yields the formula for the case $\extra = 0$. To conclude the proof of~\Cref{thm: extreme-cases}, note that the case $p=0$ corresponds to having all supply nodes inflexible almost surely, which is equivalent to the case $\extra = 0$ for any value of $p$. 
\end{enumerate}
This concludes the proof of~\Cref{thm: extreme-cases}.
\end{proof}


\section{Simulations} \label{sec-simulations}
In this section, we present simulations illustrating three main findings. First, we confirm that the uniformity principle holds when the service range parametrizes the volume of the service region. Second, we demonstrate that the uniformity principle no longer holds when the service range parametrizes the radius of the service region. Third, we verify the agreement between the matching rate obtained by analyzing the embedded Markov chain and the true fraction of matched nodes in a maximum matching. We then validate the upper and lower bounds~\eqref{eq: upper-bound} and~\eqref{eq: lower-bound-with-q} for the dual service range model against the simulation data.

Throughout, we generate $n=400$ supply nodes and $m=400$ demand nodes i.i.d.\ uniformly on $[0,1]^{\dimension}$, form the induced bipartite graph under the specified connectivity rule, and compute the maximum matching size. Curves report the average matching fraction, and shaded regions indicate $\pm 1$ standard deviation over $10^4$ independent trials.

To generate service range vectors that are ordered according to majorization, we use a randomized version of our dual service range model, parameterized by a scalar $\alpha\in[0,1]$:
\begin{align} \label{eq:service-range-random}
    \servicevec_i(\alpha)=\base(\alpha)+\extra(\alpha)\,X_i(\alpha), ~~~~~~
    X_i(\alpha) \stackrel{\mathrm{i.i.d}}{\sim} \mathrm{Bernoulli}\big(p(\alpha)\big)\, ,
\end{align}
i.e., each supply node has service range $\base(\alpha)$ with probability $1-p(\alpha)$, and range $\base(\alpha)+\extra(\alpha)$ with probability $p(\alpha)$. To isolate the effect of distribution shape, the parameters $\base(\alpha)$, $\extra(\alpha)$, and $p(\alpha)$ are constrained such that the expected service range 
\[
\bar{\base} \triangleq \mathbb{E}[\servicevec_i(\alpha)] = \base(\alpha) + p(\alpha) \extra(\alpha)
\] 
remains constant across all $\alpha$. The scalar $\alpha$ serves as a control for inequality: for a two-value distribution with a fixed mean, increasing $\alpha$ increases the variance, which results in a more unequal allocation, i.e. a more majorized vector.

Accordingly, we consider three families of service range vectors, each parameterized by $\alpha$: (i) the fixed $r$ family with $\base(\alpha)=\base_0$, (ii) the fixed $b$ family with $\extra(\alpha)= \extra_0 $, and (iii) the fixed $p$ family with $p(\alpha)= p_0$. For each family, the varying parameters are tuned to vary monotonically with $\alpha$ such that the variance of $\servicevec_i(\alpha)$ increases with $\alpha$. This construction ensures that
\[
    \alpha_1 < \alpha_2 \implies \servicevec(\alpha_1) \preceq_{\mathrm{st}}  \servicevec(\alpha_2) \, ,\footnote{
   Since $\servicevec(\alpha)$ is random, $\preceq_{\mathrm{st}}$ denotes stochastic majorization. Let $S_\ell(\alpha) = \sum_{i=1}^\ell \servicevec_{[i]}(\alpha)$ be the sum of the $\ell$ largest coordinates of the random vector. The condition implies that for all $\ell \in [n]$, the partial sum $S_\ell(\alpha_2)$ stochastically dominates $S_\ell(\alpha_1)$ (denoted $S_\ell(\alpha_1) \leq_{\mathrm{st}} S_\ell(\alpha_2)$).}
\]
i.e., a larger $\alpha$ typically produces a more majorized vector of service ranges. For example, in the fixed $r$ family with $\base(\alpha)=\base_0$, this is achieved by decreasing $p(\alpha)$ while increasing $\extra(\alpha)$ as $\alpha$ increases. The specific definitions for all three families are provided in Table~\ref{tab:sim-uniformity}.

\begin{table}[t]
    \centering
    \setlength{\tabcolsep}{6pt}
    \renewcommand{\arraystretch}{1.15}
    \begin{tabular}{ccc}
        \toprule
        \textbf{Family} &  \textbf{Varied in $\alpha$} & \textbf{Implied by $\bar{\base}=\base+p\extra$} \\
        \midrule
        Fixed $\base$: $\base(\alpha)=\base_0$
        & $p(\alpha)=p_{\max}-\alpha\big(p_{\max}-p_{\min}\big)$
        & $\extra(\alpha)=(\bar{\base}-\base_0)/p(\alpha)$ \\
        \addlinespace
        Fixed $\extra$: $\extra(\alpha)=\extra_0$
        & $p(\alpha)=p_{\min}+\alpha\big(p_{\max}-p_{\min}\big)$
        & $\base(\alpha)=\bar{\base}-p(\alpha)\extra_0$ \\
        \addlinespace
        Fixed $p$: $p(\alpha)=p_0$
        & $\base(\alpha)=\base_{\max}-\alpha\big(\base_{\max}-\base_{\min}\big)$
        & $\extra(\alpha)=(\bar{\base}-\base(\alpha))/p_0$ \\
        \bottomrule
    \end{tabular}
    \caption{One-parameter families used in our simulations. In each family, the parameters are chosen so that the mean $\bar{\base}$ is held constant while variance (and hence majorization) increases with~$\alpha$.}
    \label{tab:sim-uniformity}
\end{table}

\paragraph{The uniformity principle.}
We evaluate the uniformity principle by plotting the average matching fraction (over $10^4$ trials) against the inequality parameter $\alpha$. We perform this evaluation for the three service range families (fixed $\base$, fixed $\extra$, and fixed $p$) defined in Table~\ref{tab:sim-uniformity}. The specific numerical parameters used for dimensions $\dimension=1, 2, 3$ are detailed in Table~\ref{tab:sim-uniformity-params}. 

\begin{table}[t]
\centering
\setlength{\tabcolsep}{6pt}
\renewcommand{\arraystretch}{1.15}
\begin{tabular}{cccc}
\toprule
\textbf{Dimension} $\dimension$
& \textbf{Family}
& \textbf{Mean} ${\bar{\base}}$
& \textbf{Parameters} \\
\midrule
\multirow{3}{*}{$1$}
    & Fixed $\base$: $\base(\alpha)=\base_0$  & \multirow{3}{*}{$1$}
        & $\base_0=0.5,\ p_{\min}=0.2,\ p_{\max}=0.8$ \\
    & Fixed $\extra$: $\extra(\alpha)=\extra_0$  & 
        & $\extra_0=0.6,\ p_{\min}=0.01,\ p_{\max}=0.5$ \\
    & Fixed $p$: $p(\alpha)=p_0$      & 
        & $p_0=0.3,\ \base_{\min}=0.5,\ \base_{\max}=1$ \\
\addlinespace
\midrule
\multirow{3}{*}{$2$}
    & Fixed $\base$: $\base(\alpha)=\base_0$  & \multirow{3}{*}{$0.15$}
        & $\base_0=0.025,\ p_{\min}=0.2,\ p_{\max}=0.8$ \\
    & Fixed $\extra$: $\extra(\alpha)=\extra_0$  & 
        & $\extra_0=0.3,\ p_{\min}=0.0,\ p_{\max}=0.5$ \\
    &Fixed $p$: $p(\alpha)=p_0$     & 
        & $p_0=0.5,\ \base_{\min}=0.006,\ \base_{\max}=0.12$ \\
\addlinespace
\midrule
\multirow{3}{*}{$3$}
    & Fixed $\base$: $\base(\alpha)=\base_0$  & \multirow{3}{*}{$0.1$}
        & $\base_0=0.02,\ p_{\min}=0.2,\ p_{\max}=0.8$ \\
    & Fixed $\extra$: $\extra(\alpha)=\extra_0$  & 
        & $\extra_0=0.2,\ p_{\min}=0.0,\ p_{\max}=0.5$ \\
    & Fixed $p$: $p(\alpha)=p_0$     & 
        & $p_0=0.5,\ \base_{\min}=0.005,\ \base_{\max}=0.1$ \\
\bottomrule
\end{tabular}
\caption{Numerical choices used to generate $\servicevec(\alpha)$ in Figure~\ref{fig: uniformity}. For each dimension, the mean constraint $\bar{\base}=\base+p\extra$ is fixed across families, and the remaining parameters vary linearly with $\alpha$ as specified in Table~\ref{tab:sim-uniformity}.}
\label{tab:sim-uniformity-params}
\end{table}

Figure~\ref{fig: uniformity} demonstrates that for all three families and across all tested dimensions, the matching fraction decreases monotonically with $\alpha$. This observation validates the uniformity principle: a smaller $\alpha$ corresponds to a less majorized (more uniform) service range vector, which systematically yields a larger maximum matching.

\begin{figure}[t]
    \centering
    \subfigure[$\dimension = 1$]{
    \includegraphics[width=0.3\linewidth]{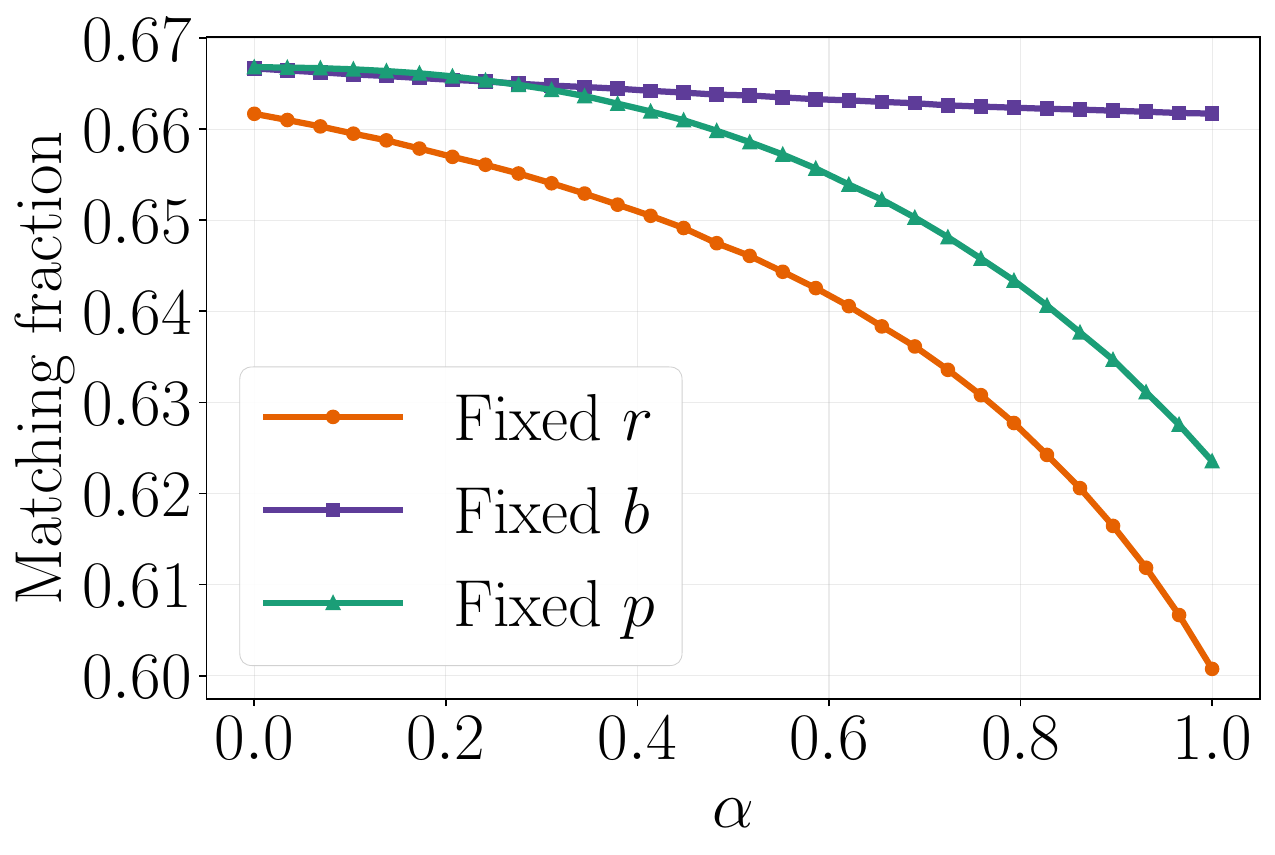}
    \label{fig: uniformity-1D}
    }
    \hfill
    \subfigure[$\dimension = 2$]{
    \includegraphics[width=0.3\linewidth]{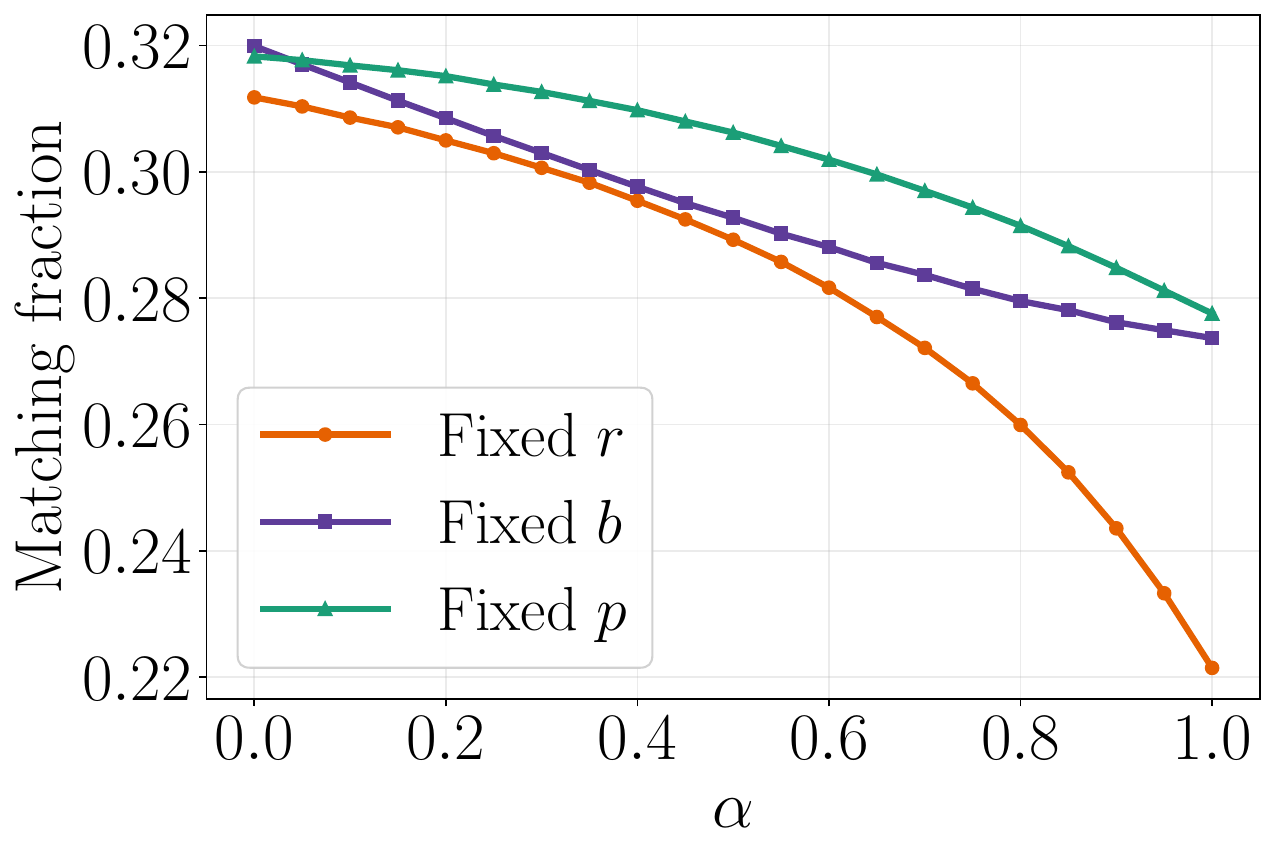}
    \label{fig: uniformity-2D}
    }
    \hfill
    \subfigure[$\dimension = 3$]{
    \includegraphics[width=0.3\linewidth]{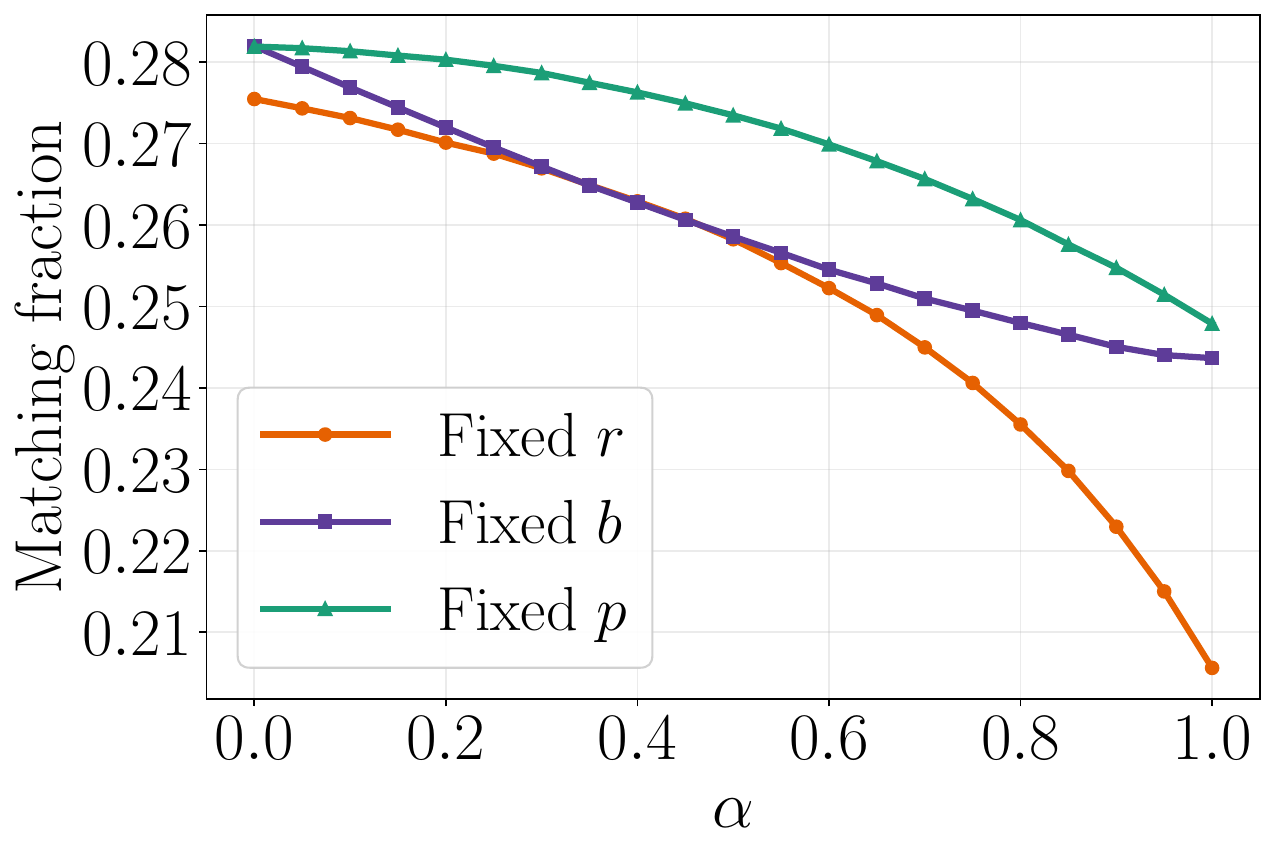}
    \label{fig: uniformity-3D}
    }
    \caption{Illustration of the uniformity principle. Each plot shows three curves corresponding to the fixed $\base$ with $\base(\alpha)=\base_0$, fixed $\extra$ with $\extra(\alpha)=\extra_0$, and fixed $p$ with $p(\alpha)=p_0$  families respectively.} 
    \label{fig: uniformity}
\end{figure}

\paragraph{Comparing volume and radius parametrizations.} 
As discussed in~\Cref{rem:radiusvsrange}, we investigate an alternative modeling choice where the service range parametrizes the radius of the service region rather than its volume. We focus on dimension $\dimension=2$, where the two models are defined as follows:
\begin{itemize}
    \item (Volume parametrization) Given a vector $\servicevec = (\service_i)_{i=1}^n$ of service ranges, $\driver_i$ and $\rider_j$ are connected iff 
    \[
        \|\driver_i-\rider_j\|_2  \, \leq  \, \sqrt{\service_i/n} \, .
    \]
    \item (Radius parametrization) Given a vector $\servicevec = (\service_i)_{i=1}^n$ of service ranges, $\driver_i$ and $\rider_j$ are connected iff
    \[
        \|\driver_i-\rider_j\|_2 \leq \service_i/\sqrt n \, .
    \]
\end{itemize}

We evaluate both models using the same three  families defined in Table~\ref{tab:sim-uniformity} (fixed $\base$, fixed $\extra$, and fixed $p$). The specific numerical parameters for these simulations are detailed in Table~\ref{tab:sim-params}. The matching fractions for both models are plotted in Figure~\ref{fig: radius-vs-range}. Consistent with~\Cref{thm:uniformity}, the volume parametrization obeys the uniformity principle, evidenced by the matching fraction decreasing monotonically with $\alpha$ across all panels. In contrast, the radius model violates this principle; the matching fraction is not monotonically decreasing. Specifically, there exist instances where $\alpha_1 < \alpha_2$ (implying $\servicevec(\alpha_1) \preceq_{\mathrm{st}} \servicevec(\alpha_2)$) yet the matching fraction is strictly larger at~$\alpha_2$.

\begin{figure}[t]
    \centering
    \subfigure[Fixed $\base$ with $\base(\alpha) = \base_0$]{
    \includegraphics[width=0.3\linewidth]{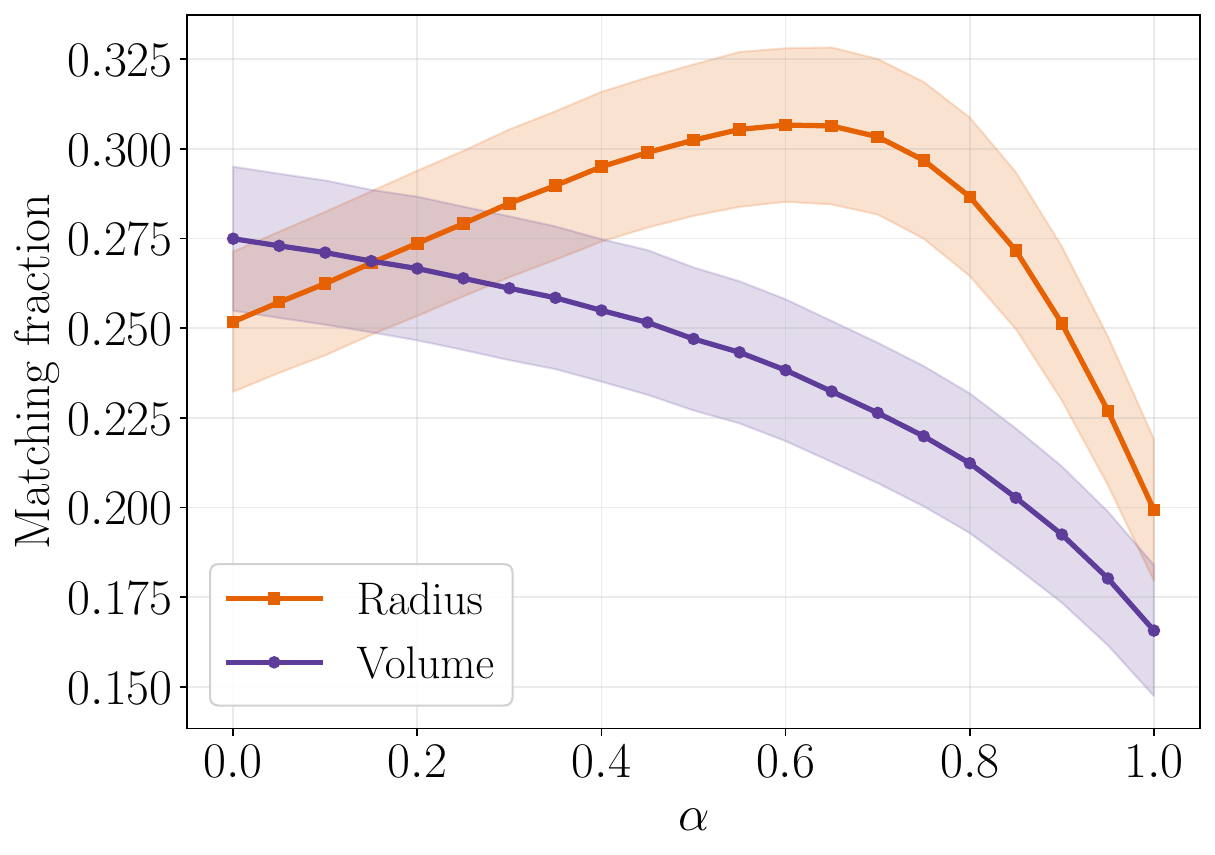}
    \label{fig: radius-vs-range-fixed-c}
    }
    \hfill
    \subfigure[Fixed $\extra$ with $\extra(\alpha) = \extra_0$ ]{
    \includegraphics[width=0.3\linewidth]{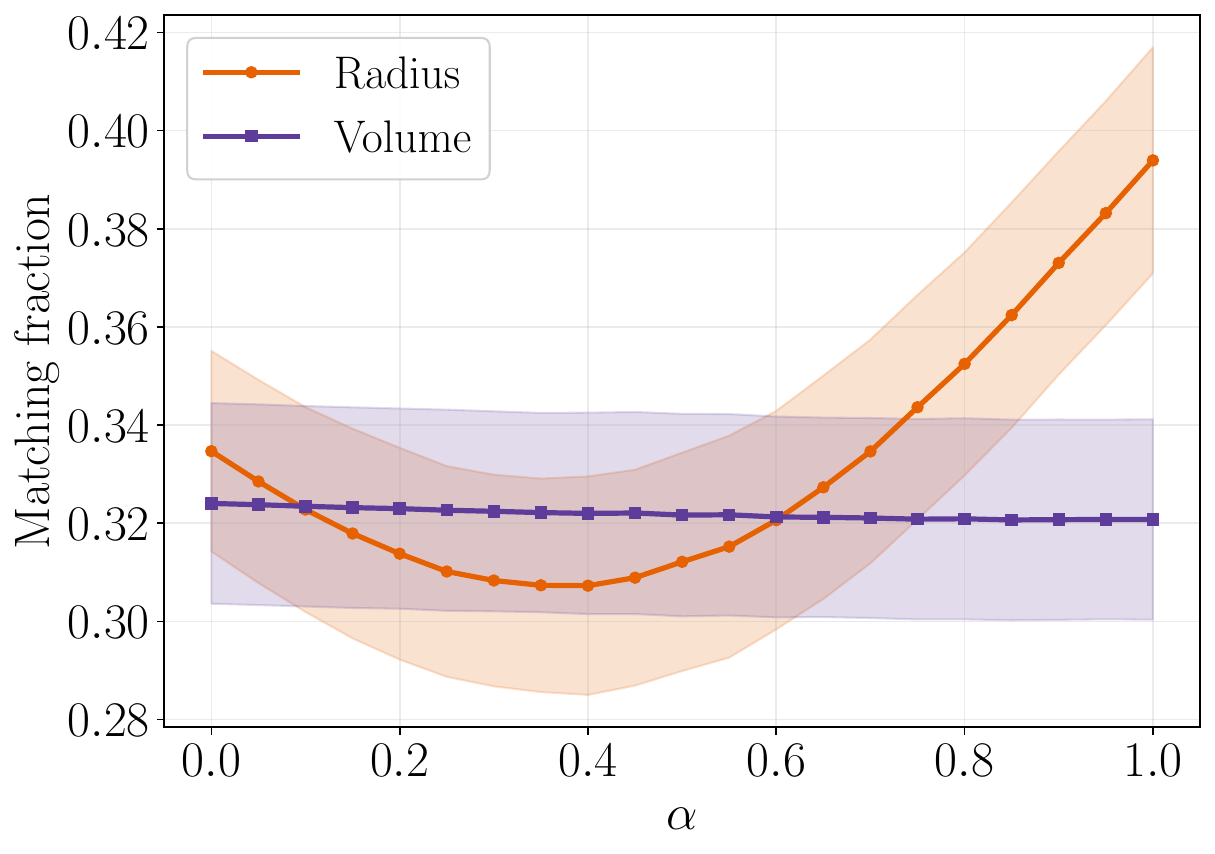}
    \label{fig: radius-vs-range-fixed-eps}
    }
    \hfill
    \subfigure[Fixed $p$ with $p(\alpha) = p_0$]{
    \includegraphics[width=0.3\linewidth]{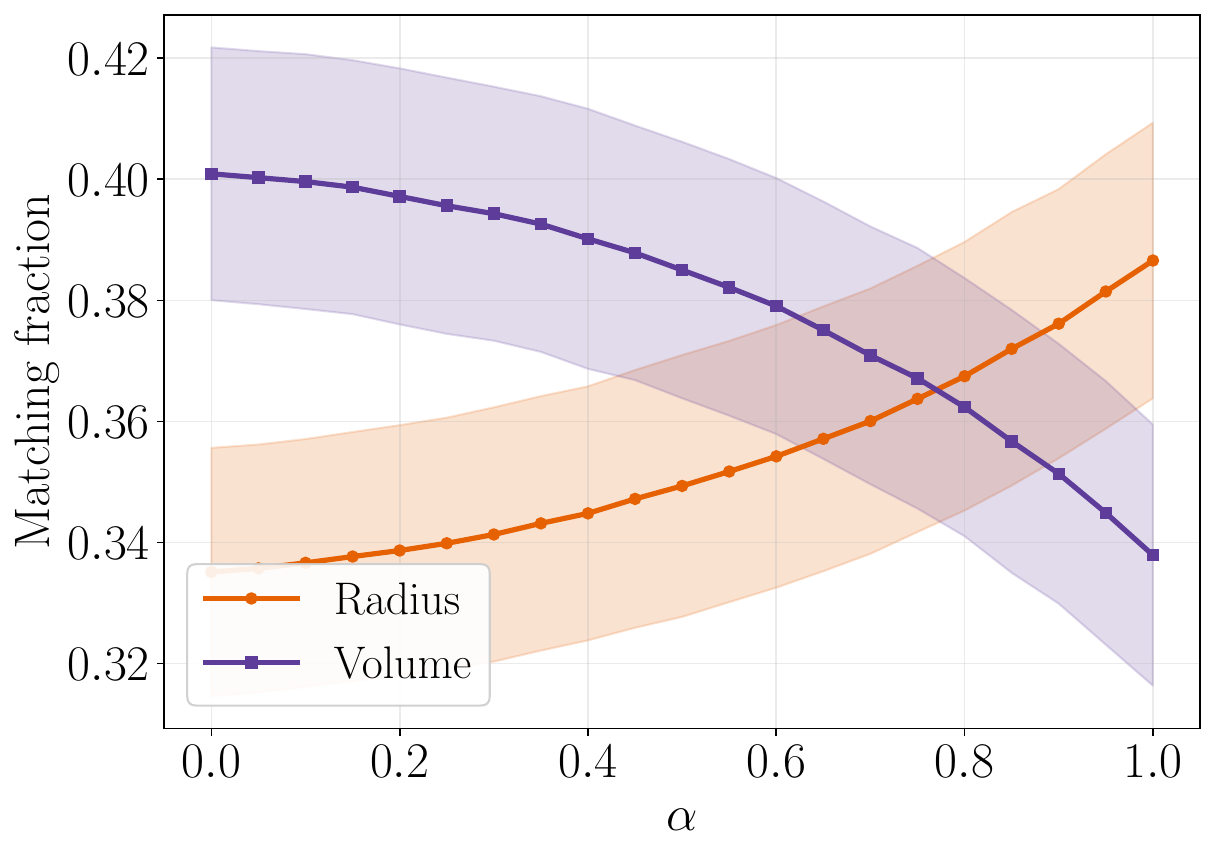}
    \label{fig: radius-vs-range-fixed-p}
    }
    \caption{
    Illustration of the uniformity principle holding when service range parametrizes volume, but not holding when it parametrizes radius, for the three families in dimension $\dimension = 2$ and parameters in~\Cref{tab:sim-params}. 
    }
    \label{fig: radius-vs-range}
\end{figure}

\begin{figure}[t]
    \centering
    \subfigure[Varying $\base$]{
    \includegraphics[width=0.3\linewidth]{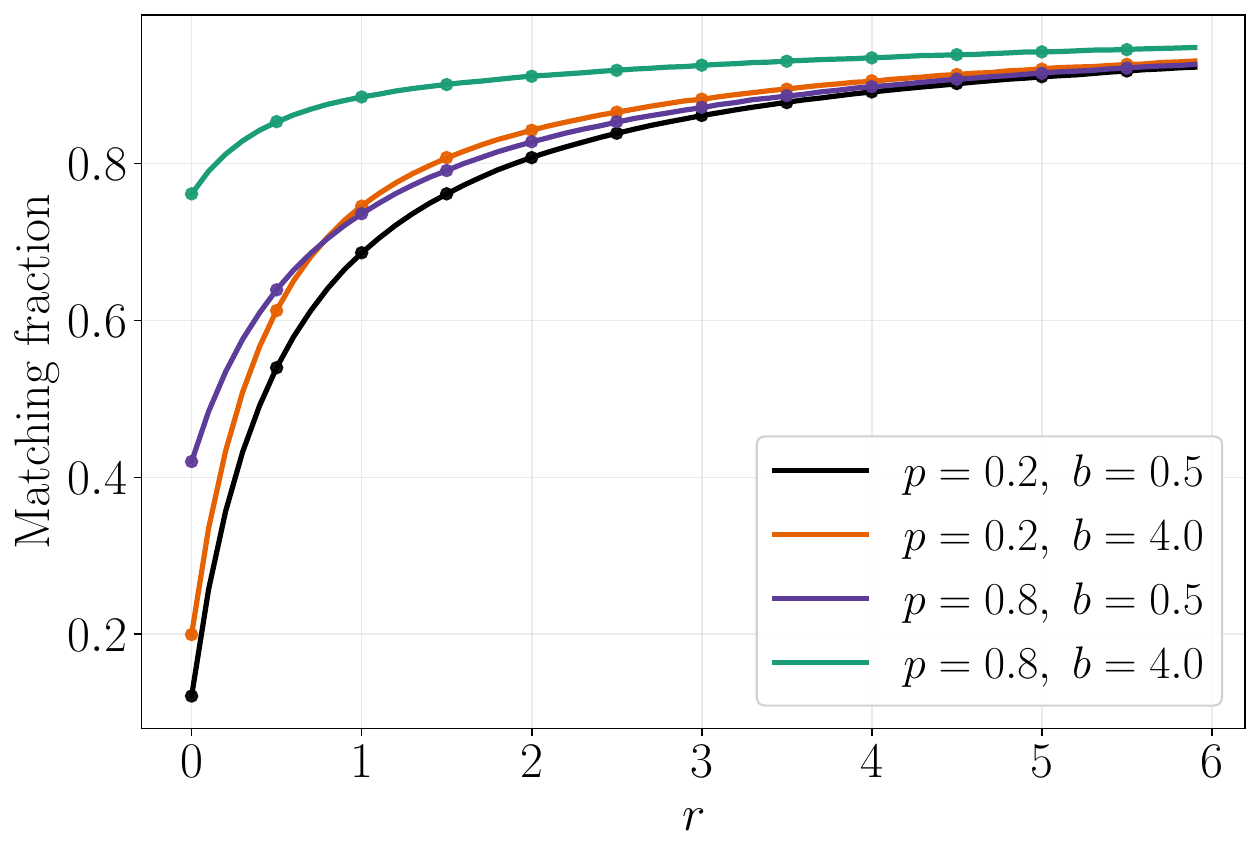}
    \label{fig: markov-versus-r}
    }
    \hfill 
    \subfigure[Varying $\extra$]{
    \includegraphics[width=0.3\linewidth]{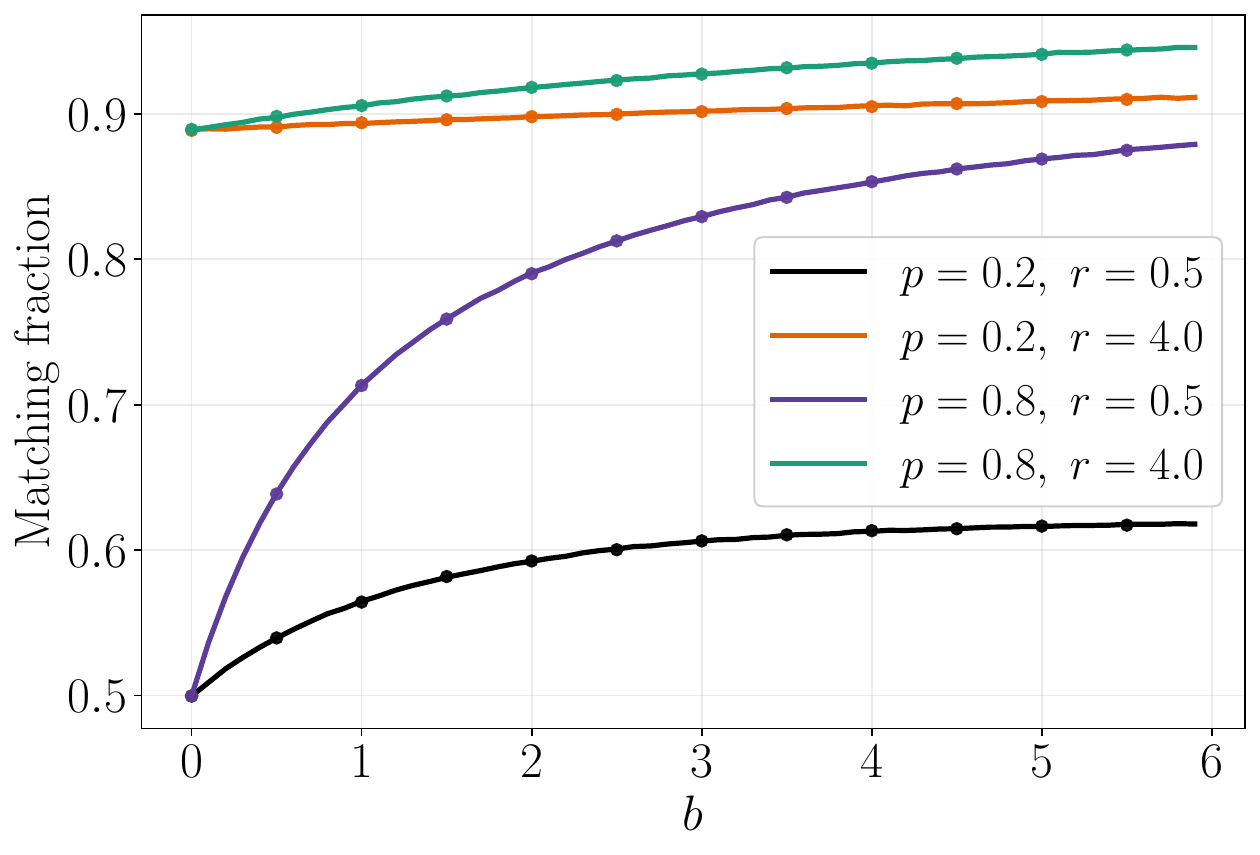}
    \label{fig: markov-versus-b}
    }
    \hfill 
    \subfigure[Varying $p$]{
    \includegraphics[width=0.3\linewidth]{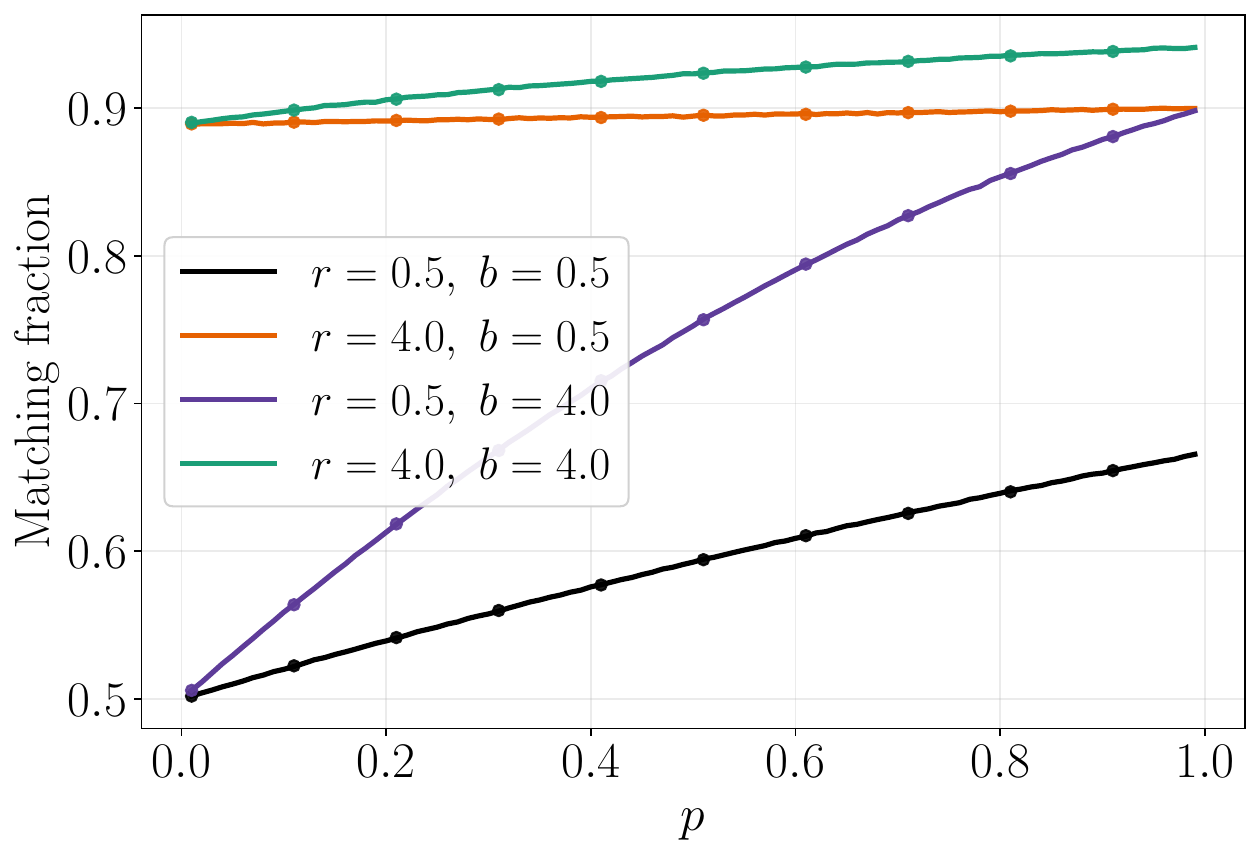}
    \label{fig: markov-versus-p}
    }
    \caption{Solid line represents true maximum matching size, with markers representing the formula~\eqref{eq: match-total} obtained from the stationary distribution of the Markov chain $\potential(t)$.}
    \label{fig: markov-chain-simulation}
\end{figure}

\begin{figure}[t]
    \centering
    \subfigure[Bounds vs. $\base$ ($\extra = 2$, $p = 0.6$)]{
    \includegraphics[width=0.3\linewidth]{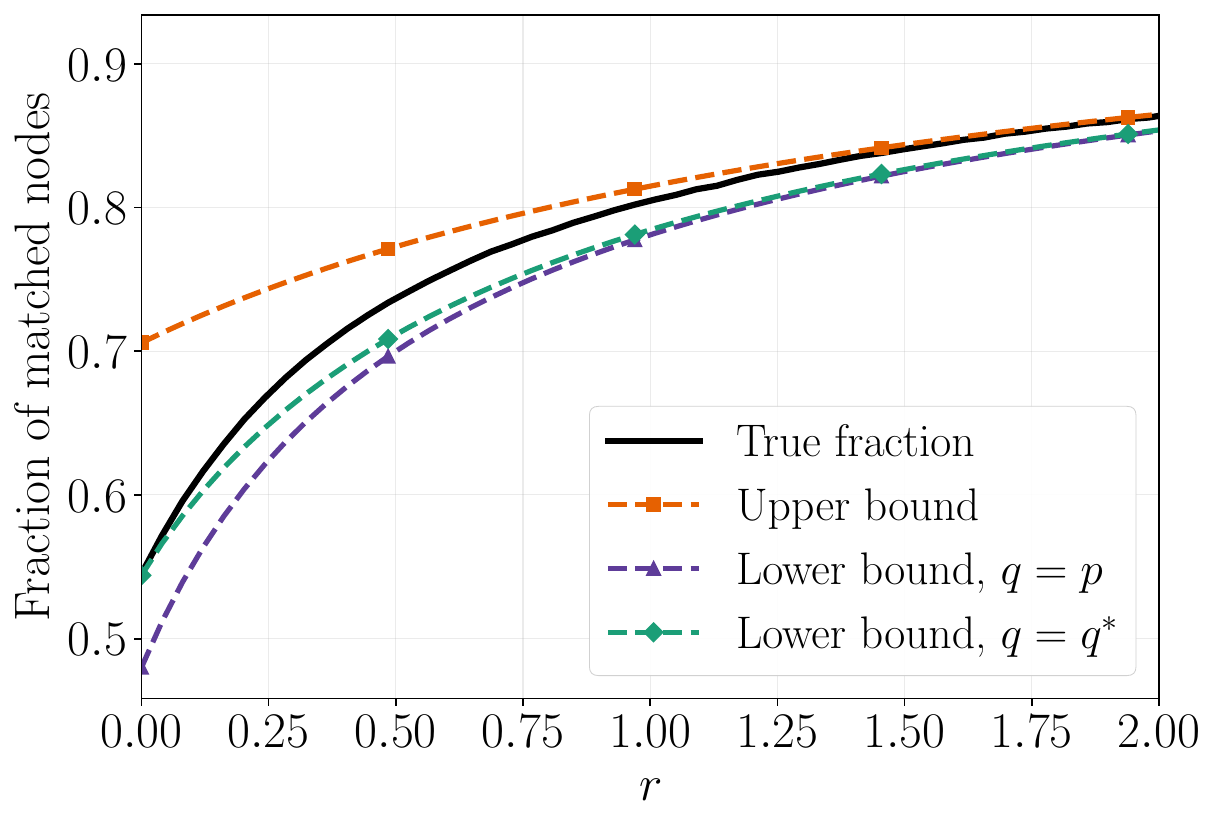}
    \label{fig: mrb-vs-r}
    }
    \hfill 
    \subfigure[Bounds vs. $\extra$ ($\base = 1/4$, $p = 1/2$)]{
    \includegraphics[width=0.3\linewidth]{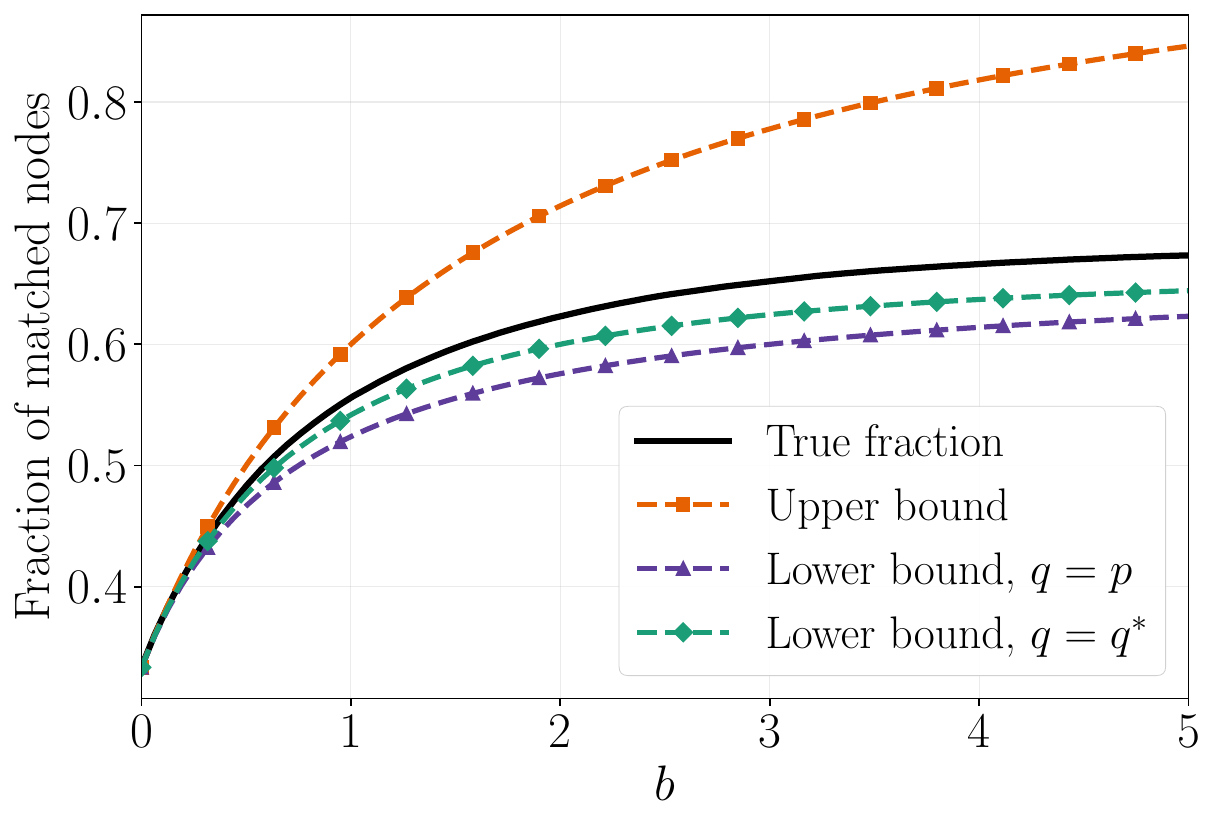}
    \label{fig: mrb-vs-b}
    }
    \hfill 
    \subfigure[Bounds vs. $p$ ($\base = 1$, $\extra =1$)]{
    \includegraphics[width=0.3\linewidth]{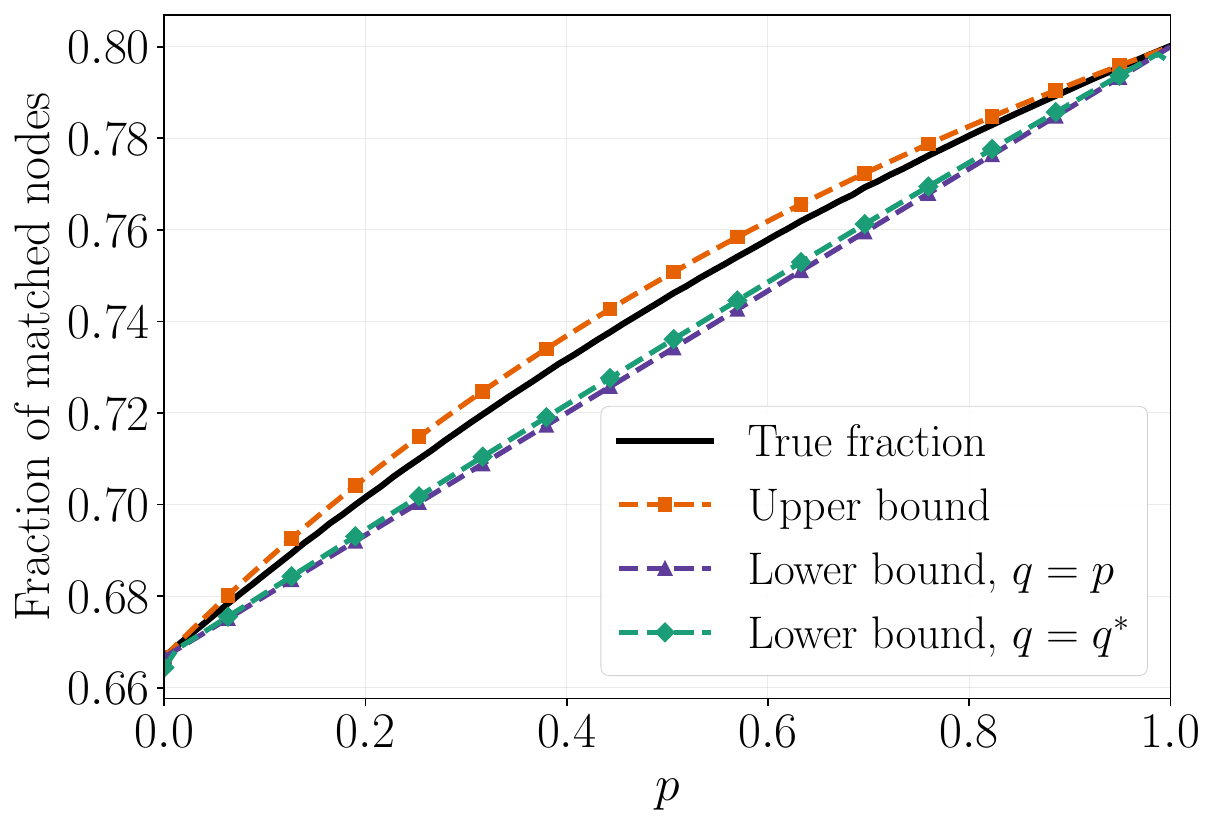}
    \label{fig: mrb-vs-p}
    }    
    \caption{Upper and lower bounds for matching fraction in dual service range model.}
    \label{fig: mrb}
\end{figure}

\begin{table}[t]
    \centering
    \setlength{\tabcolsep}{6pt}
    \renewcommand{\arraystretch}{1.15}
    \begin{tabular}{ccc}
        \toprule
        \textbf{Family}
        & \textbf{Radius interpretation}
        & \textbf{Volume interpretation} \\
        \midrule
            Fixed $\base$
                & $\bar{\base}=0.3,\,\base_{0}=0,\, p\in[0.2,0.8]$
                & $\bar{\base}=0.127,\,\base_{0}=0,\, p\in[0.2,0.8]$ \\[4pt]
            Fixed $\extra$
                & $\bar{\base}=0.4,\,\extra_{0}=0.8,\, p \in[0,0.5]$
                & $\bar{\base}=0.15,\,\extra_{0}=0.09,\, p\in[0,0.5]$ \\[4pt]
            Fixed $p$
                & $\bar{\base}=0.4,\, p_0=0.5,\,\base\in[0.02,0.36]$
                & $\bar{\base}= 0.21,\, p_0=0.5,\,\base \in[0.011,\, 0.191]$ \\
        \bottomrule
    \end{tabular}
    \caption{Simulation parameters for the three one-parameter families under both radius and volume interpretations. Each parametrization preserves the mean $\bar{\base}$ while increasing the dispersion of $\servicevec(\alpha)$ with~$\alpha$.}
    \label{tab:sim-params}
\end{table}

\paragraph{Markov chain.}
Figure~\ref{fig: markov-chain-simulation} checks the accuracy of the closed-form expression~\eqref{eq: match-total}, which is derived from the stationary distribution of the Markov chain $\potential(t)$. For each parameter setting, we plot (i) the matching fraction obtained by the embedded Markov chain, and (ii) the exact matching fraction obtained by solving the underlying bipartite matching instance via the Hopcroft--Karp algorithm. We plot the two quantities under three separate parameter sweeps, varying $\base$, $\extra$, and $p$ (holding the remaining parameters fixed). In all panels, the theoretical and empirical curves agree, validating the formula~\eqref{eq: match-total}.

Finally, Figure~\ref{fig: mrb} verifies the theoretical bounds derived in~\Cref{rmk: bounds-for-dual-service range} for the dual-service range model. We plot these bounds as functions of $\base$, $\extra$, and $p$ (holding other parameters fixed) and observe that they effectively bracket the true matching fraction. The upper bound~\eqref{eq: upper-bound} corresponds to the performance of the uniform allocation (via the uniformity principle), while the lower bounds~\eqref{eq: lower-bound-with-q} and~\eqref{eq: lower-bound-with-p} represent the performance guaranteed by decomposing the graph into disjoint subgraphs.

\section{Conclusion} \label{sec-conclusion}

We studied the ex-ante design of flexibility in spatial matching environments modeled by bipartite random geometric graphs. For a fixed total budget of service range, we established a sharp {uniformity principle}: making the vector of service ranges more uniform (in the sense of majorization) strictly increases the expected size of a maximum matching, in all dimensions $\dimension \ge 1$ under mild regularity conditions on the spatial distribution. This effect is driven by diminishing returns to additional service range and by the way bounded ranges fragment the graph into mostly local components. 

Building on this insight, we introduced a dual service range model with inflexible and flexible supply nodes and developed a Markov-chain embedding that yields tractable characterizations and bounds for the expected matching rate, including closed-form expressions. Taken together, these results provide structural guidance for spatial platforms such as autonomous fleets, drone delivery, emergency services, and feature-based matching markets by showing that, when designing coverage ex-ante, it is typically better to spread flexibility broadly than to concentrate it on a small set of ``super'' units.

Our analysis suggests several directions for future work.
First, our simulations indicate that the uniformity principle depends critically on the decision variable: it holds for service range parametrized by volume but can fail for service radius parametrized by radius. Note that while these formulations are equivalent for $\dimension=1$, they  diverge for $\dimension \ge 2$ due to the non-linear scaling of volume with radius. Characterizing optimal allocations under radius constraints, where cost and coverage volume scale non-linearly, remains an open challenge.

Second, we study an {ex-ante} benchmark in which ranges are chosen before supply and demand locations are realized. In many platforms, the spatial distribution of supply is relatively stable (e.g., depots or driver home bases), suggesting a two-stage design where ranges are chosen after observing supply locations but before demand arrives; quantifying the value of such partial adaptivity is an important direction. 

Finally, our results treat flexibility on the supply side only. Extending the analysis to {two-sided} flexibility under coupled budgets in geometric environments, in the spirit of~\cite{freundmartinzhao2024twosided} but with spatial locality and congestion, remains open, including whether flexibility should be concentrated on one side or shared across both, and whether any form of uniformity principle survives.

\section*{Acknowledgement}
The authors thank Philipp Afèche, Gérard P. Cachon, Ozan Candogan, Daniel Freund, Bruce Hajek, Ming Hu, Yashaswini Murthy, Rad Niazadeh, Jiaming Xu, and Jiayu (Kamessi) Zhao for their helpful discussions.

\bibliographystyle{alpha}
\bibliography{biblio}


\appendix

\section{Preliminaries on Markov chain convergence}  \label{apx-preliminaries}
This section reviews the theory of convergence of general state Markov chains. Throughout, $(\sfX,\mathcal B)$ is a measurable space (in our case, $\sfX=\mathbb R^2$ with the Borel $\sigma$–algebra), $P$ is a Markov kernel on $(\sfX,\mathcal B)$, and $P^t(x,\cdot)$ denotes the $t$–step transition law from $x \in \sfX$. 

\paragraph{${\varphi}$–irreducibility.} 
A Markov chain is \emph{$\varphi$–irreducible} if there exists a $\sigma$–finite measure $\varphi$ on $(\sfX,\mathcal B)$ such that for every  $x \in \sfX$ and $A\in\mathcal B$ with $\varphi(A)>0$, there exists $t\ge 1$ with $P^t(x,A)>0$. Intuitively, no set of positive $\varphi$–mass is permanently unreachable from any starting point~ \cite[Ch.~4]{meyn2012markov}.

\paragraph{Small and petite sets.}
A measurable set $C\in\mathcal B(\sfX)$ is \emph{small} if there exist $t\ge 1$, $\varepsilon>0$, and a probability measure $\nu$ on $(\sfX,\mathcal B(\sfX))$ such that
\[
    P^t(x,\cdot)\ \ge\ \varepsilon\,\nu(\cdot)\qquad\text{for all }x\in C \, .
\]
A set $C\in\mathcal B(\sfX)$ is \emph{petite} if there exist a probability distribution $a$ on the nonnegative integers, and a nontrivial measure $\nu$ on $(\sfX,\mathcal B(\sfX))$ such that
\[
    \sum_{t=0}^\infty a(t)\, P^t(x,\cdot)\  \ge \, \nu(\cdot)\qquad \text{for all }x\in C \, ,
\]
i.e. after randomizing the time step with law $a$, the $a$–mixture of the transition kernels dominates a common measure $\nu$ uniformly for all $x\in C$. Note that every small set is petite.

\paragraph{Positive Harris recurrence.}
A $\varphi$–irreducible Markov chain with kernel $P$ on $(\sfX,\mathcal B)$ is \emph{Harris recurrent} if, for every measurable set $A\in\mathcal B$ with $\varphi(A)>0$,
\[
    \mathbb P_x\{\tau_A<\infty\}=1 \qquad \text{for all }x\in \sfX\, ,
\]
where $\tau_A=\inf\{t\ge 1:\,X_t\in A\}$ is the hitting time of $A$.  
The chain is \emph{positive Harris recurrent} if it is Harris recurrent and there exists an invariant probability measure $\pi$ on $(\sfX,\mathcal B)$. A convenient way to check for positive Harris recurrence is the Foster-Lyapunov criteria.

\begin{theorem} \label{thm:PHR}
    Suppose $C \in \calB(\sfX)$ is a petite set such that there exist finite constants $\eta, \,b_0$, and a real valued function $V: \sfX \to \mathbb{R}_+ \cup \{\infty\}$ that satisfies
    \[ 
        \Delta V(x) \triangleq \mathbb{E}[V(X_{t+1}) \ | X_t=x] - V(x) \leq -\eta  + b_0\,\mathbf{1}_C(x), ~~~~ \forall \, x\in\sfX\, .
    \]
    If $V$ is finite everywhere and bounded on $C$, then the Markov chain $(X_t)$ is positive Harris recurrent.
\end{theorem}
\begin{proof}
    This is a restatement of~\cite[Theorem 11.3.4]{meyn2012markov}.
\end{proof}

We conclude this section with a takeaway theorem: Markov chains that are positive Harris recurrent satisfy the following law of large numbers.

\begin{theorem} \label{thm: convergence}
    If $\Phi$ is a positive Harris recurrent Markov chain with invariant probability $\pi$, then for any measurable function $g$ with $\pi(|g|) < \infty$,
    \[ 
        \frac 1 T \sum_{i=0}^{T-1}g(\Phi_{i}) \toas \int g \,\diff \pi \, .
    \]
\end{theorem}
\begin{proof}
    See~\cite[Theorem 17.0.1 (i)]{meyn2012markov}.
\end{proof}

\section{Supplementary proofs for the uniformity principle} \label{sec-deferred-proofs} 

In this section, we present the proofs for various lemmas used in the analysis of the uniformity principle. Following~\Cref{sec-uniformity-principle}, we separate the lemmas based on whether $\dimension=1$ or $\dimension \geq 2$.

\subsection{Supplementary proofs for~\texorpdfstring{\Cref{sec-one-dimension}}{}}

\subsubsection{Proof of~\prettyref{lem:HLP}}
Part~(a) follows from \cite[Theorem~2.B.1(a)]{marshall2011majorization}.
We give a constructive proof that also establishes the identity in~(b) and the bound $T_\star\le n-1$.

\noindent \emph{Setup.}
Without loss of generality, assume $\mathbf x=\mathbf x^\downarrow$ and $\mathbf y=\mathbf y^\downarrow$
(nonincreasing order). Let $z^{(0)}=\mathbf x$. For $r\ge1$, if $z^{(r-1)}\neq\mathbf y$, choose
\[
    i \triangleq \max\{\ell: z^{(r-1)}_\ell>y_\ell\},\qquad
    j \triangleq \min\{\ell>i: z^{(r-1)}_\ell<y_\ell\},
\]
the last surplus and first deficit. Set $a=z^{(r-1)}_i$, $b=z^{(r-1)}_j$ and define
\[
    \tau_r \triangleq \min\{a-y_i, y_j-b\},\qquad t_r \triangleq  \frac{\tau_r}{a-b}\in[0,1].
\]
Update \emph{componentwise} via the $\sfT$-transform on $(i,j)$:
\[
    z_\ell^{(r)} \triangleq 
    \begin{cases}
        a' = (1-t_r) a + t_r\, b = a-\tau_r, & \ell=i,\\
        b' = (1-t_r)\, b + t_r\, a = b+\tau_r, & \ell=j,\\
        z_\ell^{(r-1)}, & \ell\notin\{i,j\}.
    \end{cases}
\]
Thus $z_i^{(r)}=a'$, $z_j^{(r)}=b'$, and all other coordinates are unchanged.

\medskip
\noindent \emph{Order preservation.}
Since $i$ is the last surplus and $j$ the first deficit, it follows that
\[
    z^{(r-1)}_{i+1}\le y_i,\qquad y_j\le z^{(r-1)}_{j-1}.
\]
Hence $a-y_i\le a-z^{(r-1)}_{i+1}$ and $y_j-b\le z^{(r-1)}_{j-1}-b$, which give
\[
    z^{(r)}_i=a'=a-\tau_r\ge z^{(r-1)}_{i+1},\qquad
    z^{(r)}_j=b'=b+\tau_r\le z^{(r-1)}_{j-1}.
\]
Moreover, because $i<j$ and $\mathbf y$ is nonincreasing,
\[
    (a-y_i)+(y_j-b)=(a-b)+(y_j-y_i)\le a-b,
\]
so $\tau_r\le \frac12(a-b)$ and $z^{(r)}_i-z^{(r)}_j=a'-b'=(1-2t_r)(a-b)\ge0$.
Therefore $z^{(r)}$ is still nonincreasing (with the usual boundary conventions).

\medskip
\noindent \emph{Exact $\ell_1$ decrease (part (b)).}
By the choice of $\tau_r$ there is no overshoot toward $\mathbf y$: $a'\ge y_i$ and
$b'\le y_j$, with equality in at least one. Hence
\[
    |z_i^{(r)}-y_i|=|a'-y_i|=|a-y_i|-\tau_r,\qquad
    |z_j^{(r)}-y_j|=|b'-y_j|=|b-y_j|-\tau_r,
\]
and for $\ell\notin\{i,j\}$, $|z_\ell^{(r)}-y_\ell|=|z_\ell^{(r-1)}-y_\ell|$. Therefore
\begin{equation}\label{eq:l1-drop}
    \|z^{(r)}-\mathbf y\|_1=\|z^{(r-1)}-\mathbf y\|_1-2\tau_r .
\end{equation}
Telescoping \eqref{eq:l1-drop} from $r=1$ to $T_\star$ yields
\[
    0 = \|z^{(T_\star)}-\mathbf y\|_1=\|\mathbf x-\mathbf y\|_1-2\sum_{r=1}^{T_\star}\tau_r,
\]
which proves the identity in~(b).

\medskip
\noindent \emph{Bound on the number of steps.}
Because $\tau_r=\min\{a-y_i, y_j-b\}$, each step fixes at least one of the two coordinates exactly ($z_i^{(r)}=y_i$ or $z_j^{(r)}=y_j$). Fixed coordinates never move again, so after at most $n-1$ steps all coordinates are fixed (the last is forced by equality of sums).
Hence $T_\star\le n-1$.

\subsubsection{Proof of~\prettyref{lem:main_decomposition-shift}}\label{sec:proof_main_decomposition-shift}

For convenience, define the matching gain
$\Delta(x, G) \triangleq M(G\oplus x)-M(G)\in\{0,1\}$,
and the connectivity indicator
\[
    \mathbf{I}(x \sim y, G) \triangleq \mathds{1}\{x\text{ and }y\text{ lie in the same connected component of }G\oplus x\oplus y\}.
\]
We establish the result by first proving deterministic bounds for the augmentation problem with two supply nodes, then extending them to the random setting through expectation.

\begin{claim}[Augmentation bound for two supply nodes]\label{clm:two_aug}
For any two new supply nodes $x, y\in[0,1]\times\mathbb{R}_{\ge0}$,
\begin{align}\label{eq:two-driver-ineq}
    \Delta(x, G)+\Delta(y, G)-\mathbf{I}(x \sim y, G)
    \le
    M(G\oplus x\oplus y)-M(G)
    \le
    \Delta(x, G)+\Delta(y, G).
\end{align}
\end{claim}

We now take expectations of the bounds in~\Cref{clm:two_aug}. For any fixed realization of
$G$, $x$, and $y$, we have
\[
    \Delta(x, G) + \Delta(y, G) - \mathbf{I}(x \sim y, G)
    \le
    M(G \oplus x \oplus y) - M(G)
    \le
    \Delta(x, G) + \Delta(y, G).
\]
Taking expectations over $\driver_x,\driver_y \stackrel{\text{i.i.d.}}{\sim} \mathbb{D}_1 $ and $G \sim \mathbb{G}(m,\servicevec)$, and using linearity of expectation with the definitions of $\delta_m(\service,\servicevec)$ and $\rho_m^{\servicevec}(\service_x, \service_y)$, we obtain
\begin{align*}
    \delta_m(\service_x,\servicevec) + \delta_m(\service_y,\servicevec) - \rho_m^{\servicevec}(\service_x, \service_y)
    &\le \Expect[M(G \oplus x \oplus y)] - \Expect[M(G)]\\
    &= \mu_m(\servicevec  \oplus \service_x  \oplus \service_y) - \mu_m(\servicevec)\\
    &\le \delta_m(\service_x,\servicevec) + \delta_m(\service_y,\servicevec),
\end{align*}
Thus establishing that any service ranges $\servicevec$ and any $\service_x,\service_y \in \reals_{\ge 0}$,
\begin{align}\label{eq:two-driver-ineq-random}
    \deltaM(\service_x,\servicevec) + \deltaM(\service_y,\servicevec) - \rhoS{\servicevec}(\service_x, \service_y)
    & \le \muM(\servicevec  \oplus \service_x  \oplus \service_y) - \muM(\servicevec) \le \deltaM(\service_x,\servicevec) + \deltaM(\service_y,\servicevec).
\end{align}
To compare different range configurations, we apply~\prettyref{eq:two-driver-ineq-random} twice. For adding two supply nodes with service range $(\base_1 + \shift, \base_2 - \shift)$, the lower bound gives
\[ 
    \muM \left( \servicevec \oplus (\base_1 + \shift) \oplus (\base_2 - \shift)\right) - \muM\left(\servicevec\right) \geq \deltaM\left(\base_1 + \shift, \servicevec\right) + \deltaM \left(\base_2 - \shift, \servicevec \right) - \rhoS{\servicevec}\left(\base_1 + \shift, \base_2 - \shift \right).
\]
Similarly, for adding two supply nodes with ranges $(\base_1, \base_2)$, the upper bound gives
\[ 
    \muM \left( \servicevec \oplus \base_1 \oplus \base_2 \right) - \muM\left(\servicevec\right) \leq \deltaM\left(\base_1, \servicevec\right) + \deltaM \left(\base_2, \servicevec \right).
\]
Subtracting the second inequality from the first yields
\begin{align}
    & \muM\left(\servicevec \oplus  (\base_1+\shift) \oplus (\base_2-\shift) \right)-\muM \left(\servicevec \oplus \base_1 \oplus  \base_2 \right) \nonumber \\
    & ~~~~  \geq \deltaM(\base_1+\shift,\servicevec) + \deltaM(\base_2-\shift,\servicevec) - \deltaM(\base_1,\servicevec) - \deltaM(\base_2,\servicevec)
    - \rhoS{\servicevec} \left(\base_1+\shift, \base_2-\shift\right). 
\end{align}

\begin{proof}[Proof of \prettyref{clm:two_aug}]
Fix a maximum matching $\calM$ of $G$ with $|\calM|=M(G)$.
By Berge's lemma,\footnote{An $\calM$-augmenting path is an alternating path (i.e. a path whose edges alternate between matching edges of $\calM$ and non-matching edges), such that its endpoints are unmatched under $\calM$.}
\[
    \Delta(x, G)=1 \iff \text{there exists an $\calM$-augmenting path starting at $x$ in } G\oplus x,
\]
and similarly for $y$.

Let $H \triangleq G\oplus x\oplus y$, and let $\mathcal{P}$ be a maximum family of
pairwise vertex-disjoint $\calM$-augmenting paths in $H$ starting at either $x$ or $y$. Since $G$ differs from $H$ only by edges incident
to $x$ or $y$, and $G$ has no $\calM$-augmenting paths, every $\calM$-augmenting
path in $H$ must include $x$ or $y$ as an endpoint.
Therefore $|\mathcal{P}|\in\{0,1,2\}$ and
\[
    M(H) = M(G)+|\mathcal{P}|.
\]
 \emph{Upper bound.}
Since at most one augmenting path can start at each supply node,
\[
    M(G\oplus x\oplus y)-M(G) = |\mathcal{P}| \le \Delta(x, G)+\Delta(y, G).
\]
\emph{Lower bound.}
If $x$ and $y$ lie in distinct connected components of $H$, their augmenting paths are disjoint, so
$|\mathcal{P}|\ge \Delta(x, G)+\Delta(y, G)$.
If they lie in the same component, the paths may intersect, destroying at most one:
\[
    |\mathcal{P}| \ge \Delta(x, G) + \Delta(y, G) - 1 = \Delta(x, G) + \Delta(y, G) - \mathbf{I}(x\sim y, G).
\]
Combining both cases yields \eqref{eq:two-driver-ineq}.
\end{proof}

\subsubsection{Proof of \prettyref{lem:rarelysamecomponent}}\label{sec:proofoflemmararelysamecomponent}

Partition $[0,1]$ into intervals of length $\ell \triangleq  \frac{2\MPar }{n}$:
\[
    I_i   \triangleq   \left [(i-1) \times \frac{2\MPar }{n}\, , \,  i \times \frac{2\MPar }{n}\right]\cap[0,1], \qquad  \text{ for } i=1,\dots,J, \text{ where } J \triangleq \left\lceil \frac{n}{2\MPar }\right\rceil.
\]
For demand nodes, define $\HasRider_i \triangleq \left\{\riderset\cap I_i\neq\emptyset\right\}$. The events $\HasRider_1,\dots,\HasRider_J$
are $1$–negatively correlated (e.g., Lemma~1.10.26 of~\cite{DOERR18}), hence for any
$a\in\mathbb{N}$ and any run $i_0,\dots,i_0+a$,
\[
    \Prob \left(\bigcap_{t=0}^{a} \HasRider_{i_0+t}\right)
    \le \prod_{t=0}^{a} \Prob(\HasRider_{i_0+t})
    \le \left[1-(1-\smoothness \ell)^{m}\right]^{a}.
\]
A union bound over the $J$ starting positions gives
\[
    \Prob \left(\calX \right) \triangleq \Prob \left(\exists\text{ a run of length }a\text{ with all }\HasRider_i\right) 
    \le J\exp \left(-a(1-\smoothness\ell)^{m}\right).
\]
Since $n\ge 4\MPar\smoothness $, we have $\ell\, \smoothness\le  \frac{1}{2}$ and thus $(1-\smoothness \, \ell)^{m}\ge e^{-2\smoothness \ell m}=e^{- \frac{4 \smoothness\MPar  m}{n}}$. Therefore,
\[
    \Prob(\calX) \leq  \left\lceil \frac{n}{2\MPar }\right\rceil \exp \left(-ae^{- \frac{4 \smoothness  \MPar  m}{n}}\right).
\]
Choose
\[
    a  \triangleq  \left\lceil 2\exp\left( \frac{4 \smoothness\MPar  m}{n}\right)\log n \right\rceil
    \quad\Longrightarrow \quad \Prob(\calX) \le  \frac{1}{2 \MPar n}.
\]
Work on the complement event $\overline{\calX}$. Any segment of length at least $(a + 1)\ell$ contains an $I_i$ with no demand nodes. Because every edge in $G\oplus x\oplus y$ spans distance at most $ \frac{\MPar }{n}$ (all ranges $\le \MPar $), no alternating path can cross a rider–free gap of length $\ell= \frac{2\MPar }{n}$. Hence, on $\overline{\calX}$,
\[
    \mathbf{I} \left(x\sim y  , G\right)=1 \Longrightarrow\ \left|\driver_y-\driver_x\right| \le (a+2)\ell =  \frac{2 \MPar  (a+2)}{n}.
\]
Therefore, using $\Prob \left(|U-V|\le z\right)\le 2\smoothness z$ for independent $U,V\sim\mathbb{D}_1$,
\[
    \Expect \left [\mathbf{I} \left(x\sim y  , G\right)\right] 
    \le
    \Prob \left(\left|\driver_y-\driver_x\right|\le  \frac{2 \MPar  (a+2)}{n}\right)  + \Prob(\calX)
    \le
    \frac{4\smoothness \MPar }{n}(a+2) +  \frac{1}{2 \MPar n},
\]
where the last inequality holds because
\[
    \Prob \left(\left|\driver_y-\driver_x\right|\le  \frac{2 \MPar  (a+2)}{n}\right)  \leq 2\cdot  \frac{2 \smoothness \MPar  (a+2)}{n}   .
\]
By substituting the chosen $a$, the desired result follows.

\subsubsection{Proof of \prettyref{lmm:service_calR_+}}
Fix $G$. A new supply node at location $\driver$ with range $\newrange$ increases the matching if and only if it can reach an unmatched demand node $\rider$:
\[
    \Delta((\driver,\newrange), G)=1 \iff \exists \,  \rider \in\riderset_+(G) \text{ with } |\driver-\rider |\le \newrange/n.
\]
This equivalence follows from Berge's lemma: the new supply node creates an augmenting path of length one precisely when it connects to an unmatched demand node.  {For $\driver \sim \mathbb{D}_1$ with density $g_1$, the success probability equals the corresponding measure of locations that can reach $\riderset_+(G)$, yielding \eqref{eq:service_calR_+}.
}

\subsubsection{Proof of~\prettyref{lem:concavitytwoterms}}\label{sec:proofofconcavitytwoterms}

We decompose the left hand side of~\eqref{eq:shifted_lower_bound_general} as $\termI + \termII$, where
\begin{align*}
    \termI \triangleq \widetilde{\delta}_m(\base_1+\shift,\servicevec) - \widetilde{\delta}_m(\base_1,\servicevec), ~~~~~~~~
    \termII \triangleq \widetilde{\delta}_m(\base_2,\servicevec) - \widetilde{\delta}_m(\base_2-\shift,\servicevec).
\end{align*}
Using the definition~\eqref{eq:forward-window-surrogate} of $\widetilde{\delta}_m$, we may write
\begin{align*} 
    \termI & \!=\! \E\int_0^1 \! \mathbf{1}\Big\{\riderset_+(G)\cap \Big[x, x+ \frac{2 \base_1}{n} \Big] \!=\! \emptyset \Big\} \cdot \mathbf{1} \Big\{ \riderset_+(G)\cap \Big[x+ \frac{2 \base_1}{n}, x+ \frac{2 \base_1+2\shift}{n} \Big] \!\neq\! \emptyset\Big\} \,  g_1(x) \, \diff x, 
    \\
    \termII & \!=\! \E\int_0^1 \! \mathbf{1}\Big\{\riderset_+(G)\cap \Big[x, x+ \frac{2 (\base_2 \!-\! \shift)}{n} \Big] \!=\! \emptyset \Big\} \cdot \mathbf{1} \Big\{\riderset_+(G)\cap \Big[x+ \frac{2( \base_2 \!-\! \shift)}{n}, x+ \frac{2 \base_2}{n} \Big] \! \neq  \! \emptyset \Big\}  \,  g_1(x) \, \diff x.
\end{align*}
We may now decompose $\termI = \termI_A + \termI_B$, with 
\begin{align}
    \termI_A & \! \triangleq \mathbb{E} \! \int_0^1  \! \mathbf{1} \Big\{ \riderset_+(G) \cap \Big[ x -  \frac{\shift}{n},x +  \frac{2 \base_1}{n}\Big] = \emptyset \Big\}  \mathbf{1}\Big\{\riderset_+ (G) \cap \Big[ x +  \frac{2 \base_1}{n}, x + \frac{2 (\base_1 \!+\! \shift) }{n}\Big] \neq \emptyset \Big\}  g_1(x) \, \diff x,  \label{eq:termI_A}\\
    \termI_B & \triangleq \mathbb{E} \int_0^1 \mathbf{1} \left\{ \Gap_\shift(\riderset_+(G),   x) \right\}   g_1(x)  \, \diff x,
\end{align}
where for any $x \in [0,1]$ and any set $\mathcal{L}$ of locations, the event $\Gap_{\shift}(\mathcal{L},x)$ is defined as
\[
    \Gap_\shift(\calL, x) \triangleq \left\{
    \calL\cap [x- \frac{\shift}{n}, x] \neq \emptyset,    
    \calL\cap [x, x+ \frac{2\base_1}{n}] = \emptyset,    
    \calL\cap [x+ \frac{2\base_1}{n}, x+ \frac{2\base_1+2\shift}{n}] \neq \emptyset
    \right\}.
\]
A change of variables \(u = x+ \frac{2(\base_2-\base_1-\shift)}{n}\) in~\eqref{eq:termI_A} yields
\(
    \termII \le \termI_A +  \frac{2(\base_2-\base_1-\shift)}{n} \, .
\)
Therefore,
\begin{align}
    \termI - \termII  \ge \termI_B -  \frac{2(\base_2-\base_1-\shift)}{n}. \label{eq:IminusII-first}
\end{align}

\paragraph{Lower bounding the gap probability.} The key insight is that the pattern $\Gap_{\shift}(\riderset_+(G),x)$ occurs when demand nodes exhibit a specific spatial configuration and nearby supply nodes are absent. Let
\[
    I_x=\left(x,x+ \frac{2 \base_1}{n}\right), ~~~~
    L'_x=\left(x - \frac{\shift}{n},x\right), ~~~~
    R'_x=\left(x+ \frac{2 \base_1}{n},x+ \frac{2 \base_1+\shift}{n}\right).
\]
Define the \emph{inner pattern}
\[
    \InnerGap_\shift(\calL,x) \triangleq  \{\calL\cap L'_x\neq\emptyset, \calL\cap I_x=\emptyset, \calL\cap R'_x\neq\emptyset\},
\]
and the supply node void event
\[
    \NoSupply(x) \triangleq  \left\{\driverset\cap\left(x- \frac{\MPar +\shift}{n}, x+ \frac{2 \base_1+\MPar +\shift}{n}\right)=\emptyset\right\}.
\]
On $\NoSupply(x)$, all demand nodes in $L'_x\cup R'_x$ are unmatched (all ranges $\le \MPar $), so
\[
    \InnerGap_\shift(\riderset, x)   \cap   \NoSupply(x) \subseteq  \Gap_\shift(\riderset_+(G), x).
\]
By independence of supply nodes and demand nodes,
\[
    \termI_B  \ge \int_{0}^{1} \Prob(\InnerGap_\shift(\riderset,x)) \cdot \Prob(\NoSupply(x)) \, \diff x.
\]
Since the void interval for the event $\NoSupply(x)$ has length $\leq  \frac{2\MPar +2 \base_1+2\shift}{n} \leq  \frac{4\MPar }{n}$, for $n\geq 8\MPar \smoothness$,
\begin{align} \label{eq:no-driver-bound}
    \Prob(\NoSupply(x))  \ge \left(1- \frac{4\MPar \smoothness}{n}\right)^n \ge e^{-8\MPar \smoothness}.
\end{align}

\begin{claim}[Pattern probability lower bound]\label{clm:pattern-lb-general}
For $x\in\left(\shift/n, 1-(2 \base_1+\shift)/n\right)$, if $n\ge 4 \, 
\smoothness\,\MPar $, then
\begin{align} \label{eq:inner-gap-bound}
    \Prob(\InnerGap_\shift(\riderset,x))  \geq \frac{7}{48}\left[1-\frac{\smoothness^4+1}{(\smoothness^2+1)^2}\right]e^{-4\base_1\smoothness\betaPar} \, (\MPar\smoothness)^{-3} \shift ^3  \,.  
\end{align}
\end{claim}

Let \(J_\shift\triangleq[\shift/n, 1-(2\base_1+\shift)/n]\). Using the nonnegativity of the integrand and restricting to \(x\in J_\shift\),
\begin{align*}
    \termI_B
    &\ge \int_{J_\shift} \Prob(\InnerGap_\shift(\riderset,x))\cdot \Prob(\NoSupply(x)) \, \diff x \\
    & \ge \frac{7 }{48  }\left[1-\frac{\smoothness^4+1}{(\smoothness^2+1)^2}\right]\left(1- \frac{2 \base_1+2\shift}{n}\right) e^{-8\MPar \smoothness  - 4 \base_1\smoothness\betaPar} (\MPar \smoothness)^{-3}\shift ^3   \,,
\end{align*}
where the last step used~\eqref{eq:no-driver-bound} and~\eqref{eq:inner-gap-bound}. Combining with \eqref{eq:IminusII-first} gives
\begin{align*}
    \termI-\termII
    & \ge  \frac{7 }{48  }\left[1-\frac{\smoothness^4+1}{(\smoothness^2+1)^2}\right] e^{-8\MPar \smoothness  - 4 \base_1\smoothness\betaPar} (\MPar \smoothness)^{-3}\shift ^3 -  \frac{2 \base_1+2\shift}{n} -  \frac{2( \base_2-\base_1-\shift)}{n} \\
    & \ge  \frac{7 }{48  }\left[1-\frac{\smoothness^4+1}{(\smoothness^2+1)^2}\right] e^{-8\MPar \smoothness  - 4 \base_1\smoothness\betaPar} (\MPar \smoothness)^{-3}\shift ^3 -  \frac{2 \base_1+2\shift}{n}  -\frac{2 \base_2}{n}   \\
    & \ge  {8} \, \CMB  \, \shift ^3  -\frac{2 \base_2}{n}   \,,
\end{align*}
where the first inequality holds because $\frac{\shift}{\MPar} \le \frac{1}{2}$ by construction, and the last inequality holds by $\base_1 \le \MPar$ and \eqref{eq:def-CMB-NMB-alphaMB}. Finally, from \eqref{eq:delta_tilde_general} and \eqref{eq:shifted_lower_bound_general},
\begin{align*}
    \deltaM(\base_1+\shift,\servicevec) + \deltaM(\base_2-\shift,\servicevec) - \deltaM(\base_1,\servicevec) - \deltaM(\base_2,\servicevec)
    & \geq \termI-\termII -  \frac{12 \smoothness \MPar}{n} \\
    & \geq  {8} \CMB  \, \tau^3 - \frac{14 \smoothness \MPar}{n}\, .
\end{align*}

\begin{proof}[Proof of~\prettyref{clm:pattern-lb-general}]
Let $I_x=[x,x+2\base_1/n]$. Then 
\[ 
    \Prob(\riderset\cap I_x=\emptyset)\ge(1-2\smoothness \, \base_1/n)^m\ge e^{-4\base_1\smoothness\betaPar}, ~~~~ \mbox{for } n\ge 4\smoothness\MPar \ge 4 \base_1.
\]
On the union $L'_x\cup R'_x$ (total length $2\shift/n$), the count of demand nodes $K\sim \Binom(m,2\shift/n)$ satisfies
\begin{align*}
    \Prob(K\ge 2)
    & \geq 1-\exp\left(-2\, \frac{\betaPar \shift}{\smoothness}\right)-2\,\frac{\betaPar \shift}{\smoothness} \exp\left(-\frac{\betaPar \shift}{\smoothness}\right) \triangleq h\left(\frac{\betaPar \shift}{\smoothness}\right) \ge h\left(\frac{\shift}{\MPar \smoothness}\right) \ge \frac{7}{48 }(\MPar\smoothness)^{-3} \shift ^3  \,, 
\end{align*}
where the second inequality holds because the function $h(u)$ is monotone increasing on $u$ and $\frac{\betaPar \shift}{\smoothness}\ge  \frac{\shift}{\MPar \smoothness}$ given that $\MPar \ge \betaPar^{-1}$, and the last inequality holds because
\[
    e^{-u} \le 1-u+ \frac{u^2}{2}- \frac{u^3}{6}+ \frac{u^4}{24},
    \qquad
    e^{-2u} \le 1-2u+2u^2- \frac{4}{3}u^3+ \frac{2}{3}u^4,
\]
for any $u\ge 0$, we have 
\begin{align*}
    h(u) & \geq \left(1 - \left(1-2u+2u^2- \frac{4}{3}u^3+ \frac{2}{3}u^4\right)\right)
     - 2u\left(1-u+ \frac{u^2}{2}- \frac{u^3}{6}+ \frac{u^4}{24}\right) \\
    &=  \frac{1}{3}u^3 -  \frac{1}{3}u^4 -  \frac{1}{12}u^5 =   u^3\!\left( \frac{1}{3}- \frac{u}{3}- \frac{u^2}{12}\right) \ge \frac{7}{48} u^3 \,,
\end{align*}
given $0\le u\le \frac12 $. 
Conditional on $K\ge 2$, both equal halves are hit with probability at least $1-\frac{\smoothness^4+1}{(\smoothness^2+1)^2}$. Multiplying the bounds gives~\eqref{eq:inner-gap-bound}.

\end{proof}

\subsection{Supplementary proofs for~\texorpdfstring{\Cref{sec-high-dimension}}{}}

\subsubsection{Proof of~\prettyref{lem: high-dim-delta-formula}}
 
In the graph $G_{\varepsilon}$, a new supply node can augment a maximum matching precisely when it connects to a node in $\riderset_+(G_{\varepsilon})$. When demand is distributed uniformly, the success probability exactly equals the Lebesgue measure of locations that can reach $\riderset_+(G_{\varepsilon})$. However, a new supply node $\driver$ with service range $\newrange$, connects to an unmatched demand node $ \rider_j \in \riderset_+(G_{\varepsilon})$ precisely when 
\begin{align} 
 \driver \in B\big(\rider_j, (\ell/n)^{1/\dimension} \big) ~~~~\mbox{ and } ~~~~ \driver \in \tcell(\rider_j). \label{eq:augment-2}
\end{align}
Since this is true for any $\rider_j \in \riderset_+$, we have that 
\begin{align*}
    \deltaM^{\varepsilon}(\newrange, \servicevec) &= \, \mathbb{E}_{G \sim \mathbb{G}(m,\servicevec)} \bigg[ \, \mathbb{E}_{\driver \sim \Uniform([0,1]^\dimension)} \bigg[ M( G_{\varepsilon} \oplus (\driver,\service)) - M(G_{\varepsilon}) \bigg] \bigg]  
    \\
    & = \, \mathbb{E}_{G \sim \mathbb{G}(m,\servicevec)} \bigg[ \, \mathbb{E}_{\driver \sim \Uniform([0,1]^\dimension)} \bigg[ \mathbf{1}\big\{ \exists \, j:  \driver \in B(\rider_j, (\newrange/n)^{1/\dimension}) \cap  \tcell(\rider_j) \big\} \bigg] \bigg]  
    \\
    & \stepa{=} \mathbb{E}_{G \sim \mathbb{G}(m,\servicevec)} \bigg[ \, \mathbb{P}\,\bigg( \bigcup_{\rider_j \in \riderset_+} \left( \driver \in B(\rider_j,(\newrange/n)^{1/\dimension}) \cap \tcell(\rider_j) \right) \bigg) \bigg] 
    \\
    & \stepb{=} \mathbb{E}_{G \sim \mathbb{G}(m,\servicevec)} \bigg[ \, \mathrm{Vol}\, \bigg(\bigcup_{\rider_j \in \riderset_+} \left( B(\rider_j, (\newrange/n)^{1/\dimension}) \cap \tcell(\rider_j) \right) \bigg)\bigg] 
    \\
    & \stepc{=} \mathbb{E}_{G \sim \mathbb{G}(m,\servicevec)} \bigg[ \, \mathrm{Vol}\, \bigg( \bigcup_{\rider_j \in\riderset_+} B_{\newrange}(\rider_j) \bigg) \bigg],
\end{align*}
where (a) uses~\eqref{eq:augment-2}; (b) uses that $\driver$ is uniformly distributed, so the probability measure is precisely the Lebesgue measure $\mathrm{Vol}(\cdot)$; (c) uses the definition of $B_{\newrange}(\cdot)$. This concludes the proof.

\subsubsection{Proof of~\prettyref{lmm:concaverho2}}
Partition $\riderset_{+}$ as $\riderset^{{*}}_+ \cup \riderset^{\overline{*}}_+$, where $\riderset^{{*}}_+$ are the nodes in the special pattern cells, and $\riderset^{\overline{*}}_+$ is its complement, i.e. the set of non-special cells. We therefore have
\begin{align*}
    \sum_{A \in \patterncells\setminus \activepatterncells} \!\!\!\! \text{Vol} \bigg( \bigcup_{ \rider_j \in \riderset_+}  \big( B_{{\newrange}}(\rider_j) \cap A \big) \bigg) & = \!\! \sum_{A \in \patterncells\setminus \activepatterncells} \!\!\!\! \text{Vol} \bigg( \bigcup_{ \rider_{j'} \in \riderset_+^{*}}  \big( B_{{\newrange}}(\rider_{j'}) \cap A \big)   \bigcup_{ \rider_{j} \in \riderset_+^{\overline{*}}}  \big( B_{{\newrange}}(\rider_{j}) \cap A \big) \bigg) 
    \\
    & \stepa{=} \!\! \sum_{A \in \patterncells\setminus \activepatterncells} \!\!\!\! \text{Vol} \bigg( \bigcup_{ \rider_{j} \in \riderset_+^{\overline{*}}}  \big( B_{{\newrange}}(\rider_{j}) \cap A \big) \bigg) 
    \stepb{=} \text{Vol} \bigg( \bigcup_{ \rider_{j} \in \riderset_+^{\overline{*}}} B_{{\newrange}}(\rider_{j})  \bigg) \,, 
\end{align*}
where given ${\newrange} \in [0,\MPar]$, (a) follows from~\ref{prop:insidecell}, which implies that $B_{{\newrange}}(\rider_{j'}) \cap A = \emptyset$ for any $\rider_{j'}$ in a special cell, and for any non-special pattern cell $A$; (b) follows from~\ref{prop:distance}, since any $\rider_j$ outside a special pattern cell satisfies $B_{{\newrange}}(\rider_j) \cap A' = \emptyset$ for any special pattern cell $A'$. 

It is convenient to write out the set in the last expression as a disjoint union, by using the Voronoi cells\footnote{For a ground set $X$ and a seed set $S \subseteq X$, the Voronoi cell of a point $s\in S$ is the set of all points in $X$ that are closer to $s$ than any other point of $S$.} of $\riderset_+^{\overline{*}}$. Let $V_{\rider_j}(\riderset_+^{\overline{*}})$ denote the Voronoi cell of $\rider_j$. Then 
\begin{align}
    \text{Vol}\bigg(
    \bigcup_{\substack{\rider_j \in \riderset^{\overline{*}}_+}}
     B_{\newrange}(\rider_j)
    \bigg)
    =
    \sum_{\rider_j \in \riderset^{\overline{*}}_+} f_j({\newrange}), ~~~~ \mbox{ where } f_j({\newrange}) \triangleq \text{Vol}\left(
     B_{\newrange}(\rider_j)\cap V_{\rider_j}(\riderset^{\overline{*}}_+)\right) \, . \label{eq:voronoi-decomposition}
\end{align}
To see why~\eqref{eq:voronoi-decomposition} holds, note that:
\begin{enumerate}
    \item[(i)] The sets $\big( B_{\newrange}(\rider_j) \cap V_{\rider_j}(\riderset^{\overline{*}}_+)\big)_{j}$ are pairwise disjoint, by definition of the Voronoi cells.
    \item[(ii)] For any distinct $\rider_i$ and $\rider_j$ in $\riderset^{\overline{*}}_+$, and any point $y \in B_{\newrange}(\rider_i) \cap V_{\rider_j}(\riderset^{\overline{*}}_+)$, we have
    \[ 
    \| y - \rider_j\|_2 \leq \|y - \rider_i \|_2 \leq (\newrange/n)^{1/\dimension},
    \]
    i.e. inside $V_{\rider_j}(\riderset^{\overline{*}}_+)$, being in any $B(\rider_i, (\newrange/n)^{1/\dimension})$ implies being in $B \left(\rider_j, (\newrange/n)^{1/\dimension} \right)$.
\end{enumerate}
Together, (i) and (ii) imply that the sets $\big( B_{\newrange}(\rider_j) \cap V_{\rider_j}(\riderset^{\overline{*}}_+)\big)_{j}$ partition $\medcup_{\rider_j \in \riderset_+^{\overline{*}}} B_{\newrange}(\rider_j)$, implying~\eqref{eq:voronoi-decomposition}.

It suffices to show that each term $f_j({\newrange})$ is concave, for any realization $\riderset_+^{\overline{*}}$. Recall that $\tcell(\rider_j)$ is the trimming cell containing $\rider_j$, and set
\[
    W  =  \left(\tcell(\rider_j)\,\cap\, V_{\rider_j}(\riderset^{\overline{*}}_+)\right).
\]
$W$ is the intersection of two convex sets, and is itself convex. Without loss of generality, we may translate space so that $\rider_j = 0$, and so
\[
    f_{j}({\newrange}) = \mathrm{Vol}\left( B\left(0, ({\newrange}/n)^{1/\dimension}\right) \cap W \right).
\]
Let $\mathcal{H}^{\dimension-1}$ denote the surface area (Hausdorff) measure in dimension $\dimension-1$. 
By the coarea formula in polar coordinates,
\[
    \mathrm{Vol}\left( B\left(0, ({\newrange}/n)^{1/\dimension}\right) \cap W
    \right)
    =
    \int_0^{({\newrange}/n)^{1/\dimension}} \mathcal{H}^{\dimension-1}(S_s \cap W)\, \diff s,
\]
where $S_s = \partial B(0,s)$. A change of variable $s = (t/n)^{1/\dimension}$ yields
\[
    f_{j}({\newrange}) = \int_0^{\newrange} \frac{1}{\dimension\,n^{1/\dimension}} \,
       t^{-\frac{\dimension-1}{\dimension}}\,
       \mathcal{H}^{\dimension-1}\left(S_{(t/n)^{1/\dimension}} \cap W\right)\, \diff t.
\]
Hence, for almost every ${\newrange}>0$,
\[
    f_{j}'({\newrange}) = \frac{1}{\dimension\,n^{1/\dimension}}\, {\newrange}^{-\frac{\dimension-1}{\dimension}}\,
     \mathcal{H}^{\dimension-1}\left(S_{({\newrange}/n)^{1/\dimension}} \cap W\right).
\]
In polar coordinates, let $\mathbb{S}^{\dimension-1}$ denote the unit sphere with surface measure $\sigma$.   For $\omega \in \mathbb{S}^{\dimension-1}$, define the stopping radius $\tau(\omega) \in [0,\infty]$ as
\[
    \tau(\omega) \triangleq  \sup\{ t \ge 0 : t\omega \in W \} \,,
\]
so that 
\[
    \mathcal{H}^{\dimension-1}\left(S_{({\newrange}/n)^{1/\dimension}} \cap W\right)
   = \left(\frac{{\newrange}}{n}\right)^{\frac{\dimension-1}{\dimension}}\cdot 
     \sigma\bigl\{ \omega : \tau(\omega) \ge ({\newrange}/n)^{1/\dimension} \bigr\}.
\]
Therefore, almost everywhere,
\[
    f_{j}'({\newrange}) = \frac{1}{\dimension\,n}\, \sigma\bigl\{ \omega : \tau(\omega) \ge ({\newrange}/n)^{1/\dimension} \bigr\}.
\]
Since $W$ is convex and $0 \in W$, the set $\{\omega : \tau(\omega) \ge t\}$ decreases monotonically in $t$.  
Thus $f_{j}'({\newrange})$ is nonincreasing in ${\newrange}$, which implies that $f_{j}$ is concave.  Summing over $\rider_j$ and taking the expectation with respect to $\riderset_+$ establishes the concavity of $\rho_{1}({\newrange})$.

\subsubsection{Proof of~\texorpdfstring{\Cref{lmm:gainonepattern}}{}}
Fix $\riderset_+$ and suppose $A\in\activepatterncells$, so that there exist two points $x_0\in \riderset_+\cap B(x_A,R)$ and
$x_1\in \riderset_+\cap\big(B(x_A,R'')\setminus B(x_A,R')\big)$.
Fix any $\ell\in[0,\MPar]$ and write $r(\ell)\triangleq (\ell/n)^{1/\dimension}$.
By~\ref{prop:insidecell}, we have
$B(x_A,R''+(\MPar/n)^{1/\dimension})\subseteq A$ and $A\subseteq \tcell(x_A)$.
Since $x_0\in B(x_A,R)$ and $x_1\in B(x_A,R'')$, it follows that
\[
    B\big(x_0,r(\ell)\big)\subseteq A \, , ~~~~  B\big(x_1,r(\ell)\big)\subseteq A \, ,
\]
and in particular both balls lie in the same trimming cell. Hence, by the definition of $B_\ell(\cdot)$,
\[
    B_\ell(x_0)\cap A = B\big(x_0,r(\ell)\big) \, ,  ~~~~ B_\ell(x_1)\cap A = B\big(x_1,r(\ell)\big) \, .
\]
Moreover, if $\rider_j\in \riderset_+\setminus\{x_0,x_1\}$ then $\rider_j\notin B(x_A,R''')$. This follows from~\ref{prop:distance}, since 
$\mathrm{dist}(\rider_j,A)>2(\MPar/n)^{1/\dimension}\ge 2r(\ell)$, and hence
$B_\ell(\rider_j)\cap A=\emptyset$. Therefore
\[
    \chi_A(\ell,\riderset_+)
    =\mathrm{Vol}\Big(\big(B(x_0,r(\ell))\cup B(x_1,r(\ell))\big)\Big)
    =2\,\kappa_\dimension\,\frac{\ell}{n}-I(\ell;x_0,x_1),
\]
where $I(\ell;x_0,x_1)$ is as in~\ref{prop:no_intersection}. We therefore have
\begin{align}\label{eq:chi-gap-via-I}
    &\big(\chi_A (\base_1+\shift,\riderset_+ ) - \chi_A(\base_1,\riderset_+) \big)
    -
    \big(\chi_A(\base_2,\riderset_+) - \chi_A (\base_2-\shift,\riderset_+ )\big)
    \nonumber\\
    & ~~~~ =
    \Big(I(\base_2;x_0,x_1)-I(\base_2-\shift;x_0,x_1)\Big)
    -
    \Big(I(\base_1+\shift;x_0,x_1)-I(\base_1;x_0,x_1)\Big).
\end{align}
Since $\shift\le (\base_2-\base_1)/2$, we have $\base_1\le \base_2-2\shift$.
By convexity of $\ell\mapsto I(\ell;x_0,x_1)$, we have
\[
    I(\base_1+\shift;x_0,x_1)-I(\base_1;x_0,x_1)
    \, \le\,
    I(\base_2-\shift;x_0,x_1)-I(\base_2-2\shift;x_0,x_1).
\]
Plugging into~\eqref{eq:chi-gap-via-I} yields
\[
\eqref{eq:chi-gap-via-I}
\ \ge\
I(\base_2;x_0,x_1)-2I(\base_2-\shift;x_0,x_1)+I(\base_2-2\shift;x_0,x_1).
\]
Finally, applying~\eqref{eq:second-difference} in~\ref{prop:no_intersection} gives
\[
    I(\base_2)-2I(\base_2-\shift)+I(\base_2-2\shift)
    \, \ge\,
    \frac{\alphagainpattern}{n}\,\shift^{\,2}.
\]
This completes the proof.

\subsubsection{Proof of \prettyref{lmm:lowerboundpatterna}}

Recall that trimming cells have side length $L_T = {2\dimension \MPar^{1/\dimension}}/({\varepsilon n^{1/\dimension}})$, whereas pattern cells have side length $L_P = \patternside/{n^{1/\dimension}}$. {In order for half the pattern cells to be fully contained in trimming cells, it suffices that \( L_T \geq 2\dimension \, L_P\), or equivalently that 
\begin{align}
    \varepsilon < {\epsiloncondition}, ~~~~ \mbox{ where } \epsiloncondition \triangleq \frac{\MPar^{1/\dimension}}{\patternside} \, ,
\end{align}  
Therefore, under this condition,
\begin{align}
\sum_{A \in \patterncells} 
    \mathds{1}\bigl\{ A \subseteq {\tcell}(x_A) \bigr\} 
    \, \ge \, \frac{|\patterncells|}{2} 
    \, \ge \, \frac{1}{2} \Biggl\lfloor \frac{n^{1/\dimension}}{\patternside} \Biggr\rfloor^\dimension \, . \label{eq:lowerboundcellnumber}
\end{align}
For any pattern cell $A$ that is contained in a trimming cell $\tcell(x_A)$, we lower bound the probability that $A$ is a special pattern cell. By definition, this event requires that exactly one $\riderset_+$ node appears in $B(x_A,R)$, exactly one $\riderset_+$ node appears in $B(x_A,R'')\setminus B(x_A,R')$, and none of the other $\riderset_+$ nodes appear in $B(x_A,R''')$. A sufficient condition for that is:
\begin{enumerate}
    \item[(i)] Exactly one of the $m$ demand nodes is in $B(x_A,R)$ (probability $m\cdot \mathrm{Vol}(B(x_A,R))$)
    \item[(ii)] Exactly one of the remaining $m-1$ demand nodes is in $B(x_A, R'') \setminus B(x_A, R')$ (probability $(m-1)\cdot \mathrm{Vol}(\, B(x_A,R'') \setminus B(x_A,R')\,)$).
    \item[(iii)] None of the other $m-2$ demand nodes are in $B(x_A,R''')$ (probability $(1-\mathrm{Vol}(B(x_A,R''')))^{m-2}$).
    \item[(iv)] None of the $n$ supply nodes appear in $B(x_A,R''')$ (probability \allowbreak $(1-\mathrm{Vol}(B(x_A,R''')))^{n}$).
\end{enumerate}
Combining the above 4 events, we have that conditional on $A$ being a special pattern cell,
\[
    \mathbb{P}_{G \sim \mathbb{G}(m,\servicevec)}\Big[A \in \activepatterncells) \, \big| A \subseteq \tcell (x_A)\Big] \geq 2 \,  \tau^2 \, \binom{m}{2}\,
       \left(\frac{2\volsphere}{10^\dimension n}\right)^2 \times
   \bigl(1-\volsphere \cdot (R''')^\dimension\bigr)^{m+n-2},
\]
where the inequality uses~\ref{prop:volumes} to lower bound $\mathrm{Vol}(B(x_A,R))$ and $\mathrm{Vol}(B(x_A, R'') \setminus B(x_A, R'))$, and also uses the fact that $\volsphere$ denotes the volume of a unit ball in $\reals^{\dimension}$.
Therefore, combining with~\eqref{eq:lowerboundcellnumber}, we get that if $\varepsilon\leq \epsiloncondition$, we have:
\[
    \sum_{A \in \patterncells}
    \mathbb{P}_{G \sim \mathbb{G}(m,\servicevec)}\left[A \in \activepatterncells\right]
    \ge   2\tau^2 \, \binom{m}{2}\,
    \left(\frac{2\volsphere}{10^\dimension n}\right)^2\,
    \bigl(1-\volsphere \cdot (R''')^\dimension\bigr)^{m+n-2}\Biggl\lfloor \frac{n^{1/\dimension}}{\patternside} \Biggr\rfloor^\dimension.
\]
On the other hand, by definition:
\[
    R''' =\frac{1}{n^{1/\dimension}} \left(\sqrt{\dimension}\patternside+2\MPar^{1/\dimension}\right).
\]
Therefore, there exist two constants $\highdimN$ and $\alphaproba>0$ such that for any $n> \highdimN$, we have:
\[
    \sum_{A \in \patterncells} \mathbb{P}_{G \sim \mathbb{G}(m,\servicevec)}\left[A \in \activepatterncells\right]\geq \tau^2\alphaproba n.
\]
whenever $\varepsilon\leq \epsiloncondition$.

\subsubsection{Proof of \prettyref{lem:constantdef}} \label{sec-proof-of-properties}
We prove each item separately.

\begin{itemize}
    \item \textbf{Proof of~\ref{prop:volumes}}: The equality of the property holds by the definitions of $R$, $R'$ and $R''$. The inequality follows from $r_2\geq 2 \shift$.

    \item \textbf{Proof of~\ref{prop:no_intersection}.}
    Fix $x_0\in B(x_A,R)$ and $x_1\in B(x_A,R'')\setminus B(x_A,R')$, and set
    $a\triangleq \|x_0-x_1\|_2$.
    Recall for $\ell\geq 0$ that $I(\ell; x_0,x_1)$ is defined as
    \[
        I(\ell;x_0,x_1)\triangleq \mathrm{Vol}\Big( B\big(x_0,r(\ell) \big)\,\medcap\,B\big(x_1,r(\ell) \big)\Big),
        ~~~~ \mbox{where} ~~~~
        r(\ell)\triangleq (\ell/n)^{1/\dimension}.
    \]
    If $a \geq 2\,r(\ell)$, the two balls are disjoint and so $I(\ell;x_0,x_1)=0$. On the other hand, if $a < 2\, r(\ell)$, the intersection volume is given by $\frac{\ell}{n} \varphi_\dimension\big( a/r(\ell)\big)$, where $\varphi_\dimension(s)$ is the intersection volume of two unit balls in $\mathbb{R}^\dimension$ whose centers are a distance $s$ apart. Mathematically,
    \begin{align}\label{eq:scale-I-proof}
        I(\ell;x_0,x_1)=\frac{\ell}{n}\,\varphi_{\dimension}\big(s(\ell)\big) ~~~~\mbox{ where } ~~ s(\ell)\triangleq \frac{a}{r(\ell)}
    \end{align}
    and
    \[
        \varphi_{\dimension}(s) = 
        \begin{cases} 
            2\,\kappa_{\dimension-1}\int_{s/2}^{1}\big(1-t^2\big)^{\frac{\dimension-1}{2}}\,\diff t \,, & 0 \leq s < 2
            \\
            0\,, & s \geq 2
        \end{cases} \, .
    \]
    We show convexity of the function $\ell \mapsto I(\ell;x_0,x_1)$ by showing that (i) $I'(\ell;x_0,x_1)$ is absolutely continuous and (ii) $I''(\ell; x_0,x_1) \geq 0$ almost everywhere. 
    \begin{itemize}
        \item[(i)] For the first derivative, note that 
        \begin{align*}
            s'(\ell)= -\frac{1}{k}\frac{s(\ell)}{\ell} 
            ~~~~ \mbox{ and } ~~~~
            \varphi_{\dimension}'(s)= \begin{cases} 
                -\kappa_{\dimension-1}\Big(1-\frac{s^2}{4}\Big)^{\frac{\dimension-1}{2}}, & 0 \leq s < 2 \\
                0 \, , & s \geq 2
            \end{cases} 
            \, .
        \end{align*}
        In particular, $\varphi_{\dimension}'(s)$ is absolutely continuous in $s$ for all $k\geq 2$, and so 
        \[
            I'(\ell;x_0,x_1) = \frac 1 n \big( \varphi_{\dimension}\big(s(\ell)\big) - \frac{s(\ell)}{\dimension} \,\varphi'_{\dimension}(s(\ell))\big) 
        \]
        is absolutely continuous in $\ell$ for all $k \geq 2$.
        \item[(ii)] For the second derivative, note that
        \begin{align*}
            s''(\ell)=\frac{k+1}{k^2}\frac{s(\ell)}{\ell^2} \, 
            ~~~~ \mbox{ and } ~~~~
            \varphi_{\dimension}''(s)
            =
            \begin{cases}
                \kappa_{\dimension-1}\,\frac{\dimension-1}{4}\,s\,\Big(1-\frac{s^2}{4}\Big)^{\frac{\dimension-3}{2}} \, & 0 \leq s < 2
                \\
                0 \, , & s \geq 2
            \end{cases}
            \,
        \end{align*}
        and so we have
        \begin{align} \label{eq:second-diff}
            I''(\ell; x_0,x_1) = 
            \frac{1}{n}\cdot \frac{s(\ell)}{\dimension^2\,\ell}\Big(
            s(\ell)\,\varphi_\dimension''(s(\ell))
            -(\dimension-1)\,\varphi_\dimension'(s(\ell))
            \Big) \, .
    \end{align}
    Substituting the expressions for $\varphi_{\dimension}'$ and $\varphi_{\dimension}''$ in~\eqref{eq:second-diff} and using $s(\ell) = d/(\ell/n)^{1/\dimension}$ yields
    \begin{align}\label{eq:I-second-deriv-proof}
        I''(\ell;x_0,x_1)
        = 
        \begin{cases}
            \frac{\kappa_{\dimension-1}(\dimension-1)}{\dimension^2}\cdot \frac{1}{n}\cdot \frac{s(\ell)}{\ell}\cdot
            \Big(1-\frac{s(\ell)^2}{4}\Big)^{\frac{\dimension-3}{2}} \, , & \ell \geq n \big( \frac{d}{2} \big)^{\dimension} \\
            0 \, , & 0 \leq \ell < n\big( d/2 \big)^{\dimension} 
        \end{cases}
        \,.
    \end{align}
    In particular, $I''(\ell;x_0,x_1) \geq 0$ almost everywhere. 
    \end{itemize}
    This establishes that $I(\ell;x_0,x_1)$ is a convex function of $\ell$.

    \medskip
    Next, we lower bound
    \(
        I(\base_2;x_0,x_1)-2I(\base_2-\shift;x_0,x_1)+I(\base_2-2\shift;x_0,x_1).
    \)
    Since $I'(\ell;x_0,x_1)$ is continuous and $I''(\ell;x_0,x_1)$ is defined almost everywhere, we have
    \[
        I(\base_2)-2I(\base_2-\shift)+I(\base_2-2\shift)
        =\int_0^{\shift}\int_0^{\shift} I''(\base_2-u-v)\,\diff u\,\diff v,
    \]
    where $I''(\ell; x_0,x_1)$ is $0$ whenever the two balls are disjoint (i.e. $d\ge 2r(\ell)$).
    Restricting the integral to $u,v\in[0,\shift/4]$ (so that $\base_2-u-v\in[\base_2-\shift/2,\base_2]$), we have
    \begin{align}
        I(\base_2)-2I(\base_2-\shift)+I(\base_2-2\shift) \nonumber
        &\geq \int_0^{\shift/4}\int_0^{\shift/4} I''(\base_2-u-v)\,\diff u\,\diff v\\
        &\ge \frac{\shift^2}{16}\cdot \inf_{\ell\in[\base_2-\shift/2,\base_2]} I''(\ell;x_0,x_1) \, . \label{eq:inf}
    \end{align}
    It remains to lower bound this infimum. To that end, we claim the following.
    \begin{claim} \label{clm:inf}
        There exists a constant $C_{\dimension, \MPar} > 0$ depending only on $\dimension$ and $\MPar$ such that
        \begin{align}
        \inf_{\ell\in[\base_2-\shift/2,\base_2]} I''(\ell;x_0,x_1)
        \geq C_{\dimension, \MPar}/n \,. \label{eq:Ipp-lower-proof}
    \end{align}
    \end{claim}
     Substituting~\eqref{eq:Ipp-lower-proof} in~\eqref{eq:inf}, we have
    \begin{align*}
        I(\base_2)-2I(\base_2-\shift)+I(\base_2-2\shift)
        \ge \frac{\alphagainpattern}{n}\,\shift^{\,2},
    \end{align*}
    for a constant $\alphagainpattern>0$ depending only on $\MPar$ and $\dimension$. This concludes the proof.

    \item \textbf{Proof of~\ref{prop:insidecell}.} 
    We first show that
    \(
        B\left(x_A,\,R'' + (\MPar/n)^{1/\dimension}\right)\subseteq A \, .
    \)
    Recall that the pattern cell \(A\) is a hypercube of side length
    \(
        \frac{\patternside}{n^{1/\dimension}},
    \)
    centered at \(x_A\). It thus suffices to show that
    \begin{align}\label{eq:P3-suff}
        R'' + (\MPar/n)^{1/\dimension} \le \frac{\patternside}{2n^{1/\dimension}}.
    \end{align}
    By definition,
    \[
        R''=\left(\frac{r_2}{n}\right)^{1/\dimension} \le \left(\frac{\gamma}{n}\right)^{1/\dimension}
        \qquad \text{and}\qquad
        \patternside=6\,\gamma^{1/\dimension} \, ,
    \]
    so~\eqref{eq:P3-suff} holds, which proves
    \(
        B\left(x_A,\,R'' + (\MPar/n)^{1/\dimension}\right)\subseteq A \, .
    \)
    
    For the second inclusion, for any \( a\in A \) we have
    \[
        \|a-x_A\|_2 \le \frac{\sqrt{\dimension}}{2}\,\frac{\patternside}{n^{1/\dimension}} \, .
    \]
    By the definition of \( R''' \) (chosen larger than the half-diagonal of \(A\)),
    this implies \( \|a-x_A\|_2 < R''' \), and hence \( A\subseteq B(x_A,R''') \).
    \item \textbf{Proof of~\ref{prop:distance}.} This property holds by definition of \(R'''\). 
\end{itemize}
This completes the proof of all four properties.

\begin{proof}[Proof of~\texorpdfstring{\Cref{clm:inf}}{}]
Recall the expression~\eqref{eq:I-second-deriv-proof} for \( I''(\cdot) \),
\begin{align}\label{eq:I-second-deriv-proof-repeat}
        I''(\ell;x_0,x_1)
        =
        \frac{\kappa_{\dimension-1}(\dimension-1)}{\dimension^2}\cdot \frac{1}{n}\cdot \frac{s(\ell)}{\ell}\cdot
        \Big(1-\frac{s(\ell)^2}{4}\Big)^{\frac{\dimension-3}{2}},
        \qquad\text{for }s(\ell) \in (0,2).
\end{align}
To lower bound~\eqref{eq:I-second-deriv-proof-repeat}, we derive upper and lower bounds for \( s(\ell) \) below.

\medskip 
\noindent \textit{(Lower bounding \( s(\ell) \)).~} First, by construction, we have
\[
        a \,\ge\, R'-R \,=\,\Big(\Big(1-\frac{1}{10^{\dimension}}\Big)^{1/k}-\frac{1}{10}\Big)(\base_2/n)^{1/\dimension} \, .
\]
On the other hand, \( r(\ell) \le (\base_2/n)^{1/\dimension} \), hence,
\begin{align}\label{eq:s-lower-proof}
    s(\ell)=\frac{a}{r(\ell)} \geq \,\Big(\Big(1-\frac{1}{10^{\dimension}}\Big)^{1/\dimension}-\frac{1}{10}\Big)\geq 0.8 \, .
\end{align}

\noindent \textit{(Upper bounding \( s(\ell) \) ).~} Again by construction, we have
\[
    a \le R''+R \leq \Big(1+\frac{1}{10}\Big)(\base_2/n)^{1/\dimension} \, .
\]
On the other hand, $r(\ell) \geq \left(\left(\base_2-\tau/2\right)/n\right)^{1/\dimension}$, hence,
\begin{align}\label{eq:s-upper-proof}
    s(\ell)=\frac{a}{r(\ell)} \leq \,\frac{\left(1+\frac{1}{10}\right)r_2^{1/\dimension}}{\left(\base_2-\tau/2\right)^{1/\dimension}}\leq \frac{\left(1+\frac{1}{10}\right)}{\left(3/4\right)^{1/\dimension}} \leq \frac{\left(1+\frac{1}{10}\right)}{\left(3/4\right)^{1/2}}\leq 1.96 .
\end{align}
Combining those bounds, for any \( \ell \in [r_2-\tau/2, \, r_2] \) and \( k\geq 3 \):
\[
    \inf_{\ell\in[\base_2-\shift/2,\base_2]} I''(\ell;x_0,x_1)
    \geq 
    \frac{\kappa_{\dimension-1}(\dimension-1)}{\dimension^2}\cdot \frac{1}{n}\cdot \frac{0.8}{\gamma}\cdot
    \Big(1-\frac{1.96^2}{4}\Big)^{\frac{\dimension-3}{2}} \, .  
\]
For any \( \ell \in [r_2-\tau/2, \, r_2] \) and \( k=2 \):
\[
 \inf_{\ell\in[\base_2-\shift/2,\base_2]} I''(\ell;x_0,x_1)
        \geq 
        \frac{\kappa_{\dimension-1}(\dimension-1)}{\dimension^2}\cdot \frac{1}{n}\cdot \frac{0.8}{\gamma}\cdot
        \Big(1-\frac{0.8^2}{4}\Big)^{\frac{-1}{2}}. 
\]
In both cases, there is a constant \( C_{\dimension,\MPar} > 0 \) such that \( \inf_{\ell\in[\base_2-\shift/2,\base_2]} I''(\ell;x_0,x_1) \geq C_{\dimension,\MPar}/n \) .
\end{proof}

\section{Supplementary proofs for the dual service range model} \label{apx-deferred-proofs-for-dual-service}

In this section, we present the proofs for various lemmas used in the analysis of the dual service range model. 

\subsection{Supplementary proofs for \prettyref{sec:dual_embedding}}

\subsubsection{Proof of~\texorpdfstring{\Cref{lem: greedy-is-optimal}}{}}
Optimality of the greedy algorithm for interval graphs is a classical result~\cite{glover1967convex, lipskipreparata1981}. For completeness we present a simple proof specialized to our random geometric graph. For any matching $\calN$ and non-negative integer $t$, let $\calN_t \triangleq \{ (\rider_{(i)}, \driver_{i^*}) \in \calN : i \leq t\}$ be the restriction of $\calN$ to the first $t$ demand nodes. We show via induction that for all $t$, there exists a maximum matching $\calM^*$ such that $\calM_t = \calM^*_t$. The base case is true, since $\calM_0 = \calM^*_0 = \emptyset$ for all maximum matchings $\calM^*$. Thus, assume $t \geq 1$ and $\calM_{t-1} = \calM^*_{t-1}$ for a maximum matching $\calM^*$. 

\medskip

\noindent \textit{Case $1$.} $\rider_{(t)}$ is unmatched by $\calM$. This happens only if $\rider_{(t)}$ has no neighbors in $\drivervec'_t$, the set of unmatched supply nodes in iteration $t$ of Algorithm~\ref{alg: greedy}. In this case, $\calM^*$ cannot match $\rider_{(t)}$, so $\calM_t = \calM^*_t$.

\medskip

\noindent \textit{Case $2$.} $\rider_{(t)}$ is matched to $\driver_{t^*}$ by $\calM$. There are two subcases to consider.
\begin{itemize}
    \item $\rider_{(t)}$ is unmatched by $\calM^*$. It must be that for some $\ell > t$,  $\calM^*$ matches $\rider_{(\ell)}$ to $\driver_{t^*}$, otherwise $\calM^* + (\rider_{(t)}, \driver_{t^*})$ would be a larger matching. Then, the matching 
    \[
        \widetilde{\calM} \triangleq \calM^* - (\rider_{(\ell)}, \driver_{t^*}) + (\rider_{(t)}, \driver_{t^*})
    \]
    is maximum and satisfies $\calM_t = \widetilde{\calM}_t$.
    
    \item $\rider_{(t)}$ is matched by $\calM^*$. If it is matched to $\driver_{t^*}$, then $\calM_t = \calM^*_t$ is trivially satisfied. Therefore, assume that $\rider_{(t)}$ is matched by $\calM^*$ to some $\driver_{(\ell)} \neq \driver_{t^*}$. If $\driver_{t^*}$ is unmatched by $\calM^*$, let $\widetilde{\calM}$ be the matching
    \[
        \widetilde{\calM} \triangleq \calM^* - (\rider_{(t)},\driver_{(\ell)}) + (\rider_{(t)},\driver_{t^*}).
    \]
    Otherwise, if $\driver_{(\ell)}$ is matched to some $\rider_{(j)}$ with $j > t$, let $\widetilde{\calM}$ be the matching
    \[ 
        \widetilde{\calM} \triangleq \calM^* - (\rider_{(t)},\driver_{(\ell)}) - (\rider_{(j), }\driver_{t^*}) + (\rider_{(t)}, \driver_{t^*}) + (\rider_{(j)}, \driver_{(\ell)}).
    \]
    To see why $(\rider_{(j)}, \driver_{(\ell)})$ is a valid match, note that $\rider_{(t)}$ connects to both $\driver_{t^*}$ and $\driver_{(\ell)}$. Since $\rider_{(j)}$ comes after $\rider_{(t)}$ and connects to $\driver_{t^*}$, it must also connect to $\driver_{(\ell)}$ because $\driver_{t^*} + \service_{t^*} \leq \driver_{(\ell)} + \service_{(\ell)}$ per the matching rule. Thus,  $\widetilde{\calM}$ is a maximum matching satisfying $\calM_t = \widetilde{\calM}_t$. 
\end{itemize}
This completes the proof.

\subsubsection{Proof of Lemma~\ref{lem:agreement}}
Assume that $\calM(\widehat G_n)$ and $\widehat \calM(\widehat G_n)$ disagree, and let $t\ge 1$ be the smallest index at which $\rider_{(t)}$ is resolved differently between the two procedures. By minimality of $t$, the two runs make identical decisions for $\rider_{(1)},\ldots,\rider_{(t-1)}$; hence the sets of \emph{unmatched} flexible and non‑flexible supply nodes lying to the left of $\rider_{(t)}$ coincide in both runs.

At the moment, the generative procedure is ready to decide on $\rider_{(t)}$, each supply node type $T\in\{\mathrm{F},\mathrm{NF}\}$ has been advanced until it is either in range or ahead of the demand node. Formally, letting $v^T$ be that type's active supply node and $s_T/n$ its radius, the type status is one of:
\[
\text{(in range)}\quad \rider_{(t)}\in\bigl[v^T{-}s_T/n,\;v^T{+}s_T/n\bigr],
\qquad \mbox{or} \qquad 
\text{(ahead)}\quad v^T{-}s_T/n>\rider_{(t)}.
\]
Because we are revealing the same PPP points left‑to‑right, the \emph{active} supply node of each type in the generative run is exactly the \emph{leftmost unmatched} supply node of that type in the greedy run, and its status (in‑range vs.\ ahead) with respect to $\rider_{(t)}$ is identical in both runs. There are three possibilities:

\medskip
\begin{itemize}
    \item[(i)] \textit{Both supply types ahead}. No supply node can reach $\rider_{(t)}$, so both procedures leave $\rider_{(t)}$ unmatched.
    \item[(ii)] \textit{Exactly one supply type $T$ in range.} The other type is ahead and cannot serve $\rider_{(t)}$. Within type $T$, all supply nodes have the same range, and the leftmost unmatched feasible supply node (the active $T$ just identified) has the earliest deadline $v^T+s_T/n$ among feasible $T$-supply nodes. Hence both procedures must match $\rider_{(t)}$ to that same supply node.
    \item[(iii)] \textit{Both supply types in range.} Each procedure compares the same two deadlines
    $D^{\mathrm{F}}=v^{\mathrm{F}}+(\base+\extra)/{n}$ 
    and $D^{\mathrm{NF}}=v^{\mathrm{NF}}+\base/{n},
    $
    and, by the greedy rule, matches $\rider_{(t)}$ to the earlier deadline.
\end{itemize}
In all cases, the decisions coincide, which contradicts the definition of $t$. Therefore, no disagreement occurs, concluding the proof.

\subsubsection{Proof of~\texorpdfstring{\Cref{lem:small}}{}}
Throughout, $u_t \stackrel{\mathrm{iid}}{\sim}\Exp(p)$, $v_t\stackrel{\mathrm{iid}}{\sim}\Exp(1-p)$ and $w_t\stackrel{\mathrm{iid}}{\sim}\Exp(1)$ are independent, and the one–step increment $\nabla\potential(t)=(\Delta_x,\Delta_y)$ is given by the five cases (A)–(E). We write $x_+=\max\{x,0\}$ and $x_-=\max\{-x,0\}$. We denote the Lebesgue measure on $\reals^2$ by $\mathrm{Vol}$.

\medskip 

\noindent (i) Define the rectangle $K_0 \triangleq [d, 2d] \times [-d, 0]$. Since $K_0 \subseteq \calE$, the one step kernel from $(x,y) \in K_0$ to any $(x_1,y_1) \in \reals^2$ is
\[ 
f_{x,y}(x_1,y_1) = p \, \exp\left( - (y_1-y) -p(x+y_1-y-x_1)\right) \, \mathbf{1}_{ \{y_1 > y \}} \mathbf{1}_{ \{x_1 - x < y_1 - y \}}.
\]
$K_0$ and $R$ are located such that there is a positive constant $c_0$ for which $f_{x,y}(x_1,y_1) \geq c_0$, for any $(x,y)\in K_0$ and $(x_1,y_1) \in R$. So we have that for any $(x,y)\in K_0$ and any measurable set $A$,
\begin{align} \label{eq: small-set-A}
P((x,y),A)  = \int_{(x_1,y_1)\in A} f_{x,y}(x_1,y_1) ~\diff \mathrm{Vol} \geq \int_{(x_1,y_1)\in A \cap R} f_{x,y}(x_1,y_1) ~\diff \mathrm{Vol}  \geq c_0 \,\varphi(A).
\end{align}
We will show that we can reach $K_0$ from any point in $K_\rs$, in at most 5 steps with probability at least $c_{\mathrm{in}}$, i.e. for each $(x,y) \in K_\rs$, there is some $m_0 \leq 5$ such that
\begin{align} \label{eq: small-set-B}
 P^{m_0}\bigl( (x,y), K_0\bigr) \geq c_{\mathrm{in}}.
\end{align}
Together,~\eqref{eq: small-set-A} and~\eqref{eq: small-set-B} imply that for all $(x,y) \in K_\rs$ and measurable $A$, there is some $m_0 \leq 5$ such that
\[
P^{m_0+1}((x,y),A) \geq c_{\mathrm{in}}c_0 \, \varphi(A),
\]
verifying that $K_\rs$ is a petite set.

\bigskip
\noindent \textit{{Verifying~\eqref{eq: small-set-B}.}} Since $K_r\subset\calB \cup \calD \cup \calE$, we treat the three cases separately. We identify two rectangular sets $K_A \subseteq \calA$ and $K_C \subseteq \calC$ that play a role in the proof:
\[ 
K_A \triangleq [-d, \, 0] \times [2\base, \, 2\base+d], 
~~~~ ~~~~
K_C \triangleq [2(\base+\extra), \, 2(\base+\extra)+d] \times [-d, \, 0 ].
\]

Throughout the proof, $u\sim\Exp(p)$, $v\sim\Exp(1-p)$, $w\sim\Exp(1)$ are independent. Starting from $(x,y)$ at time $0$, we denote $(x_t,y_t)$ as the position of the Markov chain at time $t$.

\begin{itemize}
\item[--] \emph{Case $1$.} $(x,y) \in K_\rs \cap \calB$. Here, $(x_1,y_1) = (x + w, y + w -v)$. We move to $K_C$ if:
\begin{align} \label{eq: event-b-to-c}
w\in\Bigl[2(\base+\extra) - x, 2(\base+\extra)+d -x\Bigr],\qquad v\in\bigl[y+w,\ y+w+d\bigr].
\end{align}
This event has probability at least
\[
c_{\calB \to K_C}\ \triangleq\  \frac{(1-e^{-(1-p)d}) \, (1-e^{-(2-p)d})}{2-p} \, \exp(-(1-p) 2\base - (2-p)[2(\base+\extra)+r])>0,
\]
uniformly over $x,y\in\calB \cap K_r$. This is because $x \in [-\rs, 2(\base+\extra)]$ and $y \in [0,2\base]$, so the event in~\eqref{eq: event-b-to-c} has a uniform lower bound on its probability. On this event, $(x_1,y_1)\in K_C$. Since $K_C\subseteq \calC$, we have in the next time step that $(x_2, y_2) = (x_1 - u,y_1)$. Choose 
\( 
u \in [x_1 - 2d, x_1 -d],
\)
which has probability at least 
\[ \int_{x_1 - 2d}^{x_1  - d} p e^{-p u} \, \diff u = \left(e^{d p}-1\right) e^{d p }\,e^{-p x_1} \geq  \left(e^{d p}-1\right) e^{d p } \,e^{-p (2(\base+\extra ) + d)} 
\triangleq c_{K_C \to K_0}  \triangleq x_{\calB} > 0.
\]
Therefore,
\[
\inf_{(x,y)\in K_\rs\cap\calB}P^2\bigl((x,y),K_0\bigr)\ \ge\ c_{\calB \to K_C} \cdot c_{K_C \to K_0} > 0,
\]
so we can reach $K_0$ from any point in $\calB$ in two steps with strictly positive probability. 

\item[--] \textit{Case $2$.} $(x,y) \in K_\rs \cap \calE$. Here, $(x_1,y_1) = (x+w-u,y+w)$. We move to $K_A$ if:
\begin{align} \label{eq: event-e-to-a}
w\in\bigl[2\base - y, \, 2\base + d -y ],\qquad u\in\bigl[x+w,\ x+w+d\bigr],
\end{align}
which has probability at least
\[
c_{\calE\to K_A}\ \triangleq\ \frac{(1-e^{-pd}) \, (1-e^{-(1+p)d})}{1+p} \, \exp(-2p(\base+\extra) - (1+p)(2\base + \rs)) >0,
\]
uniformly over $(x,y)\in K_\rs\cap\calE$. This is because $y\leq 2\base$ and $x\leq 2(\base+\extra)$, so the event in~\eqref{eq: event-e-to-a} has a uniform lower bound on its probability. On this event, $(x_1,y_1) \in K_A$. Since $K_A \subseteq \calA$, we have in the next time step that $(x_2, y_2) = (x_1, y_1 - v)$. Choose $v \in [d,2d]$ so that $(x_2,y_2)  \in \cal B$, which can be done with probability 
\[ 
\int_d^{2d} (1-p) \,e^{-(1-p)v} \, \diff v \triangleq c_{K_A \to \calB} > 0.
\]

This choice yields $(x_2,y_2)\in K_\rs \cap \calB$, so the argument from Case 1 ensures that we can reach $K_0$ in at most 2 more steps. Thus
\[
\inf_{(x,y)\in K_\rs\cap\calE}P^{\,4}\bigl((x,y),K_0\bigr)\ \ge\ c_{\calE\to K_A} \cdot c_{K_A \to \calB} \cdot c_{\calB \to K_C} \cdot c_{K_C \to K_0} \triangleq c_{\calE}> 0.
\]

\item[--] \textit{Case $3$.} $(x,y) \in \calD$. Here, $(x_1,y_1) = (x+w,y+w)$. Take one $\calD$–step with
\[
w\in\bigl[y,\, y + d\bigr],
\]
which has probability at least 
\[
(1- e^{-d})\, e^{-y} \geq (1-e^{-d})\, e^{-\rs} \triangleq c_{\calD \to \calB \cup \calE} > 0
\] 
uniformly. This yields $(x_1,y_1)\in (\calB \cup \calE) \cap K_\rs$, so the arguments from cases 1 and 2 apply. Hence, uniformly on $K_\rs\cap\calD$,
\[
\inf_{(x,y)}P^{\,5}\bigl((x,y),K_0\bigr)\ \geq c_{\calD \to \calB \cup \calE} \cdot c_{\calE \to K_A} \cdot c_{K_A \to \calB} \cdot c_{\calB \to K_C} \cdot c_{K_C \to K_0} \triangleq c_{\calD} > 0
\]

Combining the three cases, \eqref{eq: small-set-B} holds with $m_0=5$ and $c_{\mathrm{in}}=\min\{c_{\calB},\,c_{\calE},\,c_{\calD}\}>0$. 
\end{itemize}

\medskip 
\noindent (ii) Irreducibility is immediate. From any $(x,y) \in \reals^2$ we can reach $K_0$ in at most $5$ steps with positive probability. Then, from $K_0$, a single step dominates $\mathrm{Vol}$ on $R$, hence for any $A \subseteq R$ with $\mathrm{Vol}(A) > 0$, there exists $t \leq 6$ with $P^t((x,y),A) > 0$.

\subsubsection{Proof of~\texorpdfstring{\Cref{lem:drift}}{}}
Throughout, $u_t \stackrel{\mathrm{iid}}{\sim}\Exp(p)$, $v_t\stackrel{\mathrm{iid}}{\sim}\Exp(1-p)$ and $w_t\stackrel{\mathrm{iid}}{\sim}\Exp(1)$ are independent, and the one–step increment $\nabla\potential(t)=(\Delta_x,\Delta_y)$ is given by the five cases (A)–(E). We write $x_+=\max\{x,0\}$ and $x_-=\max\{-x,0\}$. 
We use the identities, for $\lambda > 0$, $Z\sim\Exp(\lambda)$ and $s\ge0$,
\begin{align} \label{eq: expectations}
  \mathbb{E}[(s-Z)_+] = s - \mathbb E[Z\wedge s]= s -\frac{1-e^{-\lambda s}}{\lambda},
  \qquad
  \mathbb E\!\bigl[(Z-s)_+^2\bigr]=\frac{2\,e^{-\lambda s}}{\lambda^2},
\end{align}
and, for $W\sim\Exp(1)$ and $a\ge0$,
\begin{equation}\label{eq:sqdrop}
  \mathbb E\!\bigl[(a-W)_+^2-a^2\bigr]\ =\ -2a\ +\ 2(1-e^{-a})\ \le\ -a\quad\text{for }a\ge 2.
\end{equation}
Conditioned on $\potential(t) = [x,\, y]$, we have $\potential(t+1) = [x+ \Delta_x, \, y+\Delta_y]$. Let In all cases (A)--(E), we have that the expected drift $\Delta V(x,y) \triangleq \mathbb{E}[V(\potential(t+1)) | \potential(t) = (x,y)] - V(x,y)$ is 
\begin{align*}
    \Delta V(x,y) = \alpha \,\mathbb{E}[ (x \!+\! \Delta_x)_+ - x_+] + \beta \,\mathbb{E}[(y \!+\! \Delta_y)_+ - y_+] + \delta \, \mathbb{E} \left[(x\!+\!\Delta_x)_-^2 - x_-^2 +  (y \!+ \! \Delta_y)_-^2 - y_-^2\right].
\end{align*}
We bound the drift region–wise. In what follows, all expectations are conditional on $\potential(t)=(x,y)$.

\begin{itemize}
    \item[--] \emph{Region $\calA$}. Here, $\Delta_x = 0$ and $\Delta_y = -v$ with $v\sim\Exp(1-p)$. Further, $y\ge 2\base$, so $y_+ = y$ and $y_-=0$. Therefore, using~\eqref{eq: expectations} gives
    \[
        \Delta V(x,y)
        = \beta\bigl(\mathbb E[(y-v)_+]-y\bigr) +  \delta\,\mathbb E[(y-v)_-^2]
        \leq  -\,\frac{\beta}{1-p}\bigl(1-e^{-(1-p)y}\bigr) + \frac{2\delta}{(1-p)^2}e^{-(1-p)y},
    \]
    which is maximized when $y = 2\base$ in region $\calA$. Thus, we may choose
    \begin{align} \label{eq: beta-bound}
        \beta\ >\ \frac{2\delta}{1-p}\cdot\frac{e^{-2(1-p)\base}}{1-e^{-2(1-p)\base}}
    \end{align}
    to make $\Delta V \le -\eta_A<0$, uniformly in $\calA$.

    \item[--] \emph{Region $\calB$.} Here $\Delta_x=w$ and $\Delta_y=w-v$ with $w\sim\Exp(1)$ and $v\sim\Exp(1-p)$ independent, and $0\le y\le 2\base$, $x\le y+2\extra$. The only way to be outside $K_\rs$ in $\calB$ is $x\le -\rs$. In this case $x_+=0$, $x_-=-x\ge \rs$, and
    \begin{align*}
      \mathbb E\big[(x+w)_+ - x_+\big] &=\mathbb E[(w-x_-)_+]=e^{-x_-}\le e^{-\rs},
        \\
      \mathbb E\big[(x+w)_-^2 - x_-^2\big] &= -2x_- + 2(1-e^{-x_-})\ \le\ -2\rs+2.
    \end{align*}
    For the $y$–coordinate, since $y\ge 0$ we have $y_+=y$, $y_-=0$, and
    \begin{align} \label{eq: region-B-chk-1}
      \mathbb E\big[(y+w-v)_+ - y_+\big]
      &= \mathbb E\big[(y+w) - (v\wedge (y+w))\big] - y
       = \mathbb E[w] - \mathbb E\big[v\wedge (y+w)\big].
    \end{align}
    
    Since $y+w\ge w$ pointwise,
    \(
      \mathbb E[v\wedge (y+w)]\ge \mathbb E[v\wedge w]=\frac{1}{2-p},
    \)
    where the last step used that $v \wedge w \sim \mathrm{Exp}(2-p)$.
    Hence, continuing from~\eqref{eq: region-B-chk-1} we have
    \[
      \mathbb E\big[(y+w-v)_+ - y_+\big]\ \le\ 1-\frac{1}{2-p}\ =\ \frac{1-p}{2-p}.
    \]
    Moreover,
    \[
      \mathbb E\big[(y+w-v)_-^2\big]
      = \mathbb E\big[(v-(y+w))_+^2\big]\ \le\ \mathbb E[v^2]=\frac{2}{(1-p)^2}.
    \]
    Collecting terms, for $(x,y)\in \calB$ with $x\le -\rs$ we obtain
    \[
      \Delta V(x,y)\ \le\ \alpha\, e^{-\rs}\ +\ \delta \, (-2\rs+2)\ +\ \beta\,\frac{1-p}{2-p}\ +\ \delta\,\frac{2}{(1-p)^2}.
    \]
    Thus, for all sufficiently large $\rs$ (given $\alpha,\beta,\delta$), there exists $\eta_B>0$ such that
    \[
      \Delta V(x,y)\ \le\ -\,\eta_B\qquad\text{for all }(x,y)\in \calB\cap\{x\le -\rs\}.
    \]

    \item[--] \emph{Region $\calC$}. Here, $\Delta_x = -u$ with $u\sim\mathrm{Exp}(p)$ and $\Delta_y=0$. Further, $x\geq 2(\base+\extra)$, so $x_+ = x$ and $x_-=0$. Therefore, using~\eqref{eq: expectations} gives
    \[ 
        \Delta V(x,y) = \alpha \bigl( \mathbb{E}[(x-u)_+ - x]\bigr) + \delta \, \mathbb{E}[(x-u)_-^2] \leq - \frac \alpha p \left( 1 - e^{-px} \right) + \frac{2\delta}{p^2} e^{-px},
    \]
    which is maximized when $x = 2(\base+\extra)$ in region $\calC$. Thus, we may choose
    \begin{align} \label{eq: alpha-bound}
        \alpha\ >\ \frac{2\delta}{p}\cdot\frac{e^{-2p(\base+\extra)}}{1-e^{-2p(\base+\extra)}}
    \end{align}
    to make $\Delta V \le -\eta_C<0$, uniformly in $\calC$.

    \item[--] \emph{Region $\calD$.} Here $\Delta_x=\Delta_y=w$ with $w\sim\Exp(1)$. Further, $x\le 0$ and $y\le 0$, so $x_+=y_+=0$ and $x_-=-x=:a\ge 0$, $y_-=-y=:b\ge 0$. Therefore,
    \begin{align*}
      \Delta V(x,y)
      = \alpha\,\mathbb E\big[(x+w)_+\big]
       + \beta\,\mathbb E\big[(y+w)_+ \big] + \delta\,\mathbb E\Big[\big((x+w)_-^2 - x_-^2\big)+\big((y+w)_-^2 - y_-^2\big)\Big].
    \end{align*}
    Since $(x+w)_+=(w-a)_+$ and $(x+w)_-^2=(a-w)_+^2$, a direct computation gives
    \[
      \mathbb E\big[(x+w)_+\big]=\mathbb E\big[(w-a)_+\big]=\int_a^\infty (w-a)e^{-w}\,dw=e^{-a}.
    \]
    Using \eqref{eq:sqdrop} with $a$ and $b$ we obtain
    \[
      \mathbb E\big[(x+w)_-^2 - x_-^2\big]= -2a + 2(1-e^{-a}),
      \qquad
      \mathbb E\big[(y+w)_-^2 - y_-^2\big]= -2b + 2(1-e^{-b}).
    \]
    Hence
    \[
      \Delta V(x,y)=\alpha e^{-a}+\beta e^{-b}
      +\delta\Big(-2a+2(1-e^{-a})-2b+2(1-e^{-b})\Big).
    \]
    On $\calD\cap\{x\le -\rs\ \text{or}\ y\le -\rs\}$ at least one of $a,b$ is $\ge \rs$. Without loss of generality assume $a\ge \rs$. Then $e^{-a}\le e^{-\rs}$ and $-2a+2(1-e^{-a})\le -2\rs+2$, while for $b\ge 0$ we have $e^{-b}\le 1$ and $-2b+2(1-e^{-b})\le 2$. Therefore
    \[
      \Delta V(x,y)\ \le\ \alpha e^{-\rs}+\beta+\delta(-2\rs+2)+2\delta.
    \]
    Choosing $\rs$ sufficiently large (given $\alpha,\beta,\delta$), we see there exists $\eta_D>0$ such that
    \[
      \Delta V(x,y)\ \le\ -\,\eta_D\qquad\text{for all }(x,y)\in \calD\cap\{x\le -\rs\ \text{or}\ y\le -\rs\}.
    \]

    \item[--] \emph{Region $\calE$.} Here $\Delta_x=w-u$ and $\Delta_y=w$ with $u\sim\Exp(p)$ and $w\sim\Exp(1)$ independent, and $0\le x\le 2(\base+\extra)$, $y\le (x-2\extra)_+\le 2\base$. The only way to be outside $K_\rs$ in $\calE$ is $y\le -\rs$. Then $y_+=0$, $y_-=-y\ge \rs$, and
    \begin{align*}
      \mathbb E\big[(y+w)_+ - y_+\big] & =\mathbb E[(w-y_-)_+]=e^{-y_-}\le e^{-\rs},
      \\
      \mathbb E\big[(y+w)_-^2 - y_-^2\big] &= -2y_- + 2(1-e^{-y_-})\ \le\ -2\rs+2.
    \end{align*}
    For the $x$–coordinate, since $x\ge 0$ we use that the function $g(z) = z_+$ is 1–Lipschitz, so
    \[
      (x+w-u)_+ - x_+\ \le\ (w-u)_+.
    \]
    Therefore, we have
    \[
      \mathbb E\big[(x+w-u)_+ - x_+\big]\ \le\ \mathbb E[(w-u)_+]\ =\ \frac{p}{1+p}.
    \]
    Also, $(x+w-u)_-^2 \le (u-(x+w))_+^2\le u^2$, hence
    \(
      \mathbb E\big[(x+w-u)_-^2\big]\le \mathbb E[u^2]=2/p^2.
    \)
    Collecting terms, for $(x,y)\in \calE$ with $y\leq -\rs$,
    \[
      \Delta V(x,y)\ \leq \ \alpha\,\frac{p}{1+p} + \beta\, e^{-\rs} + \delta\,(-2\rs+2) + \delta\,\frac{2}{p^2}.
    \]
    Thus, for $\rs$ large (given $\alpha,\beta,\delta$), there exists $\eta_E>0$ such that
    \[
      \Delta V(x,y)\ \le\ -\,\eta_E\qquad\text{for all }(x,y)\in \calE\cap\{y\le -\rs\}.
    \]
\end{itemize}

Now, pick $\alpha,\beta$ satisfying~\eqref{eq: beta-bound} and~\eqref{eq: alpha-bound}, so that the uniform bounds in Regions $\calA$ and $\calC$ are strictly negative. Then choose $\rs$ sufficiently large so that the bounds in Regions $\calD$, $\calB$, and $\calE$ are also strictly negative outside $K_\rs=[-\rs,\,2(\base+\extra)]\times[-\rs,\,2\base]$. 

Inside the rectangle $K_\rs$, it is easy to see that $\Delta V \leq b_0$ for some finite constant $b_0$. This is because $K_\rs$ is a compact set, and the one-step increments $(\Delta_x,\Delta_y)$ are affine in exponential random variables with finite first and second moments, so the conditional expectation $\mathbb{E}[V(\potential(t+1)) | \potential(t) = (x,y)]$ is uniformly finite over $K_\rs$. Therefore, there exist $\eta>0$ and $b_0 < \infty$ such that
\[
  \mathbb E\!\left[V\bigl(\potential(t+1)\bigr)-V\bigl(\potential(t)\bigr)\,\middle|\,\potential(t)=(x,y)\right]
  \ \le\ -\,\eta\,\mathbf 1_{K_\rs^c}(x,y)\ +\ b_0\,\mathbf 1_{K_\rs}(x,y).
\]
This concludes the proof of~\Cref{lem:drift}.

\subsection{Supplementary proofs for \prettyref{sec:dual_exact}}

\subsubsection{Proof of~\texorpdfstring{\Cref{lem-stationary-density-c-equal-0}}{}}
Recall that 
\[
    \pi(x,y)=
    \begin{cases}
        C\,e^{p x-(1+p)y}, & (x,y)\in \calA_1,\\
        C\,e^{-(1-p)x-py}, & (x,y)\in \calA_2,\\
        C\,e^{2\extra}\,e^{-(2-p)x+(1-p)y}, & (x,y)\in \calC,\\
        C\,e^{p x+(1-p)y}, & (x,y)\in \calD,\\
        C\,e^{-(1-p)x+(1-p)y}, & (x,y)\in \calE_2,
    \end{cases}
    \qquad
    \mbox{ where }C \triangleq \frac{e^{2\extra}\,p(1-p)^2}{(2-p)\,e^{2\extra}-p\,e^{2p\extra}}.
\]
With respect to Lebesgue measure (and using Dirac deltas), the regionwise kernels are
\begin{align*}
    K_{\calA}((x_1,y_1),(x_2,y_2))&=(1-p)e^{-(1-p)(y_1-y_2)}\,
    \delta(x_2-x_1)\,\mathbf{1}_{\{y_2<y_1\}},\\
    K_{\calC}((x_1,y_1),(x_2,y_2))&=p\,e^{-p(x_1-x_2)}\,
    \delta(y_2-y_1)\,\mathbf{1}_{\{x_2<x_1\}},\\
    K_{\calD}((x_1,y_1),(x_2,y_2))&=e^{-(x_2-x_1)}\,
    \delta\big((y_2-y_1)-(x_2-x_1)\big)\,\mathbf{1}_{\{x_2\ge x_1\}},\\
    K_{\calE_2}((x_1,y_1),(x_2,y_2))&=p\,e^{-(1+p)(y_2-y_1)}\,e^{p(x_2-x_1)}\,
    \mathbf{1}_{\{y_2\ge y_1\}}\mathbf{1}_{\{x_2-x_1\le y_2-y_1\}}.
\end{align*}
The global balance equations are verified below, i.e. we show that for any $(x,y)\in \reals^2$:
\[ 
    \pi(x,y)
    = \!\! \sum_{R\in\{\calA_1, \calA_2,\calC,\calD,\calE_2\}}\!\!\!\!\!\!  T_R(x,y), ~~~~ \mbox{where } T_R(x,y) \triangleq 
    \iint_{(x_1,y_1)\in R} \!\!\!\! \pi(x_1,y_1)\,K_{R}\big((x_1,y_1),(x,y)\big)\,\diff x_1\,\diff y_1,
\]
and $\iint_{\reals^2} \pi(x,y) \, \diff x \, \diff y = 1$.

\noindent \textbf{Target $(x,y)\in\calA_1$.} We compute $T_R(x,y)$ in each region $R$. 
\begin{align*}
    T_{\calA_1}(x,y) \! &= \int_{y}^{\infty} C \, e^{p x-(1+p)y_1}\,(1-p)e^{-(1-p)(y_1-y)}\,\mathrm{d}y_1
    =\tfrac{C(1-p)}{2}\,e^{p x-(1+p)y}
    \\
    T_{\calA_2}(x,y) \! &= \int_0^x C\,  e^{-(1-p) x - p y_1}\, \times 0 = 0
    \\
    T_{\calC}(x,y) &= \int_{y+2\extra}^{\infty} C \, e^{2\extra}e^{-(2-p)x_1+(1-p)y}\,p e^{-p(x_1-x)}\,\mathrm{d}x_1
    =\tfrac{C p}{2}e^{-2\extra}\,e^{p x-(1+p)y}
    \\
    T_{\calD}(x,y) & = \int_{-\infty}^{\,x-y} C \, e^{p x_1+(1-p)(y-(x-x_1))}\,e^{-(x-x_1)}\,\mathrm{d}x_1
    =\tfrac{C}{2}\,e^{p x-(1+p)y}
    \\
    T_{\calE_2}(x,y) & = \int_{0}^{2\extra}\!\!\int_{-\infty}^{\,y}
    C \, e^{-(1-p)x_1+(1-p)y_1}\,p e^{-(1+p)(y-y_1)} e^{p(x-x_1)}\,\mathrm{d}y_1\,\mathrm{d}x_1
    =\tfrac{C p}{2}(1-e^{-2\extra})\,e^{p x-(1+p)y}
\end{align*}  
Adding the coefficients gives $ C \big( \frac12+\frac{(1-p)}{2}+\frac{p}{2}e^{-2\extra}+\frac{p}{2}(1-e^{-2\extra}) \big)=C$, as desired.

\noindent \textbf{Target $(x,y)\in\calA_2$.} We compute $T_R(x,y)$ in each region $R$. 
\begin{align*}
    T_{\calA_1}(x,y) \! &= \int_{x}^{\infty} C \, e^{p x-(1+p)y_1}\,(1-p)e^{-(1-p)(y_1-y)}\,\mathrm{d}y_1
    =\tfrac{C(1-p)}{2}\,e^{-(2-p)x+(1-p)y}
    \\
    T_{\calA_2}(x,y) \! &= \int_{y}^{x} C\,  e^{-(1-p)x-py_1}\,(1-p)e^{-(1-p)(y_1-y)}\,\mathrm{d}y_1
    = C(1-p) \big( e^{-(1-p)x-py}- e^{-(2-p)x+(1-p)y} \big)
    \\
    T_{\calC}(x,y) &= \int_{y+2\extra}^{\infty} C e^{2\extra}e^{-(2-p)x_1+(1-p)y}\,p e^{-p(x_1-x)}\,\mathrm{d}x_1
    =\tfrac{C p}{2}e^{p x-2\extra}\,e^{-(1+p)y}
    \\
    T_{\calD}(x,y) & = \int_{-\infty}^{\,x-y} C e^{p x_1+(1-p)(y-(x-x_1))}\,e^{-(x-x_1)}\,\mathrm{d}x_1
    =\tfrac{C}{2}\,e^{-(2-p)x+(1-p)y}
    \\
    T_{\calE_2}(x,y) & = \int_{0}^{x-y}\int_{-\infty}^{y-(x-x_1)} C e^{-(1-p)x_1+(1-p)y_1} \, p e^{-(1+p)(y-y_1)} e^{p(x-x_1)} \, \diff y_1 \, \diff x_1 \\
    & ~~~~ +\int_{x-y}^{2\extra}\int_{-\infty}^{0} C e^{-(1-p)x_1+(1-p)y_1} \, p e^{-(1+p)(y-y_1)} e^{p(x-x_1)}\, \diff y_1 \, \diff x_1
    \\ 
     & ~~~~ = C p \left[e^{-(1-p)x-py}-\tfrac12 e^{-(2-p)x+(1-p)y}-\tfrac12 e^{p x-2\extra} \, e^{-(1+p)y}\right]
\end{align*}
Adding the terms, only the $e^{-(1-p)x -py}$ terms survive and their coefficients add up to $C$, as desired.

\noindent \textbf{Target $(x,y)\in\calC$.} We compute $T_R(x,y)$ in each region $R$. 
\begin{align*}
    T_{\calA_1}(x,y) \! &= \int_{x}^{\infty} C e^{p x-(1+p)y_1}\,(1-p)e^{-(1-p)(y_1-y)}\,\mathrm{d}y_1
    =\tfrac{C(1-p)}{2}\,e^{-(2-p)x+(1-p)y}
    \\
    T_{\calA_2}(x,y) \! &= \int_{y}^{x} C e^{-(1-p)x-py_1}\,(1-p)e^{-(1-p)(y_1-y)}\,\mathrm{d}y_1
    =C(1-p)(e^{2\extra}-1)\,e^{-(2-p)x+(1-p)y}
    \\
    T_{\calC}(x,y) &= \int_{x}^{\infty} C e^{2\extra}e^{-(2-p)x_1+(1-p)y}\,p e^{-p(x_1-x)}\,\mathrm{d}x_1
    =\tfrac{C p}{2}e^{2\extra}\,e^{-(2-p)x+(1-p)y}
    \\
    T_{\calD}(x,y) & = \int_{-\infty}^{0} C e^{p x_1+(1-p)(y-(x-x_1))}\,e^{-(x-x_1)}\,\mathrm{d}x_1
    =\tfrac{C}{2}\,e^{-(2-p)x+(1-p)y}
    \\
    T_{\calE_2}(x,y) & = \int_{0}^{2\extra}\!\!\int_{-\infty}^{\,y-(x-x_1)}
    C e^{-(1-p)x_1+(1-p)y_1}\,p e^{-(1+p)(y-y_1)} e^{p(x-x_1)}\,\mathrm{d}y_1\,\mathrm{d}x_1
    \\
    &= \tfrac{C p}{2}(e^{2\extra}-1)\,e^{-(2-p)x+(1-p)y}
\end{align*}
Adding the coefficients gives $ C \big( \frac{1-p}{2} + (1-p)(e^{2\extra} - 1) +  \frac{p}{2} e^{2\extra} + \frac 1 2 + \frac{p}{2}(e^{2\extra} - 1)  \big)= C e^{2\extra}$ as desired.

\noindent \textbf{Target $(x,y)\in\calD$.} We compute $T_R(x,y)$ in each region $R$. 
\begin{align*}
    T_{\calA_1}(x,y) \! &= \int_{0}^{\infty} C\,  e^{p x-(1+p)y_1}\,(1-p)e^{-(1-p)(y_1-y)}\,\diff y_1
    =\tfrac{C(1-p)}{2}\,e^{p x+(1-p)y} 
    \\
    T_{\calA_2}(x,y) \! &= \int_0^x C\,  e^{p x-(1+p)y_1}\, \times 0 = 0 \\
    T_{\calC}(x,y) &= \int_{2\extra}^{\infty} C\, e^{2\extra}e^{-(2-p)x_1+(1-p)y}\,p e^{-p(x_1-x)}\,\diff x_1
    =\tfrac{C p}{2}e^{-2\extra}\,e^{p x+(1-p)y} 
    \\
    T_{\calD}(x,y) & = \int_{-\infty}^{x} C \, e^{p x_1+(1-p)(y-(x-x_1))}\,e^{-(x-x_1)}\,\diff x_1
    =\tfrac{C}{2}\,e^{p x+(1-p)y} 
    \\
    T_{\calE_2}(x,y) & = \int_{0}^{2\extra}\!\!\int_{-\infty}^{0 \wedge (y-(x-x_1))}
    C\,  e^{-(1-p)x_1+(1-p)y_1}\,p e^{-(1+p)(y-y_1)} e^{p(x-x_1)}\,\mathrm{d}y_1\,\mathrm{d}x_1
    \\ 
    & =\tfrac{C p}{2}(1-e^{-2\extra})\,e^{p x+(1-p)y}.
\end{align*}
Adding the coefficients gives $C\big( \frac12+\frac{1-p}{2}+\frac{p}{2}e^{-2\extra}+\frac{p}{2}(1-e^{-2\extra}) \big)= C$, as desired.

\bigskip

\noindent \textbf{Target $(x,y)\in\calE_2$.} We compute $T_R(x,y)$ in each region $R$. 
\begin{align*}
    T_{\calA_1}(x,y) \! &= \int_{x}^{\infty} C e^{p x-(1+p)y_1}\,(1-p)e^{-(1-p)(y_1-y)}\,\mathrm{d}y_1
    =\tfrac{C(1-p)}{2}\,e^{-(2-p)x+(1-p)y}
    \\
    T_{\calA_2}(x,y) \! &= \int_{0}^{x} C e^{-(1-p)x-py_1}\,(1-p)e^{-(1-p)(y_1-y)}\,\mathrm{d}y_1
    =C(1-p)(1-e^{-x})\,e^{-(1-p)x+(1-p)y}
    \\
    T_{\calC}(x,y) &= \int_{2\extra}^{\infty} C e^{2\extra}e^{-(2-p)x_1+(1-p)y}\,p e^{-p(x_1-x)}\,\mathrm{d}x_1
    =\tfrac{C p}{2}e^{(1-p)y}\,e^{p x-2\extra}
    \\
    T_{\calD}(x,y) & = \int_{-\infty}^{0} C e^{p x_1+(1-p)(y-(x-x_1))}\,e^{-(x-x_1)}\,\mathrm{d}x_1
    =\tfrac{C}{2}\,e^{-(2-p)x+(1-p)y}
    \\
    T_{\calE_2}(x,y) & = \int_{0}^{x} \int_{-\infty}^{y-(x-x_1)} C \, e^{-(1-p)x_1+(1-p)y_1} \, p e^{-(1+p)(y-y_1)} e^{p(x-x_1)} \,\diff y_1 \, \diff x_1
   \\ 
    & ~~~~ + \int_{x}^{2\extra} \int_{-\infty}^{0} C \, e^{-(1-p)x_1+(1-p)y_1} \, p e^{-(1+p)(y-y_1)} e^{p(x-x_1)} \,\diff y_1 \, \diff x_1 \\
    & ~~~~ = \tfrac{C p}{2}e^{(1-p)y} \left(2e^{-(1-p)x}-e^{-(2-p)x}-e^{p x-2\extra}\right)
\end{align*}
Adding the terms, only the $e^{-(1-p)x + (1-p)y}$ terms survive and their coefficients add up to $C$, as desired.

\bigskip

\noindent  Finally, we verify normalization. Indeed
\begin{align*}
    1
    &= \int_{(x,y) \in \reals^2} \pi(x,y) \, \diff x \,\diff y = \sum_{R \in \{ \calA_1,\calA_2,\calC,\calD,\calE_2\}} \iint_{(x,y) \in R} \pi(x,y) \, \diff x\, \diff y \\
    &= 
    \int_{0}^{\infty}\!\!\int_{-\infty}^{y} C\,e^{p x-(1+p)y}\,\diff x\,\diff y
    +\int_{0}^{\infty}\!\!\int_{y}^{\,y+2\extra} C\,e^{-(1-p)x-py}\,\diff x\,\diff y
    +\int_{-\infty}^{0}\!\!\int_{-\infty}^{0} C\,e^{p x+(1-p)y}\,\diff y\,\diff x\\
    & ~~~~~~
    +\int_{2\extra}^{\infty}\!\!\int_{-\infty}^{\,x-2\extra} C\,e^{2\extra}\,e^{-(2-p)x+(1-p)y}\,\diff y\,\diff x
    +\int_{0}^{2\extra}\!\!\int_{-\infty}^{0} C\,e^{-(1-p)x+(1-p)y}\,\diff y\,\diff x\\
    &= C\left[
    \frac{1}{p}
    +\frac{1-e^{-2(1-p)\extra}}{1-p}
    +\frac{e^{-2(1-p)\extra}}{1-p}
    +\frac{1}{p(1-p)}
    +\frac{1-e^{-2(1-p)\extra}}{(1-p)^2}
    \right] \\
    &= \frac{C}{p(1-p)^2}\Big(2 - p - p\,e^{-2(1-p)\extra}\Big),
\end{align*}
for our choice of $C$. This concludes the proof.

\subsubsection{Proof of~\texorpdfstring{\Cref{lem-stationary-density-eps-equal-0}}{}} 
Recall that
\[
    \pi(x,y)=
    \begin{cases}
        C e^{2\base}\,e^{p x-(1+p)y}, & (x,y)\in \calA_1,\\[2pt]
        C e^{p x-p y}, & (x,y)\in \calB_1,\\[2pt]
        C e^{2\base}\,e^{-(2-p)x+(1-p)y}, & (x,y)\in \calC,\\[2pt]
        C e^{p x+(1-p)y}, & (x,y)\in \calD,\\[2pt]
        C e^{-(1-p)x+(1-p)y}, & (x,y)\in \calE_1.
    \end{cases}
    \qquad \mbox{ where } C = \frac{p(1-p)}{2(1+\base)}.
\]     
With respect to Lebesgue measure (and using Dirac deltas), the regionwise kernels are
\begin{align*}
    K_{\calA_1}((x_1,y_1),(x_2,y_2))
    &=(1-p)e^{-(1-p)(y_1-y_2)}\,\delta(x_2-x_1)\,\mathbf{1}_{\{y_2<y_1\}},\\
    K_{\calB_1}((x_1,y_1),(x_2,y_2))
    &=(1-p)\,e^{-(2-p)(x_2-x_1)}\,e^{(1-p)(y_2-y_1)}\,
    \mathbf{1}_{\{x_2\ge x_1\}}\mathbf{1}_{\{y_2-y_1\le x_2-x_1\}}, \\
    K_{\calC}((x_1,y_1),(x_2,y_2))
    &=p\,e^{-p(x_1-x_2)}\,\delta(y_2-y_1)\,\mathbf{1}_{\{x_2<x_1\}},\\
    K_{\calD}((x_1,y_1),(x_2,y_2))
    &=e^{-(x_2-x_1)}\,\delta\big((y_2-y_1)-(x_2-x_1)\big)\,\mathbf{1}_{\{x_2\ge x_1\}},\\
    K_{\calE_1}((x_1,y_1),(x_2,y_2))
    &=p\,e^{-(1+p)(y_2-y_1)}\,e^{p(x_2-x_1)}\,
    \mathbf{1}_{\{y_2\ge y_1\}}\mathbf{1}_{\{x_2-x_1\le y_2-y_1\}}.
\end{align*}  

The global balance equations are verified below, i.e. we show that for any $(x,y) \in \reals^2$:
\[
\pi(x,y)
=\!\!\! \sum_{R\in\{A_1,B_1,C,D,E_1\}} \!\!\!\!\!\!\!\!\! T_R(x,y), ~~~~\mbox{where }
T_R(x,y) \triangleq  \iint_{(x_1,y_1)\in R} \!\!\!\! \pi(x_1,y_1)\,K_R\big((x_1,y_1),(x,y)\big)\, \diff x_1\, \diff y_1,
\]
and $\iint_{\reals^2} \pi(x,y) \, \diff x \, \diff y = 1$.

\noindent \textbf{Target $(x,y)\in\calA_1$.} We compute $T_R(x,y)$ in each region $R$:
\begin{align*}
    T_{\calA_1}(x,y)
    &=\int_{y}^{\infty} C e^{2\base} e^{p x-(1+p)y_1}\,(1-p)e^{-(1-p)(y_1-y)}\,\diff y_1
    =\tfrac{C(1-p)}{2}\,e^{2\base}\,e^{p x-(1+p)y},\\[2pt]
    T_{\calB_1}(x,y)
    &=\int_{0}^{2\base}\!\!\int_{-\infty}^{\,y_1+(x-y)}
    C e^{p x_1-p y_1}\,(1-p)e^{-(2-p)(x-x_1)}e^{(1-p)(y-y_1)}\,\diff x_1\,\diff y_1\\
    &=\tfrac{C(1-p)}{2}(e^{2\base}-1)\,e^{p x-(1+p)y},\\[2pt]
    T_{\calC}(x,y)
    &=\int_{y}^{\infty} C e^{2\base}e^{-(2-p)x_1+(1-p)y}\,p e^{-p(x_1-x)}\,\diff x_1
    =\tfrac{C p}{2}\,e^{2\base}\,e^{p x-(1+p)y},\\[2pt]
    T_{\calD}(x,y)
    &=\int_{-\infty}^{\,x-y} C e^{p x_1+(1-p)(y-(x-x_1))}\,e^{-(x-x_1)}\,\diff x_1
    =\tfrac{C}{2}\,e^{p x-(1+p)y},\\[2pt]
    T_{\calE_1}(x,y)
    &=\int_{0}^{2\base}\!\!\int_{-\infty}^{\,y}
    C e^{-(1-p)x_1+(1-p)y_1}\,p e^{-(1+p)(y-y_1)}e^{p(x-x_1)}\,\diff y_1\,\diff x_1
    =\tfrac{C p}{2}(e^{2\base}-1)\,e^{p x-(1+p)y}.
\end{align*}
Adding coefficients gives
\[
    C\Big(\tfrac{1-p}{2}e^{2\base}+\tfrac{1-p}{2}(e^{2\base}-1)+\tfrac{p}{2}e^{2\base}+\tfrac12+\tfrac{p}{2}(e^{2\base}-1)\Big)
    = C e^{2\base},
\]
hence $\sum_R T_R(x,y)= C e^{2\base} e^{p x-(1+p)y}=\pi(x,y)$.

\noindent \textbf{Target $(x,y)\in\calB_1$.} We compute $T_R(x,y)$ in each region $R$:
\begin{align*}
    T_{\calA_1}(x,y)
    &=\int_{2\base}^{\infty} C e^{2\base} e^{p x-(1+p)y_1}\,(1-p)e^{-(1-p)(y_1-y)}\,\diff y_1
    =\tfrac{C(1-p)}{2}\,e^{-2\base}\,e^{p x+(1-p)y},\\[2pt]
    T_{\calB_1}(x,y)
    &=\int_{0}^{y}\!\!\int_{-\infty}^{\,y_1+(x-y)}\!\!+\int_{y}^{2\base}\!\!\int_{-\infty}^{\,x}
    C e^{p x_1-p y_1}\,(1-p)e^{-(2-p)(x-x_1)}e^{(1-p)(y-y_1)}\,\diff x_1\,\diff y_1\\
    &=\tfrac{C(1-p)}{2}\Big(2 e^{p x-p y}-e^{p x-(1+p)y}-e^{p x+(1-p)y-2\base}\Big),\\[2pt]
    T_{\calC}(x,y)
    &=\int_{2\base}^{\infty} C e^{2\base}e^{-(2-p)x_1+(1-p)y}\,p e^{-p(x_1-x)}\,\diff x_1
    =\tfrac{C p}{2}\,e^{-2\base}\,e^{p x+(1-p)y},\\[2pt]
    T_{\calD}(x,y)
    &=\int_{-\infty}^{\,x-y} C e^{p x_1+(1-p)(y-(x-x_1))}\,e^{-(x-x_1)}\,\diff x_1
    =\tfrac{C}{2}\,e^{p x-(1+p)y},\\[2pt]
    T_{\calE_1}(x,y)
    &=\int_{0}^{y}\!\!\int_{-\infty}^{\,x_1}\!\!+\int_{y}^{2\base}\!\!\int_{-\infty}^{\,y}
    C e^{-(1-p)x_1+(1-p)y_1}\,p e^{-(1+p)(y-y_1)}e^{p(x-x_1)}\,\diff y_1\,\diff x_1\\
    &=\tfrac{C p}{2}\Big(2 e^{p x-p y}-e^{p x-(1+p)y}-e^{p x+(1-p)y-2\base}\Big).
\end{align*}
Summing, the $e^{p x-(1+p)y}$ and $e^{p x+(1-p)y-2\base}$ terms cancel, leaving
\(\,C e^{p x-p y}=\pi(x,y)\).

\noindent \textbf{Target $(x,y)\in\calC$.} We compute $T_R(x,y)$ in each region $R$:
\begin{align*}
    T_{\calA_1}(x,y)
    &=\int_{x}^{\infty} C e^{2\base} e^{p x-(1+p)y_1}\,(1-p)e^{-(1-p)(y_1-y)}\,\diff y_1
    =\tfrac{C(1-p)}{2}\,e^{2\base}\,e^{-(2-p)x+(1-p)y},\\[2pt]
    T_{\calB_1}(x,y)
    &=\int_{0}^{2\base}\!\!\int_{-\infty}^{\,y_1}
    C e^{p x_1-p y_1}\,(1-p)e^{-(2-p)(x-x_1)}e^{(1-p)(y-y_1)}\,\diff x_1\,\diff y_1\\
    &=\tfrac{C(1-p)}{2}(e^{2\base}-1)\,e^{-(2-p)x+(1-p)y},\\[2pt]
    T_{\calC}(x,y)
    &=\int_{x}^{\infty} C e^{2\base}e^{-(2-p)x_1+(1-p)y}\,p e^{-p(x_1-x)}\,\diff x_1
    =\tfrac{C p}{2}\,e^{2\base}\,e^{-(2-p)x+(1-p)y},\\[2pt]
    T_{\calD}(x,y)
    &=\int_{-\infty}^{\,0} C e^{p x_1+(1-p)(y-(x-x_1))}\,e^{-(x-x_1)}\,\diff x_1
    =\tfrac{C}{2}\,e^{-(2-p)x+(1-p)y},\\[2pt]
    T_{\calE_1}(x,y)
    &=\int_{0}^{2\base}\!\!\int_{-\infty}^{\,x_1+(y-x)}
    C e^{-(1-p)x_1+(1-p)y_1}\,p e^{-(1+p)(y-y_1)}e^{p(x-x_1)}\,\diff y_1\,\diff x_1\\
    &=\tfrac{C p}{2}(e^{2\base}-1)\,e^{-(2-p)x+(1-p)y}.
\end{align*}
Adding coefficients yields \(C e^{2\base}\), hence \(\sum_R T_R(x,y)=C e^{2\base}e^{-(2-p)x+(1-p)y}=\pi(x,y)\).

\noindent  \textbf{Target $(x,y)\in\calD$.} We compute $T_R(x,y)$ in each region $R$:
\begin{align*}
    T_{\calA_1}(x,y)
    &=\int_{2\base}^{\infty} C e^{2\base} e^{p x-(1+p)y_1}\,(1-p)e^{-(1-p)(y_1-y)}\,\diff y_1
    =\tfrac{C(1-p)}{2}\,e^{-2\base}\,e^{p x+(1-p)y},\\[2pt]
    T_{\calB_1}(x,y)
    &=\int_{0}^{2\base}\!\!\int_{-\infty}^{\,x}
    C e^{p x_1-p y_1}\,(1-p)e^{-(2-p)(x-x_1)}e^{(1-p)(y-y_1)}\,\diff x_1\,\diff y_1\\
    &=\tfrac{C(1-p)}{2}(1-e^{-2\base})\,e^{p x+(1-p)y},\\[2pt]
    T_{\calC}(x,y)
    &=\int_{2\base}^{\infty} C e^{2\base}e^{-(2-p)x_1+(1-p)y}\,p e^{-p(x_1-x)}\,\diff x_1
    =\tfrac{C p}{2}\,e^{-2\base}\,e^{p x+(1-p)y},\\[2pt]
    T_{\calD}(x,y)
    &=\int_{-\infty}^{\,x} C e^{p x_1+(1-p)(y-(x-x_1))}\,e^{-(x-x_1)}\,\diff x_1
    =\tfrac{C}{2}\,e^{p x+(1-p)y},\\[2pt]
    T_{\calE_1}(x,y)
    &=\int_{0}^{2\base}\!\!\int_{-\infty}^{\,y}
    C e^{-(1-p)x_1+(1-p)y_1}\,p e^{-(1+p)(y-y_1)}e^{p(x-x_1)}\,\diff y_1\,\diff x_1
    =\tfrac{C p}{2}(1-e^{-2\base})\,e^{p x+(1-p)y}.
\end{align*}
The coefficients add to \(C\), so \(\sum_R T_R(x,y)=C e^{p x+(1-p)y}=\pi(x,y)\).

\bigskip

\noindent \textbf{Target $(x,y)\in\calE_1$.} We compute $T_R(x,y)$ in each region $R$:
\begin{align*}
    T_{\calA_1}(x,y)
    &=\int_{2\base}^{\infty} C e^{2\base} e^{p x-(1+p)y_1}\,(1-p)e^{-(1-p)(y_1-y)}\,\diff y_1
    =\tfrac{C(1-p)}{2}\,e^{-2\base}\,e^{p x+(1-p)y},\\[2pt]
    T_{\calB_1}(x,y)
    &=\int_{0}^{x}\!\!\int_{-\infty}^{\,y_1}\!\!+\int_{x}^{2\base}\!\!\int_{-\infty}^{\,x}
    C e^{p x_1-p y_1}\,(1-p)e^{-(2-p)(x-x_1)}e^{(1-p)(y-y_1)}\,\diff x_1\,\diff y_1\\
    &=\tfrac{C(1-p)}{2}\Big(2 e^{-(1-p)x+(1-p)y}-e^{-(2-p)x+(1-p)y}-e^{p x+(1-p)y-2\base}\Big),\\[2pt]
    T_{\calC}(x,y)
    &=\int_{2\base}^{\infty} C e^{2\base}e^{-(2-p)x_1+(1-p)y}\,p e^{-p(x_1-x)}\,\diff x_1
    =\tfrac{C p}{2}\,e^{-2\base}\,e^{p x+(1-p)y},\\[2pt]
    T_{\calD}(x,y)
    &=\int_{-\infty}^{\,0} C e^{p x_1+(1-p)(y-(x-x_1))}\,e^{-(x-x_1)}\,\diff x_1
    =\tfrac{C}{2}\,e^{-(2-p)x+(1-p)y},\\[2pt]
    T_{\calE_1}(x,y)
    &=\int_{0}^{x}\!\!\int_{-\infty}^{\,x_1-(x-y)}\!\!+\int_{x}^{2\base}\!\!\int_{-\infty}^{\,y}
    C e^{-(1-p)x_1+(1-p)y_1}\,p e^{-(1+p)(y-y_1)}e^{p(x-x_1)}\,\diff y_1\,\diff x_1\\
    &=\tfrac{C p}{2}\Big(2 e^{-(1-p)x+(1-p)y}-e^{-(2-p)x+(1-p)y}-e^{p x+(1-p)y-2\base}\Big).
\end{align*}
The $e^{-(2-p)x+(1-p)y}$ and $e^{p x+(1-p)y-2\base}$ terms cancel, so
\(\sum_R T_R(x,y)= C e^{-(1-p)x+(1-p)y}=\pi(x,y)\).

\noindent Finally, we verify normalization. Indeed,
\begin{align*}
    1 &=\iint_{\reals^2}\pi(x,y)\,\diff x\,\diff y
    =\sum_{R\in\{\calA_1,\calB_1,\calC,\calD,\calE_1\}}
    \iint_{(x,y)\in R}\pi(x,y)\,\diff x\,\diff y\\
    &=
    \int_{2\base}^{\infty}\!\!\int_{-\infty}^{\,y}
    C e^{2\base}e^{p x-(1+p)y}\,\diff x\,\diff y
    +\int_{0}^{2\base}\!\!\int_{-\infty}^{\,y}
    C e^{p x-p y}\,\diff x\,\diff y
    +\int_{2\base}^{\infty}\!\!\int_{-\infty}^{\,x}
    C e^{2\base}e^{-(2-p)x+(1-p)y}\,\diff y\,\diff x \\
    & ~~~~~~
    +\int_{-\infty}^{0}\!\!\int_{-\infty}^{0}
    C e^{p x+(1-p)y}\,\diff y\,\diff x
    +\int_{0}^{2\base}\!\!\int_{-\infty}^{\,x}
    C e^{-(1-p)x+(1-p)y}\,\diff y\,\diff x \\
    &= C\Big[\tfrac{1}{p}+\tfrac{2\base}{p}+\tfrac{1}{1-p}+\tfrac{1}{p(1-p)}+\tfrac{2\base}{1-p}\Big]
    = \frac{C}{p(1-p)}\,2(1+\base)
    =1,
\end{align*}
for our choice \(C=\dfrac{p(1-p)}{2(1+\base)}\). This concludes the proof.

\subsection{Proof of~\texorpdfstring{\Cref{rmk: bounds-for-dual-service range}}{}} \label{apx-remark}

\begin{enumerate}
    \item[(i)] The service range vector $\servicevec_p$ has $p$ fraction of entries equal to $\base + \extra$, and the rest equal to $\base$. Therefore, $\servicevec_p$ majorizes $\servicevec'$, where $\servicevec' = (\base + p\extra)\mathbf{1}$ and $\mathbf{1}$ is the all-ones vector in $\mathbb{R}^n$.~\Cref{thm:uniformity} establishes that the matching is larger for $\servicevec'$ than $\servicevec$. Equivalently, $\nu_n(\base,\extra,p) \leq \nu_n(\base+ p\extra,0,1)$. Applying~\Cref{thm: extreme-cases}~(ii), we have
    \[ 
        \frac 1 n \, \nu_n(\base,\extra,p) \leq \frac 1n \, \nu_n(\base+ p \extra,0,1) = \frac{\base+ p\extra}{\base + p \extra + 1/2}  + o(1) \,  .
    \]
    \item[(ii)] Split the set $\driverset$ of supply nodes as $\driverset = \driverset^{\mathrm{F}} \cup \driverset^{\mathrm{NF}}$, according to their flexibility. Let $q \in (p,1]$ be fixed, and construct the set $\riderset^{\mathrm{F}}$ by assigning each demand node in $\riderset$ to $\riderset^{\mathrm{F}}$ independently with probability $q$. Let $\riderset^\mathrm{NF} \triangleq \riderset \setminus \riderset^{\mathrm{F}}$.

    Let $\widetilde G^{\mathrm{F}}$ (resp. $\widetilde G^{\mathrm{NF}}$) denote the induced subgraph of $G$ on the vertex set $\driverset^{\mathrm{F}} \cup \riderset^{\mathrm{F}}$, (resp. $\driverset^{\mathrm{NF}} \cup \riderset^{\mathrm{NF}}$). Further, let $G^{\mathrm{F}}$ (resp. $G^{\mathrm{NF}}$) denote the graph obtained by adding isolated nodes to $\widetilde G^{\mathrm{F}}$ (resp. $\widetilde G^{\mathrm{NF}}$), uniformly at random to the side with fewer nodes so that $G^{\mathrm{F}}$ (resp. $G^{\mathrm{NF}}$) is balanced, i.e. it has the same number $n$ of demand and supply nodes. It follows that $G^\mathrm{F}\sim \mathbb{G}(q n,\, \servicevec^{\mathrm{F}} )$ and $G^\mathrm{NF}\sim \mathbb{G}((1-q) n, \,\servicevec^{\mathrm{NF}} )$, where 
    \[ 
        \servicevec_i^{\mathrm{F}} = 
        \begin{cases} 
            \base+\extra, & \mbox{with probability  } \tfrac p q  \\
            0, &\mbox{with probability } 1 - \tfrac p q 
        \end{cases},
        ~~~~~~
        \servicevec_i^{\mathrm{NF}} = 
        \begin{cases} 
            \base, & \mbox{with probability  } \tfrac {1-q} {1-p} \\
            0, &\mbox{with probability } 1 - \tfrac {1-q} {1-p}.
        \end{cases}
    \]
    These service ranges satisfy the conditions of~\Cref{thm: extreme-cases}~(i). Therefore, the fraction of matched nodes in a maximum matching of $G^{\mathrm{F}}$ and $G^{\mathrm{NF}}$ are respectively
    \[ 
        \frac{ e^{2\base\frac{1-q}{1-p}} - e^{ 2\base} }{e^{2\base\frac{1-q}{1-p}} \!-\! \frac{1-p}{1-q}\cdot  e^{2\base} }  (1-q) + o(1) ~~~~ \mbox{ and }~~~~  \frac{ e^{2(\base+\extra)\frac{p}{q}} - e^{ 2(\base+\extra)} }{ e^{2(\base+\extra)\frac{p}{q}} \!-\! \frac{q}{p}\cdot e^{ 2(\base+\extra)} } \cdot q   + o(1) \, .
    \]
    Finally, the union of any two matchings in $G^{\mathrm{F}}$ and $G^{\mathrm{NF}}$ is a valid matching in $G$, and has size equal to the sum of the individual matchings. Since this argument holds for all $q \in (p,1]$, we obtain~\eqref{eq: lower-bound-with-q}. This concludes the proof. \qedhere
\end{enumerate}

\end{document}